\documentclass[11pt]{article}
\usepackage[margin=1in]{geometry}
\usepackage[utf8]{inputenc}
\usepackage[english]{babel}
\usepackage{amsmath}
\usepackage[all=normal,paragraphs=tight,bibliography=tight]{savetrees}
\usepackage[disable]{todonotes}
\usepackage{amsthm}
\usepackage{amssymb}
\usepackage{microtype}
\usepackage{xspace}
\usepackage{comment}
\usepackage{letltxmacro}
\usepackage{comment}
\usepackage{tikz}
\usepackage{url}
\usepackage[colorlinks=true,citecolor=red,linkcolor=blue]{hyperref}
\usepackage{import}

\usepackage{graphicx}
\usepackage{caption}
\usepackage{subcaption}
\usepackage{multirow}
\usepackage{enumitem}
\usepackage{mathtools}
\usepackage[absolute]{textpos}

\newtheorem{theorem}{Theorem}[section]

\newtheorem{lemma}[theorem]{Lemma}
\newtheorem{claim}[theorem]{Claim}
\newtheorem{corollary}[theorem]{Corollary}

\theoremstyle{definition}
\newtheorem{definition}[theorem]{Definition}

\newcommand{\N}{\mathbb{N}}
\newcommand{\Oh}{\ensuremath{\mathcal{O}}}

\renewcommand{\leq}{\leqslant}
\renewcommand{\geq}{\geqslant}

\newcommand{\famgenus}{\mathcal{F}_{\mathrm{genus}}}

\newcommand{\Vtw}{V_{\mathrm{tw}}}
\newcommand{\Vgenus}{V_{\mathrm{genus}}}

\newcommand{\ctime}{c_{\mathrm{time}}}
\newcommand{\cgenus}{c_{\mathrm{genus}}}

\newcommand{\funtw}{f_{\mathrm{tw}}}
\newcommand{\fungenus}{f_{\mathrm{genus}}}
\newcommand{\funchain}{f_{\mathrm{chain}}}
\newcommand{\funtime}{f_{\mathrm{time}}}
\newcommand{\funcut}{f_{\mathrm{cut}}}
\newcommand{\funstep}{f_{\mathrm{step}}}

\newcommand{\tw}{\mathrm{tw}}
\newcommand{\embed}{\Pi}
\newcommand{\nembed}{\mathcal{E}}
\newcommand{\nembedA}{\nembed}
\newcommand{\nembedB}{\widetilde{\nembed}}
\newcommand{\nembedC}{\widehat{\nembed}}
\newcommand{\nembedT}{\dddot{\nembed}}

\newcommand{\eulerg}{\mathtt{eg}}
\newcommand{\napices}{\mathtt{ap}}
\newcommand{\Apices}{\mathbf{A}}
\newcommand{\ApicesA}{\Apices}
\newcommand{\ApicesB}{\widetilde{\Apices}}
\newcommand{\ApicesC}{\widehat{\Apices}}
\newcommand{\WB}{\widetilde{W}}
\newcommand{\linkage}{\mathcal{L}}
\newcommand{\VortexSmb}{\circ}
\newcommand{\DongleSmb}{\triangle}
\newcommand{\Vortex}[2]{{#1}_{#2}^{\VortexSmb}}
\newcommand{\Dongle}[2]{{#1}_{#2}^{\DongleSmb}}
\newcommand{\VortexBag}[2]{\mathbf{B}_{#1,#2}^{\VortexSmb}}
\newcommand{\Main}[1]{{#1}_0}
\newcommand{\MainPlus}[1]{{#1}_{\oplus}}
\newcommand{\nvortices}{k^{\VortexSmb}}
\newcommand{\ndongles}{k^{\DongleSmb}}
\newcommand{\vortexlength}[1]{n^{\VortexSmb}_{#1}}
\newcommand{\vortexadhwidth}{k^{\VortexSmb}_{\mathrm{adh}}}
\newcommand{\vortexbagsize}{k^{\VortexSmb}_{\mathrm{bag}}}
\newcommand{\donglesize}{k^{\DongleSmb}_{\mathrm{size}}}
\newcommand{\SocietyVtx}[2]{v^{\VortexSmb}_{#1,#2}}
\newcommand{\DongleDisc}[1]{D^{\DongleSmb}_{#1}}
\newcommand{\VortexDisc}[1]{D^{\VortexSmb}_{#1}}
\newcommand{\PlusDongleVtx}[1]{\mathbf{v}^{\DongleSmb}_{#1}}
\newcommand{\PlusVortexVtx}[1]{\mathbf{v}^{\VortexSmb}_{#1}}

\newcommand{\PathProj}{\overrightarrow{\pi}}
\newcommand{\SubProj}{\pi}
\newcommand{\Predongles}{\mathcal{D}^{\mathrm{pre}\triangle}}
\newcommand{\PredonglesB}{\widetilde{\mathcal{D}}^{\mathrm{pre}\triangle}}
\newcommand{\preimage}{\overleftarrow{\mathbb{P}}}

\newcommand{\GraphA}{G}
\newcommand{\GraphB}{\widetilde{G}}
\newcommand{\GraphC}{\widehat{G}}
\newcommand{\GraphH}{\widetilde{H}}

\newcommand{\VortexA}[1]{\Vortex{\GraphA}{#1}}
\newcommand{\DongleA}[1]{\Dongle{\GraphA}{#1}}

\newcommand{\MainA}{\Main{\GraphA}}
\newcommand{\MainPlusA}{\MainPlus{\GraphA}}
\newcommand{\VortexB}[1]{\Vortex{\GraphB}{#1}}
\newcommand{\DongleB}[1]{\Dongle{\GraphB}{#1}}

\newcommand{\MainB}{\Main{\GraphB}}
\newcommand{\MainPlusB}{\MainPlus{\GraphB}}
\newcommand{\VortexC}[1]{\Vortex{\GraphC}{#1}}

\newcommand{\MainC}{\Main{\GraphC}}
\newcommand{\MainPlusC}{\MainPlus{\GraphC}}

\newcommand{\DongleH}[1]{\Dongle{\GraphH}{#1}}

\newcommand{\MainH}{\Main{\GraphH}}
\newcommand{\MainPlusH}{\MainPlus{\GraphH}}

\newcommand{\GraphT}{\dddot{\GraphA}}
\newcommand{\VortexT}[1]{\Vortex{\GraphT}{#1}}
\newcommand{\DongleT}[1]{\Dongle{\GraphT}{#1}}

\newcommand{\MainT}{\Main{\GraphT}}

\newcommand{\DongleComps}{\mathcal{D}}

\newcommand{\rball}[2]{\mathbf{B}_{#1}(#2)}
\newcommand{\rballX}[3]{\mathbf{B}^{#1}_{#2}(#3)}
\newcommand{\rdisc}[3]{\mathbf{D}^{#1}_{#2}(#3)}
\newcommand{\rcycle}[3]{\mathbf{C}^{#1}_{#2}(#3)}
\newcommand{\rintcycle}[4]{\mathbf{iC}^{#1}_{#2,#3}(#4)}
\newcommand{\attach}[2]{\Gamma^{#1}(#2)}
\newcommand{\attall}[1]{\Gamma^{#1}}

\newcommand{\cover}[1]{\Lambda^{#1}}
\newcommand{\projection}[2]{\pi^{#1}(#2)}
\newcommand{\rmax}{r_\mathrm{max}}

\newcommand{\bigC}{c}

\newcommand{\WallHitSize}{k^{\mathrm{wall}}}
\newcommand{\WallHitRadius}{r^{\mathrm{wall}}}
\newcommand{\AttCoverSize}{k^{\mathrm{att}}}
\newcommand{\AttCoverRadius}{r^{\mathrm{att}}}
\newcommand{\ProxCoverSize}{k^{\mathrm{prox}}}
\newcommand{\ProxCoverRadius}{r^{\mathrm{prox}}}

\newcommand{\WallHitC}{C^{\mathrm{wall}}}

\newcommand{\OptNembeds}{{\mathbb{E}}}
\newcommand{\VortexVtcs}{V^\circ}

\def\cqedsymbol{\ifmmode$\lrcorner$\else{\unskip\nobreak\hfil
\penalty50\hskip1em\null\nobreak\hfil$\lrcorner$
\parfillskip=0pt\finalhyphendemerits=0\endgraf}\fi} 

\newcommand{\cqed}{\renewcommand{\qed}{\cqedsymbol}}

\newcommand{\executeiffilenewer}[3]{%
\ifnum\pdfstrcmp{\pdffilemoddate{#1}}%
{\pdffilemoddate{#2}}>0%
{\immediate\write18{#3}}\fi%
} 

\newcommand{\embedsection}[1]{\subsection{#1}}

\title{Highly unbreakable graph with a fixed excluded minor\\are almost rigid}

\author{
  Daniel Lokshtanov\thanks{
    University of California, Santa Barbara, USA, \texttt{daniello@ucsb.edu}.
    Supported by BSF award 2018302 and NSF award CCF-2008838.
  }
  \and
  Marcin Pilipczuk\thanks{
    Institute of Informatics, University of Warsaw, Poland, \texttt{marcin.pilipczuk@mimuw.edu.pl}.
 This work is 
a part of project CUTACOMBS that has received funding from the European Research Council (ERC) 
under the European Union's Horizon 2020 research and innovation programme (grant agreement No.~714704).
  }
  \and
  Micha\l{} Pilipczuk\thanks{
    Institute of Informatics, University of Warsaw, Poland, \texttt{michal.pilipczuk@mimuw.edu.pl}.
 This work is 
a part of projects TOTAL and BOBR that have received funding from the European Research Council (ERC) 
under the European Union's Horizon 2020 research and innovation programme (grant agreements No.~677651 and~948057, respectively).
  }
  \and 
  Saket Saurabh\thanks{
    Institute of Mathematical Sciences, India, \texttt{saket@imsc.res.in}, and
    Department of Informatics, University of Bergen, Norway, \texttt{Saket.Saurabh@ii.uib.no}. Supported by the European Research Council (ERC) under the European Union's Horizon 2020 research and innovation programme (grant agreement No. 819416), and Swarnajayanti Fellowship (No. DST/SJF/MSA01/2017-18).
  }
}

\date{}

\usepackage{centernot}

\begin{document}

\begin{titlepage}
\def\thepage{}
\thispagestyle{empty}
\maketitle

\begin{textblock}{20}(0, 12.7)
\includegraphics[width=40px]{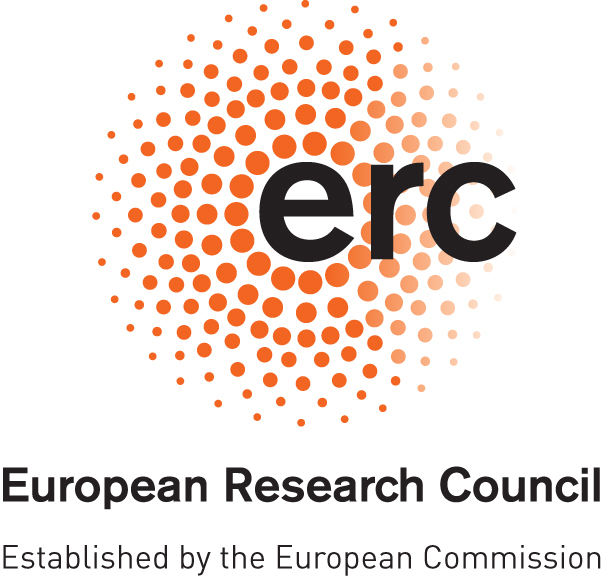}%
\end{textblock}
\begin{textblock}{20}(0, 13.5)
\includegraphics[width=40px]{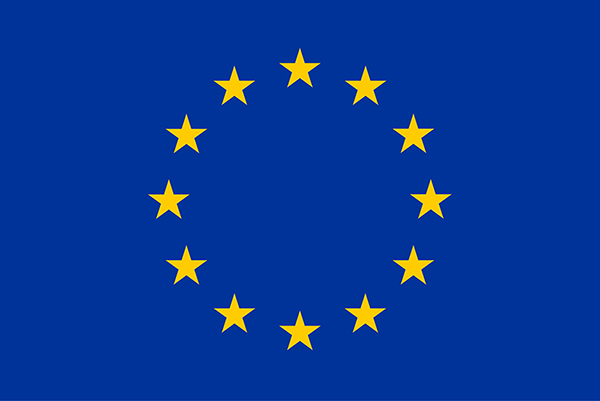}%
\end{textblock}

\begin{abstract}
A set $X \subseteq V(G)$ in a graph $G$ is {\em{$(q,k)$-unbreakable}} if every separation $(A,B)$ of order at most $k$ in $G$ satisfies $|A \cap X| \leq q$ or $|B \cap X| \leq q$.
In this paper, we prove the following result: If a graph $G$ excludes a fixed complete graph $K_h$ as a minor and satisfies certain unbreakability guarantees, then $G$ is almost rigid in the folloring sense: the vertices of $G$ can be partitioned in an isomorphism-invariant way into a part inducing a graph of bounded
treewidth and a part that admits a small isomorphism-invariant family of labelings. This result is the key ingredient in the fixed-parameter algorithm for \textsc{Graph Isomorphism}
parameterized by the Hadwiger number of the graph, which is presented in a companion paper.

\end{abstract}
\newpage
\tableofcontents
\end{titlepage}

\section{Introduction}
The \textsc{Graph Isomorphism} problem is arguably the most widely known problem whose membership in $\mathsf{P}$ is unknown, but which is not believed to be $\mathsf{NP}$-hard. 
After decades of research, a quasi-polynomial time algorithm was proposed by Babai in 2015~\cite{Babai16}.

We focus on the \textsc{Graph Isomorphism} problem restricted to graphs with a fixed
excluded minor. Recall that a graph $H$ is a minor of a graph $G$
if $H$ can be obtained from a subgraph of $G$ by a series of edge contractions.
Ponomarenko~\cite{ponomarenko} showed that when restricted to $H$-minor-free graphs,
the \textsc{Graph Isomorphism} problem can be solved in time $\Oh(n^{p_H})$ where $n$ is the size
of the input and $p_H$ is a constant depending on $H$ only. 

This work is a part of a two-paper series that proves that \textsc{Graph Isomorphism}
is fixed-parameter tractable when parameterized by the excluded minor $H$; that is,
it can be solved on $n$-vertex graphs in time bounded by $p_H \cdot n^c$ where $p_H$ depends on $H$
only and the constant $c$ is a universal constant (independent of $H$). 
The main paper~\cite{main-paper} provides an algorithmic framework for the \textsc{Graph Isomorphism}
problem relying on the notion of \emph{unbreakability} and decomposition into highly 
unbreakable parts. 
This paper, subordinate to the main one, provides the crucial ingredient coming from the graph
minors theory.
We refer to the introduction of~\cite{main-paper} for an extended discussion of the background
of our main result.

To state the main result of this paper, we need a few definitions.
For two graphs $G_1,G_2$, an \emph{isomorphism} is a bijection $\pi : V(G_1) \to V(G_2)$
such that $uv \in E(G_1)$ if and only if $\pi(u)\pi(v) \in E(G_2)$. 

Let $G$ be an undirected simple graph. 
A pair $(A,B)$ is a \emph{separation} if $A,B \subseteq V(G)$, $A \cup B = V(G)$, and there
is no edge of $G$ with one endpoint in $A \setminus B$ and the second endpoint in $B \setminus A$.
The \emph{order} of the separation $(A,B)$ is $|A \cap B|$. 
For integers $q,k \geq 0$, a set $X \subseteq V(G)$ is \emph{$(q,k)$-unbreakable}
if for every separation $(A,B)$ of order at most $k$, we have $|A \cap X| \leq q$
or $|B \cap X| \leq q$. In other words, separations of order at most $k$ are not able
to split $X$ into two parts larger than $q$.
A graph $G$ is \emph{$(q,k)$-unbreakable} if $V(G)$ is $(q,k)$-unbreakable.

In this work we consider $H$-minor-free graphs with very high unbreakability properties. 
As we will show, such graphs are close to being embeddable in a fixed surface:
they are so-called nearly-embeddable graphs. 
Furthermore, we show that if one properly defines a ``quality measure'' of a near-embedding,
the space of optimal (under this measure) near-embeddings of such graphs is very constrained:
all near-embeddings are essentially the same and only differ in some local details.
Let us proceed to formal definitions. 

We need to make assumptions stronger than simply $(q,k)$-unbreakability for some fixed $q,k$. In essence, we require a whole sequence of unbreakability assertions, as explained formally in the following definition.

\begin{definition}[unbreakability chain]
Let $G$ be a graph and $f\colon \N\to \N$ be a nondecreasing function with $f(x) > x$ for every $x \in \N$.
An {\em{unbreakability chain}} of $G$ with \emph{length} $\zeta$ and \emph{step} $f$ is any sequence $((q_i,k_i))_{i=0}^\zeta$ of pairs of integers satisfying the following:
\begin{itemize}[nosep]
 \item $G$ is $(q_i,k_i)$-unbreakable, for each $0\leq i\leq \zeta$; and
 \item $k_i=f(q_{i-1} + k_{i-1})$, for each $1\leq i\leq \zeta$.
\end{itemize}
We shall often say that such a chain {\em{starts}} at $k_0$.
\end{definition}
We are now ready to state the main result, whose proof spans the rest of the paper.
\begin{theorem}\label{thm:rigid}
There exist constants $\ctime$, $\cgenus$, computable functions $\funtw$, $\funstep$, $\funchain$, $\funtime$, $\funcut$, $\fungenus$ (with $\funstep 
\colon \N \to \N$ being nondecreasing
    and satisfying $\funstep(x) > x$ for every $x \in \N$), and an algorithm that, given
a graph $H$ and an $H$-minor-free graph $G$
together with $((q_i,k_i))_{i=0}^\zeta$ being an unbreakability chain of $G$ of length $\zeta \coloneqq \funchain(H)$ with step $\funstep$ starting at $k_0 \coloneqq \funcut(H)$,
works in time bounded by $\funtime(H,\sum_{i=0}^\zeta q_i) \cdot |V(G)|^{\ctime}$ and computes a partition
$V(G) = \Vtw \uplus \Vgenus$ and a family $\famgenus$ of bijections $\Vgenus \mapsto [|\Vgenus|]$ 
such that 
\begin{enumerate}
\item the partition $V(G) = \Vtw \uplus \Vgenus$ is isomorphism-invariant;
\item $G[\Vtw]$ has treewidth bounded by $\funtw(H,\sum_{i=0}^\zeta q_i)$;
\item $|\famgenus| \leq \fungenus(H,\sum_{i=0}^\zeta q_i) \cdot |V(G)|^{\cgenus}$; and
\item $\famgenus$ is isomorphism-invariant.
\end{enumerate}
\end{theorem}
Here, \emph{isomorphism-invariant} output means that for every $H$ and for
every two $H$-minor-free graphs $G_1$ and $G_2$ and every isomorphism $\pi$ between them,
the output of the algorithm invoked on $H$, $G_1$ (and some unbreakability chain)
and the output of the algorithm invoked on $H$, $G_2$ (and the same unbreakability chain)
are isomorphic and, furthermore, the isomorphism $\pi$ extends to the isomorphism
of the outputs. 

Intuitively, the $\Vgenus$ part corresponds to the main skeleton of the whole graph,
which is embedded in the same way in every optimal near-embedding. On the other hand, $\Vtw$ consists
of parts of bounded size that ``dangle'' from the skeleton; an optimal near-embedding may take some local decisions
about how to exactly represent them.

\medskip

We provide an introduction to the proof
in Section~\ref{sec:overview}. Here, we mainly sketch the proof of an analog of Theorem~\ref{thm:rigid}
for classes of graphs embeddable in a fixed surface. This proof already highlights main ideas
behind the proof of the general case, while being technically less challenging.
After preliminaries in Section~\ref{sec:prelims},
 we continue with the full proof of Theorem~\ref{thm:rigid} in subsequent sections.

\section{Overview of the proof}\label{sec:overview}
In this section we provide an intuitive overview of the proof of Theorem~\ref{thm:rigid}.

To simplify the exposition, we first focus on proving Theorem~\ref{thm:rigid} under a stronger assumption that $G$ actually has bounded genus. This case already allows us to show most of the key conceptual steps used in the argument for the general case of unbreakable $H$-minor-free graphs.
Then, we briefly discuss traps, issues, and caveats that arise when working in the general setting with near-embeddings rather than embeddings.

\subsection{Graphs of bounded genus}

In this section we sketch how to prove the following statement, which is weaker variant of Theorem~\ref{thm:rigid} tailored to the bounded genus case.

\begin{theorem}\label{thm:genus}
There exist computable functions $\funcut,\funtw,\funtime$ and an algorithm $\mathcal{A}$ with the following specification.
The algorithm $\mathcal{A}$ is given as input a graph $G$ and integers $g$ and $q$
with a promise that 
$G$ is of Euler genus at most $g$ and is $(q,k)$-unbreakable for $k = \funcut(g)$. 
Then $\mathcal{A}$
runs in time bounded by $\funtime(g,q) \cdot n^{\Oh(1)}$
and computes an isomorphism-invariant partition of $V(G)$ into $\Vtw \uplus \Vgenus$ 
and a nonempty isomorphism-invariant family $\famgenus$ of size $\Oh(|E(G)|)$ of bijections $\Vgenus \mapsto [|\Vgenus|]$
such that $G[\Vtw]$ is of treewidth at most $\funtw(g) \cdot q$.
\end{theorem}

Without loss of generality assume that  $q \geq k \geq 2$. 
If $G$ is of treewidth $\Oh(q)$, then we can return $\Vgenus = \emptyset$.
Hence, we assume that the treewidth of $G$ is much larger than $q$. 

Consider Tutte's decomposition of $G$ into $3$-connected components\footnote{The $3$-connected components of a graph are the torsos of the bags of its Tutte's decomposition. See Section~\ref{sec:3conn} for a discussion.}. Since $G$ is $(q,k)$-unbreakable and has treewidth much larger than $q$, it follows that there is a unique 3-connected component of $G$ that has more than $q$ vertices
and treewidth much larger than $q$, while all the other 3-connected components are of size at most $q$.
The same conclusion holds for the graph $G-X$ for any set $X \subseteq V(G)$ of size at most $k-2$: 
If $(A,B)$ is a separation of order at most $2$ in $G-X$, then $(A\cup X, B \cup X)$ is a separation
of order at most $k$ in $G$ and, hence, either $|A \setminus B| \leq q$ or $|B \setminus A| \leq q$.
For any set $X \subseteq V(G)$ of size at most $k-2$, by $G^X$ we denote the unique
3-connected component of $G-X$ that has more than $q$ vertices.
Note that the treewidth assumption implies that the treewidth of $G^X$ is in fact much larger than $q$.

Our goal is to define $\Vgenus$ so that $G[\Vgenus]$ is almost rigid --- admits only few automorphisms --- and this rigidity will be derived from the rigidity of embedded graphs.
Observe that if $\embed$ is an embedding of a connected graph $G$ in some surface, then an automorphism 
of $(G,\embed)$ is fixed by fixing only the image of one edge with distinguished endpoint
and an indication which side of the edge is ``left'' and which is ``right''. This gives at most $4|E(G)|$
automorphisms.

Unfortunately, a graph can have many different embeddings in a surface of Euler genus $g$, even assuming it is $3$-connected. However,
the number of different embeddings drops if the embeddings have large face-width (the minimum number of vertices on a noncontractible face-vertex noose, i.e., a closed curve without self-intersections):
\begin{theorem}[\cite{SeymourT96}]\label{thm:overview:fw-unique}
There is a function $f_{\mathrm{ue}}(g) \in \Oh(\log g / \log \log g)$
such that if $G$ is a $3$-connected graph and 
$\embed$ is an embedding of $G$ of Euler genus $g$
 and face-width at least $f_{\mathrm{ue}}(g)$,
 then $g$ is equal to the (minimum) Euler genus of $G$ and $\embed$ is the unique embedding of 
 Euler genus $g$.
\end{theorem}

The first idea would be to apply Theorem~\ref{thm:overview:fw-unique} to $G^\emptyset$ --- the unique large $3$-connected component of $G$. That is,
consider an embedding $\embed$ of $G^\emptyset$ of minimum possible Euler genus. If the face-width of this embedding is at least
$f_{\mathrm{ue}}(g)$, then we may simply set $\Vgenus = V(G^\emptyset)$. Then $\embed$ is the unique embedding of $G[\Vgenus]$ of minimum Euler genus, which gives rise to an isomorphism-invariant family of at most $4|E(G)|$ automorphisms of $G$, from which a suitable family $\famgenus$ can be easily derived. Note that by $(q,k)$-unbreakability, every connected component of $G[\Vtw] = G-\Vgenus$ has at most $q$ vertices, hence in particular $G[\Vtw]$ has treewidth at most $q$.

However, it may happen that $\embed$ has face-width smaller than $f_{\mathrm{ue}}(g)$ and Theorem~\ref{thm:overview:fw-unique} cannot be applied. But then there is a set $X_1 \subseteq V(G^\emptyset)$ of size at most $f_{\mathrm{ue}}(g)$
such that $G^\emptyset-X_1$ has strictly smaller Euler genus than $G^\emptyset$. 
This implies that $G^{X_1}$ has strictly smaller Euler genus than $G^\emptyset$. 

We can iterate this process: if we take any minimum Euler genus embedding $\embed_1$ of $G^{X_1}$ and it turns out that $\embed_1$ has small face-width,
there is a set $X_2$ of small cardinality such that $G^{X_1 \cup X_2}$ has strictly smaller Euler genus. 
It will be useful later to allow subsequent steps in this process to ask for larger and larger
face-width (and thus allowing larger cut sets $X_2,X_3,\ldots$). Note that the number of steps is bounded by the Euler genus of $G$.

More formally, let us define a function $\kappa \colon \{0,1,\ldots,g,g+1\} \to \mathbb{N}$ by setting $\kappa(0) = 0$ and $\kappa(\gamma+1) = p_\kappa(\kappa(\gamma))$
for a polynomial $p_\kappa(x) = x + f_{\mathrm{ue}}(g)+ \bigC\cdot (x+g+1)^4$, where $\bigC$ is a sufficiently large constant.
We set $k = \funcut(g) = \kappa(g+1) \coloneqq p_\kappa(\kappa(g)) \in  2^{2^{\Oh(g)}}$.

Let $0 \leq \gamma \leq g$. A set $X \subseteq V(G)$ is a \emph{potential deletion set for $\gamma$} if $|X| \leq \kappa(\gamma)$ and the Euler genus of $G^X$ is at most $g-\gamma$. 
Let $0 \leq \gamma_0 \leq g$ be the maximum integer such that there exists a potential deletion set for $\gamma_0$;
note that $\gamma_0$ exists as $X = \emptyset$ is a potential deletion set for $\gamma = 0$. 
Let $\kappa_0 \leq \kappa(\gamma_0)$ be the minimum size of a potential deletion set for $\gamma_0$. 
A potential deletion set for $\gamma_0$ of size $\kappa_0$ shall be called a \emph{deletion set}.

Let $\mathcal{X}$ be the family of all deletion sets. 
Let $X_0 = \bigcap \mathcal{X}$ be the set of those vertices of $G$ that are contained in every deletion set; clearly $|X_0| \leq \kappa_0$.
We define also $Z \coloneqq \bigcup \mathcal{X}$ to be the set of all vertices contained in any deletion set. 
We remark that known algorithms for \textsc{Genus Vertex Deletion} (e.g.,~\cite{KociumakaP19}) can be modified to compute
$\gamma_0$, $\kappa_0$, $X_0$, $Z$, and an arbitrary element of $\mathcal{X}$ in FPT time when parameterized by $g$. 

In this overview we assume $\gamma_0 < g$; the case $\gamma_0 = g$ (i.e., we reach a planar graph in the end of the process described above) can be handled very similarly, using the fact that a $3$-connected plane graphs has a unique planar embedding. By the choice of $\gamma_0$, for every $X \in \mathcal{X}$ the graph $G^X$ admits an embedding $\embed^X$ of Euler genus $g-\gamma_0$ and face-width at least $f_{\mathrm{ue}}(g) + \bigC \cdot (\kappa_0 + g+ 1)^4$, i.e., much larger than both $\kappa_0$
and $f_{\mathrm{ue}}(g)$. 
In particular, $\embed^X$ has the minimum possible Euler genus, which is $g-\gamma_0$, and is the unique embedding of $G^X$ of this Euler genus.

Although it can be easily seen that the treewidth of $G-V(G^X)$ is $\Oh(q+\kappa_0)$ and the uniqueness of the embedding of $G^X$ makes its automorphism group simple, we cannot return $\Vgenus = V(G^X)$, because the choice of $X \in \mathcal{X}$ is not isomorphism-invariant. 
The crux of the approach lies in showing that the graph $G^Z$ can be defined similarly as $G^X$, while now $Z$ --- the union of all deletion sets --- is defined in an isomorphism-invariant way. Namely, $G^Z$ is the unique large 3-connected component of $G-Z$, which can be well defined despite the fact that the size of $Z$ may be much larger than~$k$. Then
$G^Z$ has similar embedding properties as $G^X$, while it turns out that
$G-V(G^Z)$ still has low treewidth. Thus we can return $\Vgenus = V(G^Z)$. 
Intuitively, the key to proving the viability of this strategy is to show that two elements $X,X' \in \mathcal{X}$ cannot differ too much from each other, and therefore $Z=\bigcup_{X\in \mathcal{X}} X$ is structurally not much different from every single $X\in \mathcal{X}$.

For $X \in \mathcal{X}$ and $w \in V(G)\setminus X$, we define the \emph{projection} $\projection{X}{w} \subseteq V(G^X)$
as follows. If $w \in V(G^X)$, then $\projection{X}{w} = \{w\}$. 
Otherwise, letting $C_w$ be the component of $G-X-V(G^X)$ containing $w$, we set $\projection{X}{w} = N_{G-X}(C_w)$. 
Note that $|N_{G-X}(C_w)| \in \{0,1,2\}$ and if $|N_{G-X}(C_w)| = 2$, then $G^X$ contains an edge connecting the two elements of $N_{G-X}(C_w)$.

For $X \in \mathcal{X}$, we define the set $\attall{X}$ of \emph{attachment points} as follows.
For every $v \in X \setminus X_0$, we define the set $\attach{X}{v}$ of attachment points of $v$
as the union of $\projection{X}{u}$ over all $u \in N_G(v) \setminus X$. 
Finally, we define $\attall{X} = \bigcup_{v \in X \setminus X_0} \attach{X}{v}$.

The first step in our analysis of $\mathcal{X}$ is 
to observe that the attachment points are local in the following sense.
\begin{lemma}[informal statement]\label{lem:overview:genus:local1}
For every $X \in \mathcal{X}$, there is a set $\cover{X} \subseteq \attall{X}$ of size $\Oh((g+\kappa_0+1)^2)$ 
such that $\attall{X}$ is covered by the union of radial balls in $\embed^X$ of radius $9$ centered at vertices of $\cover{X}$.
\end{lemma}
The proof of Lemma~\ref{lem:overview:genus:local1} goes roughly as follows: if for some $v \in X$
the set $\attach{X}{v}$ cannot be covered by a small number of radial balls, 
then one can pack many vertex-disjoint radial balls of radius $3$ around the attachment points.
Using the 3-connectivity of $G^X$, one can argue that these radial balls give rise to many $K_5$
minor models intersecting only at $v$. This implies $v \in X_0$, a contradiction.

Using the above, we prove that any two elements of $\mathcal{X}$ need to be somewhat similar, in the sense that their attachment points are close to each other.
\begin{lemma}[informal statement]\label{lem:overview:genus:local2}
Consider any $X,X' \in \mathcal{X}$. Then
for every $w \in X' \setminus X$, the elements of $\projection{X}{w}$ are within radial distance
$\Oh((g+\kappa_0)^2)$ from $\cover{X}$ in $\embed^X$.
\end{lemma}
For the proof of Lemma~\ref{lem:overview:genus:local2},
assume there is $w \in X' \setminus X$ with $\projection{X}{w}$ radially far from $\cover{X}$ (and thus also far from
$\attall{X}$).
Since $\kappa_0$ is still significantly smaller than $k$, the graph $G^{X' \cup X}$ is well-defined and it has a unique embedding $\embed^{X\cup X'}$ with large face-width and of the same Euler genus $g-\gamma_0$. One can now argue that a big part of the graph ``between'' $\cover{X}$ and $\projection{X}{w}$ is embedded in the same way in embeddings $\embed^X$, $\embed^{X'}$, and $\embed^{X\cup X'}$. This gives ground for a replacement argument: if $\embed^X$ is able to embed $\projection{X}{w}$, and the area between $\projection{X}{w}$ and $\cover{X}$ is embedded in the same way in $\embed^X$ and in $\embed^{X'}$, then one can modify the embedding $\embed^{X'}$ using a part of the embedding $\embed^X$ to obtain an embedding of $G^{X' \setminus \{w\}}$. This embedding can be turned into an embedding of $G-(X'\setminus \{w\})$ of Euler genus $g-\gamma_0$,  contradicting the minimality of $X' \in \mathcal{X}$.

With Lemma~\ref{lem:overview:genus:local2} established, we can proceed to the analysis of $Z=\bigcup_{X\in {\mathcal{X}}} X$. 
Lemma~\ref{lem:overview:genus:local2} shows that for any fixed $X\in {\mathcal{X}}$, $Z\setminus X$ is contained in $\Oh((g+\kappa_0)^2)$ radial balls in $\embed^X$, each of radius $\Oh((g+\kappa_0)^2)$. First, since $|X|=\kappa_0$, this gives a bound on the treewidth of $G[Z]$ using the fact that graphs of bounded genus have bounded local treewidth --- the treewidth is bounded linearly in the radial radius. 
Second, it shows that $G^X \setminus Z$, with the embedding inherited from $G^X$, still has large face-width. This allows us to define $G^Z$ as the unique 3-connected component of $G-Z$ with Euler genus $g-\gamma_0$ and large face-width of the embedding.

The face-width lower bound makes the embedding of $G^Z$ unique, imposing a very rigid structure on the automorphism group of $G^Z$. With some technical care, the treewidth bound on $G[Z]$ can be lifted to a treewidth bound on $G-V(G^Z)$. Since the definition of $Z$ is isomorphism-invariant, our definition of $G^Z$ is also isomorphism invariant. So all requirements are satisfied to set $\Vgenus=V(G^Z)$ and $\Vtw=V(G-V(G^Z))$. This finishes the sketch of the proof of Theorem~\ref{thm:genus}.

\subsection{The general case}

We now discuss how the strategy presented above can be lifted to the general case of $H$-minor-free graphs. Here, the main tool will be the Structure Theorem of Robertson and Seymour. We first need to recall some terminology.

A {\em{tangle}} $\mathcal{T}$ of order $p$ in a graph
$G$ is a family of separations of order less than $p$ in $G$
such that (i) for every separation $(A,B)$ of order less than $p$, either $(A,B)$
or $(B,A)$ is in $\mathcal{T}$, (ii) for every three elements $(A_1,B_1),(A_2,B_2),(A_3,B_3) \in \mathcal{T}$, we have $A_1 \cup A_2 \cup A_3 \neq V(G)$. 
If $(A,B) \in \mathcal{T}$, then $A$ is the \emph{small side} and $B$ is the \emph{big side}
of the separation $(A,B)$.
For a vertex subset $X$ and a tangle~$\mathcal{T}$, by $\mathcal{T}-X$ we denote the set of all separations $(A,B)$ of $G-X$ such that $(A\cup X,B\cup X)\in \mathcal{T}$. It is straightforward to check that if $\mathcal{T}$ has order $p>|X|$, then $\mathcal{T}-X$ is a tangle of order $p-|X|$.

Note that the unbreakability assumption on $G$ in Theorem~\ref{thm:rigid}
allows us to define the {\em{unbreakability tangle}} of order $k_\zeta+1$: 
the tangle consists of all separations $(A,B)$ of order at most $k_\zeta$ with $|A| \leq q_\zeta)$. 
 This defines a tangle as long as $|G| > 3q_\zeta)$, and otherwise we can return  $\Vtw=V(G)$ and $\Vgenus=\emptyset$.

In the general case we will work with \emph{near-embeddings} of graphs; in this matter, we mostly follow the the notation from~\cite{DiestelKMW12}.
Given a graph $G$, a \emph{near-embedding} of $G$ consists of:
\begin{enumerate}[nosep]
\item a set $\Apices \subseteq V(G)$, called the \emph{apices};
\item a partition of the graph $G-\Apices$ into edge-disjoint subgraphs
$$G-\Apices = \MainA \cup \bigcup_{i=1}^{\nvortices} \VortexA{i} \cup \bigcup_{i=1}^{\ndongles} \DongleA{i};$$
\item an embedding $\embed$ of $\MainA$ into a (compact and connected) surface $\Sigma$.
\end{enumerate}
The following properties have to be satisfied:
\begin{enumerate}[nosep]
\item Every vertex of $G-\Apices$ that belongs to at least two subgraphs of the partition, belongs to $\MainA$. 
\item For every $1 \leq i \leq \ndongles$, $|V(\DongleA{i}) \cap V(\MainA)| \leq 3$. (The graphs $\DongleA{i}$ are henceforth called \emph{dongles}.)
\item The graphs $(\VortexA{i})_{i=1}^{\nvortices}$ are pairwise vertex-disjoint.
\item For every $1 \leq i \leq \nvortices$, the set $V(\VortexA{i}) \cap V(\MainA)$ can be enumerated as $\{\SocietyVtx{i}{1}, \ldots, \SocietyVtx{i}{\vortexlength{i}}\}$
so that $\VortexA{i}$ admits a path decomposition $(\VortexBag{i}{1}, \ldots, \VortexBag{i}{\vortexlength{i}})$ with $\SocietyVtx{i}{j} \in \VortexBag{i}{j}$ for every $1 \leq j \leq \vortexlength{i}$.
The graph $\VortexA{i}$ is called a \emph{vortex} and the vertices $\SocietyVtx{i}{j}$ are called the \emph{society vertices} of the vortex~$\VortexA{i}$. The number of society vertices, $\vortexlength{i}$, is the {\em{length}} of a vortex.
\item On the surface $\Sigma$ there exist closed discs with disjoint interiors $(\DongleDisc{i})_{i=1}^{\ndongles}$ and $(\VortexDisc{i})_{i=1}^{\nvortices}$
such that $\embed$ embeds $\MainA$ into the closure of $\Sigma \setminus \left(\bigcup_{i=1}^{\ndongles} \DongleDisc{i} \cup \bigcup_{i=1}^{\nvortices} \VortexDisc{i}\right)$ and the following conditions hold:
\begin{itemize}[nosep]
\item For every $1 \leq i \leq \ndongles$, the vertices of $\MainA$ that lie on the boundary of $\DongleDisc{i}$ are exactly the vertices of $V(\DongleA{i}) \cap V(\MainA)$.
\item For every $1 \leq i \leq \nvortices$, the vertices of $\MainA$ that lie on the boundary of $\VortexDisc{i}$ are exactly the vertices of $V(\VortexA{i}) \cap V(\MainA)$. Furthermore,
  these vertices lie around the boundary of $\VortexDisc{i}$ in the order $\SocietyVtx{i}{1}, \ldots, \SocietyVtx{i}{\vortexlength{i}}$.
\end{itemize}
\end{enumerate}

We say that a near-embedding $\nembed$ \emph{captures} a tangle $\mathcal{T}$ if for each separation of $\mathcal{T}-\Apices$, the big side of this separation
is not contained in any dongle $\DongleA{i}$ or in any vortex $\VortexA{i}$.

The following statement is the main structural result of the Graph Minors series and appears
as (3.1) in~\cite{GM16}.
\begin{theorem}\label{thm:overview:GM16}
For every nonplanar graph $H$ there exist integers $k_H,\alpha_H \geq 0$ such that
for every $H$-minor-free graph $G$ and every tangle $\mathcal{T}$ of $G$ of order at least $k_H$, 
$G$ admits a near-embedding capturing $\mathcal{T}$ where the number of apices, the Euler genus of the surface,
the number of vortices, and the maximum adhesion size of a vortex are all bounded by $\alpha_H$.
\end{theorem}

Thus, if we set the starting value $k_0$ to be at least $k_H + \alpha_H+1$, then Theorem~\ref{thm:overview:GM16} asserts that there exists a near-embedding $\nembed$ of $G$ that captures the unbreakability tangle where the number of apices and the maximum adhesion size of a vortex is bounded by $k_0$. This in particular implies that every dongle
and every bag of a vortex has size bounded by $q_0)$. 

To imitate the role of $\kappa$ from the bounded genus case, consider the following 
refinement process.
Recall that $\embed$ is the embedding of $\MainA$ in the near-embedding $\nembed$.
If the face-width of $\embed$ is small (i.e., bounded by a function of $q_0$),
or if two vortices are close in the radial distance in $\embed$ (again, meaning that the radial
distance is bounded as a function of $q_0$), we can move all vertices appearing
on the corresponding curve on the surface
to the apex set, decreasing either the Euler genus or the number of vortices.
If such a curve passes through a vortex or a disc corresponding to a dongle, we move also to the apex set the entire dongle or
two bags of a vortex --- the ones at society vertices where the curve entered and left the disc of a vortex.
In this way we have obtained a new near-embedding $\nembedB$, which uses fewer vortices or a simpler surface, but where the number of apices has grown to at most some function of $q_0$ --- call it $k_1$. 
Assume now the graph is $(q_1,k_1)$-unbreakable and iterate this process, as in the bounded genus case. The number of iterations is bounded by $\Oh(\alpha_H)$, since every iteration either decreases the genus or the number of vortices. So eventually the process stops after some $\iota\leq \Oh(\alpha_H)$ iterations, reaching a near-embedding that has at most $k_\iota$ apices, but has face-width
and the radial distance between vortices larger than any function of $q_\iota$ fixed in advance. 
Here we stop; we will analyze the family of all such near-embeddings, and call them henceforth \emph{optimal}.

Note that the argument above uses the unbreakability assumption only for a bounded --- in terms of $k_H$, $\alpha_H$, and the function $q$ --- initial values $q$. These bounds are form exactly the unbreakability chain of Theorem~\ref{thm:rigid}. 

At this moment, most of the reasoning for bounded genus can be adjusted. 
First, there is a well-defined notion of \emph{universal apices}: these are vertices of $G$ that 
are apices in all optimal-near embeddings. Then we can prove an analog of
Lemma~\ref{lem:overview:genus:local1}, which shows 
that in every optimal near-embedding, the attachment points of non-universal apices appear locally on the surface: up to technical details, they can be covered by a bounded number of bounded-radius balls.
Second, leveraging on that, we prove an analog of Lemma~\ref{lem:overview:genus:local2}: for any two optimal near-embeddings $\nembedA$ and~$\nembedB$, if we additionally
assume that $\nembedB$ has inclusion-wise minimal set of vertices contained in vortices, then every
vertex that is an apex or is in a vortex in $\nembedB$, is also either an apex of $\nembedA$
or is close (in radial distance) to a vortex or an attachment point of a non-universal apex of $\nembedA$. 

Finally, this allows us to argue the following. Define $Z$ to be the set of all vertices of $G$ that can be apices or be contained in vortices in an optimal near-embedding with an inclusion-wise minimal set of vertices
contained in vortices. Then  $G-Z$ contains a unique 3-connected component $H$ of large treewidth.
Further, any optimal near-embedding of $G$ as above projects to a near-embedding of $H$ without vortices
or apices, but into the same surface and with still large face-width. 
This allows us to use Theorem~\ref{thm:overview:fw-unique} to be able to proclaim $V(H)$ to be the $\Vgenus$
part. As before, we are left with arguing that the graph $G-\Vgenus = G[\Vtw]$ has bounded treewidth. Again, this follows from the fact that surface-embedded graphs have locally bounded treewidth, which we use in conjunction with the radial bounds provided by the analog
of Lemma~\ref{lem:overview:genus:local2}. 

Note that in the plan sketched above, we cannot directly apply Theorem~\ref{thm:overview:fw-unique} to $H$, as $H$ is not a completely embeddable graph: the near-embedding inherited from a near-embedding of $G$ has no apices or vortices, but still may contain dongles. To circumvent this issue, we prove that in $3$-connected graphs, the dongles can be chosen in a canonical way so that we may essentially speak about the \emph{unique} near-embedding of $H$. We remark that the canonical choice of dongles is also crucially used in the replacement argument underlying the proof of the analog of Lemma~\ref{lem:overview:genus:local2}.

\section{Preliminaries}\label{sec:prelims}
In most cases, we use standard graph notation, see e.g.~\cite{diestel}. For a graph $G$ we use $V(G)$ and $E(G)$ to denote the set of vertices and edges respectively. We use $n$ to denote $|V(G)|$ and $m$ to denote $|E(G)|$. For an undirected graph $G$ the neighborhood of a vertex $v$ is denoted by $N(v)$ and defined as $N(v) = \{u \colon uv \in E(G)\}$. The {\em closed neighborhood} of v$ $ is $N[v] = N(v) \cup \{v\}$. For a vertex set $S$, $N[S] = \bigcup_{v \in S} N[v]$ while $N(S) = N[S] \setminus S$. For a vertex set $S$ the {\em subgraph of $G$ induced by $S$} is the graph with vertex set $S$ and edge set $\{uv \in E(G) ~:~ \{u,v\} \subseteq S \}$. {\em Deleting} a vertex set $S$ from $G$ results in the graph $G - S$ which is defined as $G[V(G) \setminus S]$. Deleting an edge set $S$ from $G$ results in the graph $G - S $ with vertex set $V(G)$ and edge set $E(G) \setminus S$. Deleting a single vertex or edge from a graph $G$ is the same as deleting the set containing only the vertex or edge from $G$. A {\em subgraph} of $G$ is a graph $H$ which can be obtained by deleting a vertex set and an edge set from $G$. {\em Contracting} an edge $uv$ in a graph $G$ results in the graph $G / uv$ obtained from $G - \{u, v\}$ by adding a new vertex $w$ and adding edges from $w$ to all vertices in $N(\{u, v\})$. Contracting a set $S$ of edges results in the graph $G / S$, obtained from $G$ by contracting all edges in $S$. It is easy to see that the order of contraction does not matter and results in the same graph, and therefore the graph $G / S$ is well defined. If $H  = G / S$ for some edge set $S$ we say that $H$ is a {\em contraction} of $G$. $H$ is a {\em minor} of $G$ if $H$ is a contraction of a subgraph of $G$. 

In this paper all graphs can be multigraphs, that is, we allow parallel edges and loops.

A \emph{tree decomposition} of a graph $G$ is a pair $(T,\beta)$ where $T$
is a tree and $\beta : V(T) \to 2^{V(G)}$ is a function satisfying:
(i) for every $v \in V(G)$, the set $\{t \in V(T)~|~v \in \beta(t)\}$ is a nonempty connected
subgraph of $T$;
(ii) for every $uv \in V(G)$ there exists $t \in V(T)$ with $u,v \in \beta(t)$.
The \emph{width} of a tree decomposition $(T,\beta)$ is 
$\max_{t \in V(T)} |\beta(t)|-1$ and the \emph{treewidth} of a graph $G$
is the minimum possible width of its tree decomposition.
The sets $\beta(t)$ are called \emph{bags} while sets $\beta(t) \cap \beta(s)$ for $st \in E(T)$ are called \emph{adhesions}.
A \emph{path decomposition} is a tree decomposition $(T,\beta)$ where $T$ is a required
to be a path; \emph{pathwidth} is defined analogously.

For a tree decomposition $(T,\beta)$ of a graph $G$,  
the \emph{torso} of the bag $\beta(t)$ at $t$ is the graph created from $G[\beta(t)]$
by turning $\sigma(st)$ into a clique, for every edge $st$ of $T$ incident with $t$.

\embedsection{3-connected components}\label{sec:3conn}

We will need the following classic result about decomposing a graph into 3-connected components.

\begin{theorem}[\cite{CunninghamE80,HopcroftT73b}]\label{thm:3conn}
Given a graph $G$, one can in linear time compute a tree decomposition $(T,\beta)$ of $G$
such that every adhesion of $(T,\beta)$ is of size at most $2$, and every torso of a bag of $(T,\beta)$ is either $3$-connected, a cycle, or of size at most~$2$.
Furthermore, the family of bags of $(T,\beta)$ is isomorphism-invariant (but the tree $T$ may not be). 
\end{theorem}
The torsos of bags of the decomposition of $G$ provided by Theorem~\ref{thm:3conn} that are $3$-connected are called the \emph{3-connected components} of $G$.
Note that the family of 3-connected components of $G$ is isomorphism-invariant.

We will need the following simple observation.
\begin{lemma}\label{lem:tutte-tw}
Let $G$ be a graph, $A \subseteq V(G)$, and assume that for every 3-connected component $C$ of $G$, the treewidth of $C[A]$ is at most $w$.
Then the treewidth of $G[A]$ is at most $\max(w,2)$. 
\end{lemma}
\begin{proof}
Let $(T,\beta)$ be the decomposition of $G$ provided by Theorem~\ref{thm:3conn}. 
For every $t \in V(T)$, let $C_t$ be the torso of $\beta(t)$ and let $(T_t,\beta_t)$ be a tree decomposition of $C_t[A]$ of width at most $\max(w,2)$
(note that the treewidth of a cycle equals $2$). 
Connect the trees $T_t$ into a tree $T'$, obtaining a tree decomposition $(T', \bigcup_{t \in V(T)} \beta_t)$ of $G[A]$ of width at most $\max(w,2)$, as follows:
For every $st \in E(T)$, connect by an edge a node of $T_s$ and a node of $T_t$ that both contain $A \cap \beta(s) \cap \beta(t)$ in their bags;
this is possible as $A \cap \beta(s) \cap \beta(t)$ is a clique of size at most $2$ in both the torso of $\beta(s)$ and the torso of $\beta(t)$.
\end{proof}

\embedsection{Embeddings and surfaces}

We use the standard notions of combinatorial embeddings for connected graphs. The notions of their equivalence,
   the face-vertex graph, one- and two-sided cycles, edge-width, face-width, and Euler genus, are
as in the textbook of Mohar and Thomassen~\cite{MTbook}.
For reviewer's comfort, we recall them briefly here. 

\paragraph{Combinatorial embeddings.}
Let $G$ be a connected (multi)graph.
A \emph{rotation system} is a family $\pi = (\pi_v)_{v \in V(G)}$ where every $\pi_v$
is a cyclic permutation of the edges incident with $v$. 
A \emph{signature mapping} is a function $\lambda : E(G) \to \{-1,1\}$. 
A \emph{combinatorial embedding} of $G$ is a pair $(\pi,\lambda)$ consisting of a rotation
system and a signature mapping of $G$. 
Intuitively, the cyclic permutation $\pi_v$ corresponds to the clockwise order of edges around $v$
in the embedding, and $\lambda(e)$ indicates whether the edge $e$ is ``twisted'', that is,
if the clockwise/counter-clockwise orders swaps when going along $e$. 
The \emph{face transversal procedure} is defined as follows. Start from an edge $e_1$ and its endpoint $v_1$. Continue to the other endpoint $v_2$ of $e_1$ and next edge $e_2 = \pi_{v_2}^{\lambda(e_1)}(e_1)$. That is, go to the next edge at $v_2$ if $\lambda(e_1) = 1$ and the previous edge
if $\lambda(e_1) = -1$. Continue to the other endpoint $v_3$ of $e_2$ and next edge 
$e_3 = \pi_{v_3}^{\lambda(e_1)\lambda(e_2)}(e_2)$. 
That is, whenever traversing an edge $e_i$ from $v_i$ to $v_{i+1}$, set
$e_{i+1} = \pi_{v_{i+1}}^{\prod_{j=1}^{i} \lambda(e_j)}$. 
Finish at step $\ell$ when $v_{\ell+1} = v_1$, $e_{\ell+1} = e_1$, and $\prod_{j=1}^{\ell} \lambda(e_j) = 1$. 
The cyclic sequence $v_1,e_1,v_2,e_2,\ldots,v_\ell,e_\ell$ is a \emph{facial walk} and defines
a face.
It can be shown that every edge appears on exactly two facial walks.

If an embedding $\embed = (\pi,\lambda)$ of $G$ has $f$ distinct facial walks, then
its \emph{Euler genus} is $\eulerg(\embed) = 2-|V(G)|-f+|E(G)|$. 

A \emph{local change} at a vertex $v$ changes $\pi_v$ to its inverse and
changes $\lambda(e) := -\lambda(e)$ for every edge $e$ incident with $v$. 
Two embeddings are equivalent if one can be transformed into another by a sequence
of local changes. 
An embedding $(\pi,\lambda)$ is \emph{orientable} if there exists an equivalent
embedding $(\pi',\lambda')$ with $\lambda(e) = 1$ for every edge $e$.
(Equivalently, $(\pi,\lambda)$ is orientable if $\prod_{e \in C} \lambda(e) = 1$ for every
 cycle $C$ in $G$.)

A cycle $C$ is \emph{two-sided} in the embedding $\embed=(\pi,\lambda)$ if $\prod_{e \in C} \lambda(e) = 1$ and \emph{one-sided} otherwise. For a cycle $C$, there is a natural notion of \emph{cutting along $C$}: we duplicate vertices and edges of $C$ and, for every $v \in V(C)$, distribute its incident edges between the copies into the part ``left of $C$'' and ``right of $C$''. 
If $C$ is one-sided, the result is a connected graph with an embedding of strictly smaller Euler genus.
If $C$ is two-sided, the result is either a connected graph with an embedding of strictly
smaller Euler genus or two graphs with embeddings of Euler genera whose sum is the 
Euler genus of $\embed$. In the latter case, if one of the parts has zero Euler genus,
we say that $C$ is \emph{contractible}, otherwise it is \emph{noncontractible}. 
A one-sided cycle is always noncontractible.

\paragraph{Surfaces.}
Unless explicitly stated, we consider only connected surfaces.
As in~\cite{RobertsonS88}, for integers $a,b,c \geq 0$, by $\Sigma(a,b,c)$ we denote the surface created from a (2-dimensional) sphere by adding $a$ handles and $b$ crosscaps, and then removing the interiors of $c$ pairwise disjoint closed discs. Following~\cite{RobertsonS88}, we call the boundaries of these $c$ discs \emph{cuffs}. 
Every compact connected surface is homeomorphic to $\Sigma(a,b,c)$ for some $a,b,c\geq 0$. 
For a surface $\Sigma = \Sigma(a,b,c)$, the \emph{capped surface} $\widehat{\Sigma}$ is created from $\Sigma$ by glueing back the $c$ discs, i.e., $\widehat{\Sigma}$ is homeomorphic to $\Sigma(a,b,0)$.

If $c=0$, then $\Sigma(a,b,c)$ is a surface without a boundary;
we will use $\Sigma(a,b)$ as a shorthand for $\Sigma(a,b,0)$. 
Note that the Euler genus of $\Sigma(a,b)$ is $\eulerg(\Sigma(a,b)) = 2a+b$. 
If $\Sigma = \Sigma(a,b,c)$, then we denote $\epsilon(\Sigma) = 4a+2b+c = 2\eulerg(\widehat{\Sigma}) + c$.
The invariant $\epsilon(\Sigma)$ can be understood as a measure of the complexity of the surface $\Sigma$.
For example, all the following operations strictly decrease $\epsilon(\Sigma)$: cutting along a curve in order to merge two cuffs, cutting a handle into two new cuffs, or
cutting along a one-sided closed curve to turn a crosscap into a cuff.

Unless otherwise stated, whenever we speak about a surface, we mean a connected surface without boundary, that is, homeomorphic to $\Sigma(a,b)$ for some $a,b\geq 0$. All surfaces (with boundaries or not) considered in this paper are compact.

Recall that (cf. Theorem~3.1.1 of~\cite{MTbook}) every surface (possibly with boundary) 
can be created (up to homeomorphisms) from a finite even set of triangles by 
(1) orienting every triangle, (2) grouping (not necessarily all) sides of triangles into pairs, and (3) within every pair, identifying the two sides of the pair in the way that respects the orientation. The non-paired sides of the triangles form the cuffs.
Consequently, every connected surface (possibly with boundary) can be created (up to homeomorphisms) from a disc by 
(0) selecting on the boundary of the disc an even number of equal-length arcs, (1) orienting every arc, (2) grouping all arcs into pairs, and (3) within every pair, identifying the two
arcs of the pair in the way that respects the orientation.

Furthermore, it follows immediately from the proof of Theorem~3.1.1 of~\cite{MTbook} that the number of arcs required in the above construction for a surface $\Sigma$ is $\Oh(\epsilon(\Sigma))$.

\embedsection{Additivity of Euler genus}

The (minimum) Euler genus of a graph $G$ is the minimum Euler genus over all (combinatorial) embeddings of $G$.
We need the following two results on the additivity of Euler genus.
\begin{lemma}[\cite{diestel}, Lemma~B.6]\label{lem:diestelB6}
Let $\Gamma$ be a surface of Euler genus $g$ and let $\mathcal{C}$ be a set of $g+1$ disjoint circles (homeomorphic images of $S^1$) in~$\Gamma$. 
If $\Gamma \setminus \bigcup \mathcal{C}$ has a connected component $D_0$
whose closure in $\Gamma$ meets every circle in $\mathcal{C}$,
then at least one of the circles in $\mathcal{C}$ bounds a disc in $\Gamma$ that is disjoint from $D_0$.
\end{lemma}

\begin{lemma}[Stahl and Beineke~\cite{BS77}]\label{lem:genus-blocks}
The Euler genus of a graph $G$ is equal to the sum of the Euler genera of the blocks (the 2-connected components) of $G$.
\end{lemma}
\begin{corollary}\label{cor:genus-manyK5s}
Suppose a graph $G$ is constructed by taking $k$ connected nonplanar graphs, selecting one vertex from each of them, and identifying the selected vertices into a single vertex. Then the Euler genus of $G$ is at least $k$.
\end{corollary}

The next lemma is a simple corollary of Euler's formula.
\begin{lemma}\label{lem:genus:delcc}
Let $H$ be a 3-connected graph of Euler genus $h$.
Then, for every non-empty $A \subseteq V(H)$, there are at most $\max(2|A|+2h-4,1)$ connected components of $H-A$.
\end{lemma}
\begin{proof}
The claim is trivial when $|A|\leq 2$, so assume otherwise.
Consider a minor $H'$ of $H$ constructed as follows: delete all edges of $H[A]$ and 
contract every connected component of $H-A$ into a single vertex. Let $B$ be the set of vertices obtained through these contractions.
Clearly, $H'$ is of Euler genus at most $h$ and is a bipartite graph with sides $A$ and $B$.
Since $H$ is 3-connected and $|A|\geq 3$, every vertex of $B$ has degree at least $3$ in $H'$.

Fix an embedding of $H'$ of Euler genus at most $h$ and let $f'$ be the number of faces. Denote $m' = |E(H')|$ and $n' = |V(H')| = |A| + |B|$.
Since every face of this embedding is of length at least $4$, $f' \leq m'/2$. Since every vertex in $B$ is of degree at least $3$, $m' \geq 3|B|$. 
By Euler's formula
$$2-h \leq n' + f' - m' \leq n' - m'/2 \leq |A|+|B|-3|B|/2 = |A|-|B|/2.$$
Hence, $|B| \leq 2|A|+2h-4$, as desired.
\end{proof}

We will need the following theorem of Gallai.
\begin{theorem}[Gallai's theorem~\cite{Gallai61}]\label{thm:gallai}
Let $G$ be a graph, $S \subseteq V(G)$, and $\ell$ be an integer. 
Then $G$ contains at least one of the following: a family of $\ell$ vertex-disjoint paths with both endpoints in $S$,
or a set $A \subseteq V(G)$ of size at most $2\ell$ such that every connected component of $G-A$ contains at most one vertex of $S$.
\end{theorem}

By combining Lemma~\ref{lem:genus:delcc} and Theorem~\ref{thm:gallai}, we obtain the following.

\begin{corollary}\label{cor:gallai}
Let $k\geq 2$ be an integer, $H$ be a 3-connected graph of Euler genus $h$
and let $S \subseteq V(G)$ be a subset of vertices of size larger than $6k+2h-4$. 
Then $H$ contains $k$ vertex-disjoint paths with endpoints in $S$.
\end{corollary}
\begin{proof}
Assume the contrary. By Theorem~\ref{thm:gallai}, there is a set $A$ of at most $2k$ vertices
such that $H-A$ contains no path with both endpoints in $S$. But then
$H-A$ contains at least $|S \setminus A| \geq |S| - |A| > 4k+2h-4 \geq \max(2|A|+2h-4,1)$ connected
components, a contradiction to Lemma~\ref{lem:genus:delcc}. 
\end{proof}

\embedsection{Face-width, face-vertex curves, nooses}

Given a graph $G$ with an embedding $\embed$, the \emph{face-vertex graph} is the bipartite multigraph
$\widehat{G}$
with vertex set consisting of vertices and faces of $(G,\embed)$, where for every incidence between a vertex and a face in $(G,\embed)$ we add one edge between the vertex and the face in question.
The embedded graph $(G,\embed)$ naturally gives an embedding $\widehat{\embed}$ of $\widehat{G}$ into the same surface. 

The \emph{edge-width} of an embedding $\embed$ of $G$ is the minimum length of a noncontractible cycle in $G$ and $\embed$.
The \emph{face-width} of an embedded graph $(G,\embed)$ is the minimum number of facial walks
whose union contains a noncontractible cycle. Equivalently, it is half
of the edge-width of the face-vertex graph (cf.~\cite[Proposition 5.5.4]{MTbook}). 
By convention, the face-width of any planar embedding is $+\infty$. 

The \emph{radial distance} between two vertices in an embedding $\embed$ is half
of the distance between them in the face-vertex graph.

The above concepts can be also defined via \emph{face-vertex curves} and \emph{nooses}.

A \emph{face-vertex curve} is a sequence $\gamma = (v_1,f_1,v_2,f_2,\ldots,f_{\ell-1},v_\ell)$
where $v_i$ are vertices, $f_i$ are faces, and $v_i,v_{i+1}$ is incident with $f_i$
for $1 \leq i < \ell$. A \emph{self-intersection} of $\gamma$ is either a vertex appearing
at least twice on $\gamma$, except for the case $v_1 = v_\ell$ (i.e, indices $1 \leq i \neq j < \ell$, $\{i,j\} \neq \{1,\ell\}$, $v_i = v_j$),
or a face occurring twice such that their preceding and succeding
vertices interlace, that is, $f_i = f_j$ for some $1 \leq i \neq j < \ell$
with $v_i$, $v_{i+1}$, $v_j$, $v_{j+1}$ pairwise distinct and appearing in the walk around $f$ 
in the order $v_i$, $v_j$, $v_{i+1}$, $v_{j+1}$ or this order reversed.
The \emph{lenght} of a face-vertex curve is its number of faces, $\ell-1$.

A face-vertex curve has a natural topological representation as a Jordan curve in the surface
that passes through vertices $v_i$ and faces $f_i$ in the order as in $\gamma$. 
If $\gamma$ is without self-intersections, then the topological representation can be chosen
to be injective. 
A \emph{noose} is a face-vertex curve without self-intersections in which the first and 
the last vertex coincide, i.e., $v_1 = v_\ell$.

It is easy to see that the radial distance between two vertices is the length of the shortest
face-vertex curve with endpoints in the said vertices. 
Furthermore, the face-width of an embedded graph $(G,\embed)$ is the length of the shortest
noose whose topological representation is noncontractible in the surface. 

\embedsection{Walls}

An \emph{elementary wall} of height $h$ and width $w$ consists of $h$ disjoint paths $P_1,P_2,\ldots,P_h$, where each
path $P_i$ contains $2w$ vertices $(v_{i,j})_{j=1}^{2w}$ in this order.
Furthermore, for odd $1 \leq i < h$ there are edges
$v_{i,2j-1}v_{i+1,2j-1}$ for $1 \leq j \leq w$ and for even $1 \leq i < h$ there are edges $v_{i,2j}v_{i+1,2j}$ for $1 \leq j \leq w$.
If $h,w \geq 3$, then the \emph{suppressed elementary wall} of height $h$ and width $w$
is the elementary wall of height $h$ and width $w$, with first both degree-1 vertices 
deleted, and then all maximal paths with internal vertices of degree $2$ replaced with single
edges. Note that a suppressed elementary wall is 3-regular and 3-connected.

An \emph{elementary cylindrical wall} of height $h$ and width $w$ is obtained from an
elementary wall of height $h$ and width $d$ by adding edges
$v_{i,2w}v_{i,1}$ for $1 \leq i \leq h$. For $1 \leq i \leq h$, by $C_i$ we denote the cycle
consisting of the path $P_i$, closed with the edge $v_{i,2w}v_{i,1}$. 
Similarly, for $h,w \geq 2$ 
a \emph{suppressed elementary cylindrical wall} is an elementary cylindrical wall
of the same dimensions with all maximal paths with internal vertices of degree $2$ replaced with single edges. 

A \emph{wall} is any subdivision of a suppresed elementary wall and, similarly, 
a \emph{cylindrical wall} is any subdivision of a suppressed elementary cylindrical wall.

The (images under subdvision of) 
paths $P_i$ and cycles $C_i$ are called the \emph{rows} of a (cylindrical) wall.
For $1 \leq j \leq w$, the \emph{$j$-th column} of the cylindrical elementary wall
is the path 
$$v_{1,2j-1}-v_{2,2j-1}-v_{2,2j}-v_{3,2j}-v_{3,2j-1}-v_{4,2j-1}-\ldots.$$
The columns of a (cylindrical) wall are the images under subdivision of the above.

For a surface $\Sigma$ (possibly with a boundary),
we define the \emph{$\Sigma$-wall} of \emph{order $h$} as follows; see the left panel of Figure~\ref{fig:Sigma-wall} for an illustration.
Recall that $\Sigma$ can be created from a disc by selecting on the boundary of the disc $2\ell$ arcs of equal length, for some $\ell = \Oh(\epsilon(\Sigma))$,
orienting the arcs, grouping the arcs into pairs, and then identifying
the arcs within the pairs while respecting the orientation.
Let these arcs be $A_1,\ldots,A_{2\ell}$ in the counter-clockwise order around the disc, and let the pairing be $L_1,\ldots,L_\ell$ where each $L_j$ is a 2-element subset of $\{1,2,\ldots,2\ell\}$.

We start constructing an elementary $\Sigma$-wall of order $h$ from an elementary wall of height and width $2\ell h$.
For $1 \leq i \leq \ell$ and $1 \leq j \leq h$ we denote $u_{i,j} = v_{ 2\ell h - 2(i-1)h - (j-1),1}$ (these are vertices in the first column).
For $\ell < i \leq 2\ell$ and $1 \leq j \leq h$, we denote $u_{i,j} = v_{ (2i-2\ell-1)h + j,(4\ell+2) h}$ (these are vertices in the last column). 
For every pair $L_i = \{a,b\}$, we proceed as follows.
If the arcs $A_a$ and $A_b$ are oriented in the same direction (both clockwise or both counter-clockwise), then we add edges $u_{a,j} u_{b,j}$ for $1 \leq j \leq h$. 
Otherwise, if the arcs $A_a$ and $A_b$ are oriented in the opposite direction (one clockwise and one counter-clockwise), then we add edges $u_{a,j} u_{b,h+1-j}$ for $1 \leq j \leq h$.
Finally, a $\Sigma$-wall of order $h$ is any subdivision of an elementary $\Sigma$-wall of height $h$.

\begin{figure}[tb]
\begin{center}
\includegraphics{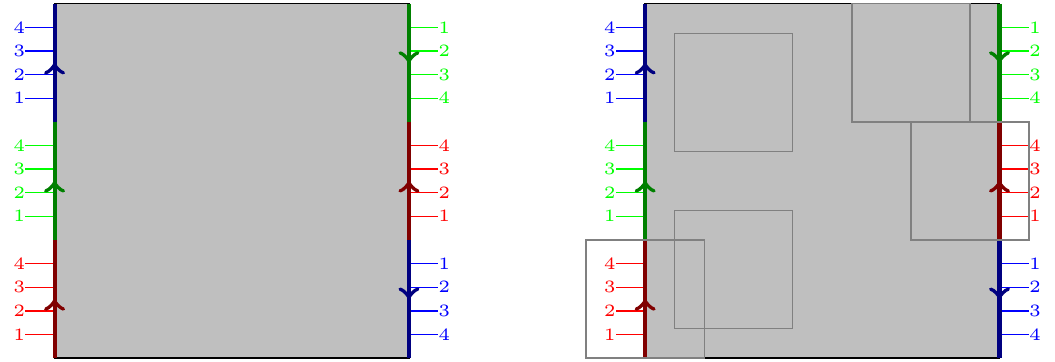}
\caption{Left panel: Construction of a $\Sigma$-wall. Colorful arrows represent arcs $A_i$, colors represent the pairs $L_j$. Two occurences of edges of the same color and number are actually two halves of the same edge. Right panel: gray squares are a few example subwalls used to cover a $\Sigma$-wall in the proof of Lemma~\ref{lem:sigma-wall-cover}.}\label{fig:Sigma-wall}
\end{center}
\end{figure}

By embedding a $\Sigma$-wall into $\Sigma$ in the natural way, following the construction of $\Sigma$ via arcs $A_1,\ldots,A_{2\ell}$, we have the following statement.
\begin{lemma}\label{lem:sigma-wall-embed}
A $\Sigma$-wall of order $h$ admits an embedding into $\Sigma$ with face-width at least $\lfloor h/2 \rfloor$. 
Furthermore, in the embedding there are no vertices on the cuffs and 
the radial distance between distinct cuffs is at least $h$.
\end{lemma}

We will also need the following two simple observations.
\begin{lemma}\label{lem:sigma-wall-cover}
Let $h \geq 4$.
In an elementary $\Sigma$-wall $W$ of order $h$ one can find a family ${\cal C}(W)$ consisting of $\Oh(\epsilon(\Sigma)^2)$ elementary walls of height and width $h$, all contained in $W$ as subgraphs, so that the following condition holds. For every edge and every vertex of $W$, either the edge or vertex in question lies within radial distance at most $\lceil h/4 \rceil$ from one of the endpoints of the arcs $A_a$ (vertices $u_{i,1}$ or $u_{i,h}$ for $i \in [2\ell]$), or 
there exists a wall $U\in {\cal C}(W)$ such that the edge or vertex in question is contained strictly between the $(\lfloor h/4 \rfloor-1)$-th row and $(h-\lfloor h/4 \rfloor+1)$-th row of $U$.
\end{lemma}
\begin{proof}
Consider the construction of $W$ presented above and let $W_0$ be the initial elementary wall of height and width $2\ell h$.
First, add to ${\cal C}(W)$ every $h \times h$ subwall of $W_0$ whose top-left corner vertex (i.e. the one with the smallest indices) is $v_{hi+\iota \lfloor h/4 \rfloor,2hj+\zeta \lfloor h/4 \rfloor}$, where $0 \leq i,j < \ell$ and $\iota,\zeta \in \{0,1,2,3\}$. 
Then, for every pair $L_i = \{a,b\}$, add to ${\cal C}(W)$ three $h \times h$ subwalls
that have the vertices $u_{a,j}$ for $j \in [h]$ contained in their $\iota \lfloor h/4 \rfloor$-th column,
for $\iota \in \{1,2,3\}$. 
It is straightforward to verify the asserted properties of ${\cal C}(W)$, see the right panel of Figure~\ref{fig:Sigma-wall}.
\end{proof}
\begin{lemma}\label{lem:sigma-wall-holes}
Let $W$ be an elementary $\Sigma$-wall of order $h$. 
Let $Z$ be a set of vertices and edges of $W$ such that there exists a
set $A \subseteq V(W)$ and an integer $r$ such that $2r+1 < h$ and every element of $Z$ is within radial
distance at most $r$ from a vertex of $A$ in the natural embedding of $W$ into $\Sigma$. 
Then $W-Z$ contains a $\Sigma$-wall of order $h-|A| \cdot (2r+1)$. 
\end{lemma}

\begin{proof}
Again, consider the construction of $W$ presented above and let $W_0$ be the initial elementary wall of height and width $2\ell h$.
(All radial distances in the proof below are with respect to the $\Sigma$-wall $W$.)

The crucial observation is as follows. Consider two vertices $v$ and $v'$ in rows $(i-1)h+j$ and $(i'-1)h+j'$, respectively, for $1 \leq i,i' \leq \ell$, $1 \leq j,j' \leq h$.
Consider a shortest radial path between them. It either traverses all rows in between or goes via one of the vertices in the row $v_{(i''-1)h,1}$, $v_{(i''-1)h+1,1}$, $v_{(i''-1)h,2\ell h}$, $v_{(i''-1)h+1,2\ell h}$
for $1 \leq i'' \leq \ell$, or via a vertex in the first or last row. 
We infer that the radial distance between $v$ and $v'$ is at least
\begin{equation}\label{eq:mark-sigma-wall-holes}
\min(|j-j'|, (j-1) + 1 + (j'-1), (h-j) + 1 + (j'-1), (j-1) + 1 + (h-j'), (h-j) + 1 + (h-j'). 
\end{equation}
We infer that for every $1 \leq i'' \leq \ell$, there is a vertex $v'$ within radial distance at most $r$ from $v$ in row $(i''-1)h+j''$ for at most $2r+1$ indices $1 \leq j'' \leq h$. 

For every $v \in A$, if $v$ lies in row $(i-1)h+j$, then we mark all rows $(i'-1)h+j'$ that satisfy~\eqref{eq:mark-sigma-wall-holes}. Observe that 
if a vertex $u$ is within radial distance $r$ from $v \in A$ and $u$ is not on a marked row that is one of the first $r$ or last $r$ rows, then 
the column index of $v$ and the column index of $u$ differ by at most $r$, if we assume cyclic order of columns. 
We infer that by discarding the edges and vertices of degree $2$ of the $(2r+1)|A|$ marked rows and all columns containing vertices $u$ as in the previous sentences, we obtain
the desired $\Sigma$-wall.
\end{proof}

\embedsection{Large face-width implies surface minor}

The following theorem is the main result of~\cite{RobertsonS88}; cf. (18.1) in~\cite{GM8}. We remark that the statement in~\cite{GM8} is more general, as it also considers the case when $H$ has vertices on cuffs. We will not need this level of generality here, hence the statement is simplified.

\begin{theorem}[\cite{RobertsonS88}]\label{thm:GM7}
For every surface $\Sigma$ (possibly with a boundary) and a graph $H$ embedded in $\Sigma$ without any vertex on a cuff, there exists a constant $f_{\mathrm{fwf}}(\Sigma,H)$ such that every graph $G$
embedded in $\Sigma$ without a vertex on a cuff, with facewidth at least $f_{\mathrm{fwf}}(\Sigma,H)$ and with radial distance between distinct cuffs at least
$f_{\mathrm{fwf}}(\Sigma, H)$, admits $H$ as a surface minor. 

Furthermore, function $f_{\mathrm{fwf}}$ is computable.
\end{theorem}

A discussion of the computability statement of Theorem~\ref{thm:GM7} is in order.
The proof of~\cite{RobertsonS88} provides only the existential part of Theorem~\ref{thm:GM7},
as one step in the proof --- statement (3.5) in~\cite{GM8} --- is proven only
in purely existential way. 
This issue has been addressed first in the last paragraph of~\cite{GM13}, where the authors
argue (albeit not rigorously)
that statement (3.5) in~\cite{GM8} admits a proof that gives a computable bound, which directly translates to the computability of function $f_{\mathrm{fwf}}$ provided by Theorem~\ref{thm:GM7}.
Then, Geelen, Huynh, and Richter in~\cite{GeelenHR18} addressed the issue rigorously and provided
an explicit bound for the statement (3.5) of~\cite{GM8} with an additional
assumption that was sufficient for their needs. 
For completeness, in Appendix~\ref{app:GM7fix} we provide 
a proof of an explicit bound for the said statement (3.5) of~\cite{GM8} without the assumption
made by~\cite{GeelenHR18}. In fact, we show how to deduce the general case from their result,
 circumventing the additional assumption. 

We will mainly use Theorem~\ref{thm:GM7} for $\Sigma$-walls. 
For convenience, by $f_{\mathrm{fwf}}(\Sigma,h)$ we denote $f_{\mathrm{fwf}}(\Sigma,H)$ where $H$ is an elementary $\Sigma$-wall of order $h$.

\embedsection{Uniqueness of an embedding}
Recall that two embeddings are considered equivalent if they have the same sets of facial walks. A connected graph $G$ has a {\em{unique embedding}} in a surface $\Sigma$ if all its embeddings are pairwise equivalent in this sense.
An important property is that large face-width makes the embedding unique.

\begin{theorem}[Whitney's theorem]\label{thm:whitney}
A 3-connected planar graph has a unique plane embedding.
\end{theorem}
\begin{theorem}[\cite{SeymourT96}]\label{thm:fw-unique}
There is a function $f_{\mathrm{ue}}(g) \in \Oh(\log g / \log \log g)$
such that if $G$ is $3$-connected and 
$\embed$ is an embedding of $G$ of Euler genus $g$
 and face-width at least $f_{\mathrm{ue}}(g)$,
 then $g$ is the mininum Euler genus of $G$ and $\embed$ is the unique embedding of 
 Euler genus $g$.
\end{theorem}
An important corollary of the uniqueness of the embedding is the following.
\begin{corollary}\label{cor:small-auto}
Assume we are given a 3-connected graph $G$ together with an embedding $\embed$
of Euler genus $g$ and face-width at least $f_{\mathrm{ue}}(g)$. 
Then the automorphism group of $G$ is of size at most $4|E(G)|$ and can be computed in polynomial time.
Furthermore, one can in polynomial time compute an isomorhism-invariant family of size at most $4|E(G)|$
of bijections $V(G) \to [|V(G)|]$.
\end{corollary}
\begin{proof}
The uniqueness of the embedding implies that any automorphism of $G$ needs to extend to an automorphism of the embedded graph $(G,\embed)$.
For any fixed pair $(e,v)$ with $e \in E(G)$ and $v \in e$ there are $2|E(G)|$ choices of the image of $(e,v)$ under an automorphism,
and additionally a choice if the automorphisms reverses the order of the incident edges at $v$ in the embedding.
This one-in-$4|E(G)|$ choice determines the entire automorphism of $(G,\embed)$.

Similarly, starting from the above one-in-$4|E(G)|$ choice one can define a bijection $V(G) \mapsto [|V(G)|]$ that maps, say, $v$ to $1$, the other endpoint of $e$ to $2$,
and then propagates in some deterministic manner across the embedding.
For example, we can traverse $G$ by a depth-first search tree, rooted in $v$, with the other endpoint of $e$ as the first explored neighbor of $v$, and then, upon entering a vertex $v'$
with an edge $e'$, explore the neighbors of $v'$ in the cyclic order from the embedding around $v'$, starting from $e'$.
Note that the initial one-in-$4|E(G)|$ includes the choice of $e \in E(G)$, the choice of $v \in e$, and
the direction of the exploration of the cyclic order in the embedding around $v$ starting from $e$.
The vertices are enumerated in the preorder of this depth-first search tree.
The set of all bijections $V(G)\mapsto [|V(G)|]$ constructed in this manner has size $4|E(G)|$ and is isomorphism-invariant.
\end{proof}

\embedsection{Radial balls}

Let $\embed$ be an embedding of a graph $G$. 
Following~\cite{MTbook},
for a vertex $v\in V(G)$,
we define \emph{radial balls} $\rball{0}{v}, \rball{1}{v}, \ldots$
as follows: $\rball{0}{v}$ consists only of vertex $v$ and for $i > 0$, $\rball{i}{v}$ consists of $\rball{i-1}{v}$ and the union
of all facial walks of $(G,\embed)$ that contain at least one vertex of $\rball{i-1}{v}$. 
Furthermore, for $i>0$, $\partial \rball{i}{v}$ is the set of edges of $\rball{i}{v}$
that are not incident with a vertex of $\rball{i-1}{v}$.
That is, the radial balls are formally subgraphs of $G$.

A vertex $u$ is \emph{within radial distance at most $r$} from $v$
if $u \in V(\rball{r}{v})$. 
A face $f$ is \emph{within radial distance at most $r$} from $v$
if $f$ is incident with a vertex $u \in V(\rball{r-1}{v})$
(i.e., the facial walk of $f$ is part of the union in the definition of $\rball{r}{v}$). 

We will need the following result:
\begin{lemma}[\cite{RobertsonS88}, cf.~Proposition 5.5.10 of~\cite{MTbook}]\label{lem:rballs}
Let $\embed$ be an embedding of a 3-connected graph $G$ of positive Euler genus and 
face-width at least $2a+1$ for an integer $a$ and let $v \in V(G)$. 
Then there exist vertex-disjoint contractible cycles $C_1,C_2,\ldots,C_a$
such that for every $1 \leq i\leq a$, $C_i$ is contained in $\partial \rball{i}{v}$
and the closed disc bounded by $C_i$ contains the whole $\rball{i}{v}$.
\end{lemma}

For planar embeddings, we can have the following analog of Lemma~\ref{lem:rballs}.
\begin{lemma}\label{lem:rballs:planar}
Let $\embed$ be a plane embedding of a 2-connected graph $G$,
let $a$ be an integer, and let $u,v \in V(G)$ be two vertices at radial distance larger than $a$.
Then there exist vertex-disjoint cycles $C_1,C_2,\ldots,C_a$
such that for every $1 \leq i \leq a$, the cycle $C_i$
separates $v$ and all cycles $C_{i'}$, $i' < i$, from 
$u$ and all cycles $C_{i'}$, $i' > i$.
\end{lemma}
\begin{proof}
Since $G$ is $2$-connected, every facial walk in $G$ is a simple cycle.
We infer that $\rball{i}{v}$ is $2$-connected for every $i \geq 1$. 
Define $C_i$ as the facial walk around the face of $\rball{i}{v}$
that contains $u$ in its interior. 
That $C_i$ is a simple cycle follows from the 2-connectivity of $\rball{i}{v}$. The separation property for cycles $C_1,\ldots,C_a$ follows from that $C_i$ consists of vertices at radial distance exactly $i$ from $v$.
\end{proof}

\embedsection{Locally bounded treewidth of embedded graphs}

We will also need the well-known observation that graphs embedded in surfaces have locally bounded treewidth.
\begin{theorem}[\cite{DemaineH04}]
There exists a function $f_\mathrm{ltw}$ such that 
for every connected graph $G$ of Euler genus $g$ and radius $r$,
    the treewidth of $G$ is bounded by $f_\mathrm{ltw}(g) \cdot r$.
For planar graphs, one can take $f_\mathrm{ltw}(0) = 3$.
\end{theorem}
Applying the above to the union of an embedded graph and its face-vertex graph, we obtain:
\begin{corollary}\label{cor:ltw}
There exists a function $f_\mathrm{ltw}$ such that 
for every connected graph $G$, an embedding $\embed$ of $G$ of Euler genus $g$,
    a vertex $v \in V(G)$,
    and radius $r$, the treewidth of $\rball{r}{v}$ is bounded
    by $f_\mathrm{ltw}(g) \cdot r$. 
For planar graphs, the treewidth of $\rball{r}{v}$ is bounded by $6r$.
\end{corollary}

\embedsection{Unbreakability tangle}

A \emph{tangle of order $k$} in a graph $G$ is a set $\mathcal{T}$ of separations of $G$, all of order less than $k$, such that:
\begin{itemize}
\item for every separation $(A,B)$ of order less than $k$, either $(A,B)$ or $(B,A)$ belongs to $\mathcal{T}$;
\item for all $(A_1,B_1),(A_2,B_2),(A_3,B_3) \in \mathcal{T}$, we have $A_1 \cup A_2 \cup A_3 \neq V(G)$. 
\end{itemize}
Note that the above implies that for every separation $(A,B)$ of order less than $k$, exactly one of $(A,B)$ and $(B,A)$ belongs to $\mathcal{T}$.
Thus, a tangle of order $k$ is in fact an orientation of every separation of order less than $k$. If $(A,B) \in \mathcal{T}$, then we call $A$ the \emph{small side} of the separation
$(A,B)$ and $B$ the \emph{big side}. 

Assume that $G$ is $(q,k)$-unbreakable and $|V(G)| > 3q$.
Consider a separation $(A,B)$ of order at most $k$. By unbreakability, $|A| \leq q$ or $|B| \leq q$.
Since $|V(G)| > 3q$, we have either $|A| \leq q$ and $|B| > 2q$ or $|B| \leq q$ and $|A| > 2q$. 
This naturally defines a tangle $\mathcal{T}$ in $G$ of order $k+1$ that contains all separations $(A,B)$ of order at most $k$ with $|A| \leq q$;
note that the assumption $|V(G)| > 3q$ implies the second property of a tangle (no three small sides sum up to the entire graph). 
Henceforth we call this tangle the \emph{unbreakability tangle} of $G$ (if the parameters $(q,k)$ are clear from the context). 

For a set $\Apices \subseteq V(G)$ and a tangle $\mathcal{T}$ of order $k$, if $k > |\Apices|$ then there is a natural notion
of the tangle $\mathcal{T}-\Apices$ in the graph $G-\Apices$: $(A,B) \in \mathcal{T}-\Apices$ if and only if $(A \cup \Apices,B \cup \Apices) \in \mathcal{T}$.
A direct check shows that $\mathcal{T}-\Apices$ is indeed a tangle of order $k-|\Apices|$.

\embedsection{Miscellaneous}

We will use the following observation; the easy proof is left to the reader.

\begin{lemma}\label{lem:brk-wall}
A 1-subdivision of an elementary $k \times k$ cylindrical wall is $(100\ell^2,\ell)$-unbreakable for every $\ell \leq k/10$.
\end{lemma}

\section{Near-embeddings}
We start the proof of Theorem~\ref{thm:rigid} with an observation that
if the treewidth of $G$ is bounded by $\funtw(H,\sum_{i=0}^\zeta q_i)$
(with $\funtw$ emerging from the proof)
then the algorithm can return $\Vgenus = \emptyset$, $\Vtw = V(G)$, and a singleton family with an empty labelling as $\famgenus$.
Since the treewidth of a graph can be computed in FPT time parameterized by the value of the treewidth~\cite{Bodlaender96},
and $\funtw$ is computable, we can discover whether this case applies and conclude. 

In a number of places in the proof we will assume that the treewidth of $G$ is sufficiently large (larger than any computable function of $H$ and $\sum_{i=0}^\zeta q_i$ fixed in the context), increasing the function $\funtw$ if necessary. For simplicity, we refrain from providing an explicit formula
on $\funtw$, and just define it as any sufficiently quickly-growing computable function 
so that in all places the assumed treewidth bound holds.

\medskip

We will heavily rely on the structure theorem for graphs with an excluded minor of Robertson and Seymour~\cite{GM16}. The central notion in this theorem is that of a \emph{near-embedding}, hence let us recall the setting.

\embedsection{Basic definitions and notation}

In notation, we mostly follow~\cite{DiestelKMW12}. 

\paragraph{Near-embeddings.}
Given a graph $G$, a \emph{near-embedding} of $G$ consists of:
\begin{enumerate}
\item A set $\Apices \subseteq V(G)$, called the \emph{apices}.
\item A partition of the graph $G-\Apices$ into edge-disjoint subgraphs
$$G-\Apices = \MainA \cup \bigcup_{i=1}^{\nvortices} \VortexA{i} \cup \bigcup_{i=1}^{\ndongles} \DongleA{i}.$$
\item An embedding $\embed$ of $\MainA$ into a (compact) surface $\Sigma$.
\item For every $1 \leq i \leq \nvortices$, the set $V(\VortexA{i}) \cap V(\MainA)$ is enumerated as $\{\SocietyVtx{i}{1}, \ldots, \SocietyVtx{i}{\vortexlength{i}}\}$
so that $\VortexA{i}$ admits a path decomposition $(\VortexBag{i}{1}, \ldots, \VortexBag{i}{\vortexlength{i}})$ with $\SocietyVtx{i}{j} \in \VortexBag{i}{j}$ for every $1 \leq j \leq \vortexlength{i}$.
The graph $\VortexA{i}$ is called a \emph{vortex} and the vertices $\SocietyVtx{i}{j}$ are called the \emph{society vertices} of the vortex $\VortexA{i}$. The number of society vertices, $\vortexlength{i}$, is the {\em{length}} of a vortex.
\end{enumerate}
The following properties have to be satisfied:
\begin{enumerate}
\item Every vertex of $G-\Apices$ that belongs to at least two subgraphs of the partition, belongs to $\MainA$. 
\item For every $1 \leq i \leq \ndongles$, $|V(\DongleA{i}) \cap V(\MainA)| \leq 3$. (The graphs $\DongleA{i}$ are henceforth called \emph{dongles}.)
\item The graphs $(\VortexA{i})_{i=1}^{\nvortices}$ are pairwise vertex-disjoint.
\item On the surface $\Sigma$ there exist closed discs with disjoint interiors $(\DongleDisc{i})_{i=1}^{\ndongles}$ and $(\VortexDisc{i})_{i=1}^{\nvortices}$
such that $\embed$ embeds $\MainA$ into $\Sigma \setminus \left(\bigcup_{i=1}^{\ndongles} \DongleDisc{i} \cup \bigcup_{i=1}^{\nvortices} \VortexDisc{i}\right)$ and the following conditions hold:
\begin{itemize}
\item For every $1 \leq i \leq \ndongles$, the vertices of $\MainA$ that lie on the boundary of $\DongleDisc{i}$ are exactly the vertices of $V(\DongleA{i}) \cap V(\MainA)$.
\item For every $1 \leq i \leq \nvortices$, the vertices of $\MainA$ that lie on the boundary of $\VortexDisc{i}$ are exactly the vertices of $V(\VortexA{i}) \cap V(\MainA)$ and, furthermore,
  these vertices lie around the boundary of $\VortexDisc{i}$ in the order $\SocietyVtx{i}{1}, \ldots, \SocietyVtx{i}{\vortexlength{i}}$.
\end{itemize}
\end{enumerate}

A few remarks are in place. We do not mandate any size bound for a dongle; however, since the studied graph will have always strong unbreakability properties,
a bound on the size of the dongle will follow from unbreakability, as $(V(\DongleA{i}) \cup \Apices, V(G) \setminus (V(\DongleA{i}) \setminus V(\MainA)))$ is a separation of bounded order.
Furthermore, the dongles may share vertices between themselves and with vortices, but all such shared vertices must belong to $\MainA$ as well.

In what follows, we usually consider one fixed near-embedding of a graph.
In that case, we denote the near-embedding by $\nembed$ and the corresponding elements
of the near-embedding as listed above: $\Apices$, $\MainA$, $\VortexA{i}$, etc. 
However, in some arguments we consider one or two near-embeddings at the same time
and try to modify or combine them into another near-embedding. 
In these cases, the source near-embeddings will be denoted by $\nembedA$ and $\nembedB$
and the constructed near-embedding as $\nembedC$; the corresponding elements
of the near-embedding $\nembedB$ will be denoted as $\MainB$, $\VortexB{i}$, etc., and
the corresponding elements of the near-embedding $\nembedC$ will be denoted as $\MainC$, $\VortexC{i}$, etc.
In other places, given a near-embedding $\nembed$ of $G$,
we will reason about a (properly defined) induced near-embedding of the main (largest) 3-connected
component of $G-\Apices$. In these places, the near-embedding in question is denoted
by $\nembedT$ with corresponding pieces $\MainT$, $\VortexT{i}$, etc.

Fix a near-embedding $\nembed$ of $G$.
With $\nembed$ we can associate a number of ``quality measures'':
\begin{itemize}
\item the Euler genus of $\Sigma$, denoted $\eulerg$;
\item the number of vortices, $\nvortices$;
\item the maximum adhesion size of a vortex, $\vortexadhwidth \coloneqq \max_{1 \leq i \leq \nvortices} \max_{1 \leq j < \vortexlength{i}} |\VortexBag{i}{j} \cap \VortexBag{i}{j+1}|$;
\item the maximum bag size of a vortex, $\vortexbagsize \coloneqq \max_{1 \leq i \leq \nvortices} \max_{1 \leq j \leq \vortexlength{i}} |\VortexBag{i}{j}|$;
\item the number of apices, denoted $\napices \coloneqq |\Apices|$.
\end{itemize}

\paragraph{Extended version.}
The assumptions on the embedding $\embed$ in a near-embedding $\nembed$ motivate us to define the \emph{extended version} $\MainPlusA$ of $\MainA$
as follows. The graph $\MainPlusA$ and its embedding is obtained from the graph $\MainA$ with the embedding $\embed$ by adding the following additional edges and vertices and their embeddings in $\Sigma$ (see also Figure~\ref{fig:mainplus}):
\begin{itemize}
\item For every $1 \leq i \leq \ndongles$, we do the following depending on the size of $V(\DongleA{i}) \cap V(G_0)$:
\begin{itemize}
\item If $|V(\DongleA{i}) \cap V(G_0)| = 3$, 
we add a new vertex $\PlusDongleVtx{i}$ in the disc $\DongleDisc{i}$ and turn $\{\PlusDongleVtx{i}\} \cup (V(\DongleA{i}) \cap V(G_0))$ into a clique, embedding all its edges in $\DongleDisc{i}$.
\item If $|V(\DongleA{i}) \cap V(G_0)| = 2$, we add a new edge between the vertices of $V(\DongleA{i}) \cap V(G_0)$ embedded in the disc $\DongleDisc{i}$, 
  and proclaim one of the vertices of $V(\DongleA{i}) \cap V(G_0)$ as $\PlusDongleVtx{i}$. 
\item If $|V(\DongleA{i}) \cap V(G_0)| = 1$, we denote the unique vertex of $V(\DongleA{i}) \cap V(G_0)$ as $\PlusDongleVtx{i}$. 
\item If $V(\DongleA{i}) \cap V(G_0) = \emptyset$ we do nothing; in particular, we do not define $\PlusDongleVtx{i}$. 
\end{itemize}
\item For every $1 \leq i \leq \nvortices$, 
we turn $V(\VortexA{i}) \cap V(G_0)$ into a cycle $\SocietyVtx{i}{1}, \ldots, \SocietyVtx{i}{\vortexlength{i}}$, embedding all its edges into $\VortexDisc{i}$, and 
we add a vertex $\PlusVortexVtx{i}$ into $\VortexDisc{i}$ inside the face surrounded
by the said cycle and add all edges $\PlusVortexVtx{i}\SocietyVtx{i}{j}$, $1 \leq j \leq \vortexlength{i}$, embedded into the said face.
\end{itemize}
\begin{figure}[tb]
\begin{center}
\includegraphics{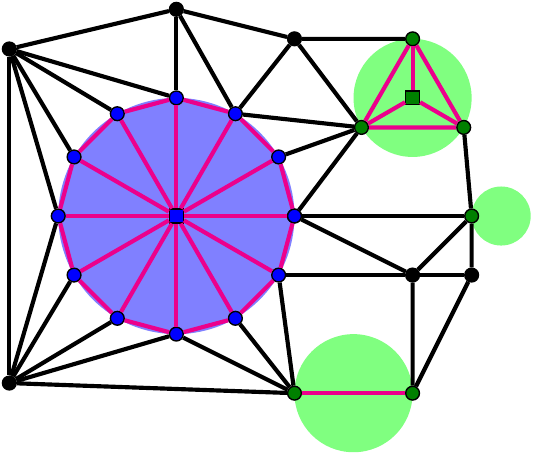}
\caption{Exemplary extended version $\MainPlusA$ in a near-embedding with one vortex
(blue) and three dongles (green) with 1, 2, and 3 vertices common with $\MainA$, respectively.
The blue circular vertices are society vertices of the vortex and the blue square vertex
is the virtual vortex vertex. 
The green circular vertices are vertices in common with a dongle and $\MainA$
and the green square vertex is the virtual dongle vertex in case of a dongle
with three vertices in common with $\MainA$. 
The pink edges are virtual edges added to $\MainPlusA$.}\label{fig:mainplus}
\end{center}
\end{figure}
The vertex $\PlusDongleVtx{i}$ if newly created and the edges of $E(\MainPlusA) \setminus E(\MainA)$ that are
added because of $\DongleA{i}$ are called the \emph{virtual dongle vertex} and the \emph{virtual dongle edges} for the dongle $\DongleA{i}$. 
Similarly, 
the vertex $\PlusVortexVtx{i}$ and the edges of $E(\MainPlusA) \setminus E(\MainA)$ that are
added because of $\VortexA{i}$ are called the \emph{virtual vortex vertex} and the \emph{virtual vortex edges} for the vortex $\VortexA{i}$. 
Both virtual dongle vertices/edges and virtual vortex vertices/edges are called jointly
\emph{virtual vertices/edges}. 

Note that we treat $\MainPlusA$ as a graph embedded in $\Sigma$ with the above embedding
extending $\embed$.

\paragraph{Removing vertices from parts of a near-embedding.}
Consider a near-embedding $\nembed$ of $G$.
Let $\VortexA{i}$ be a vortex of $\nembed$ and let $v \in V(\VortexA{i})$.
The operation of \emph{removing $v$ from $\VortexA{i}$} consists of removing $v$ from $\VortexA{i}$ 
(and all its incident edges in $\VortexA{i}$) and, furthermore, if $v$ is a society vertex
of $\VortexA{i}$,
say $v = \SocietyVtx{i}{j}$, then we do the following.
\begin{itemize}[nosep]
\item We shrink a bit the disc associated with $\VortexA{i}$ in a close neighborhood of $v$ so that $v$ is no longer
embedded in its boundary.
\item If $\vortexlength{i} > 1$, we pick any $j' \in \{j-1,j+1\} \cap [\vortexlength{i}]$
and merge $\VortexBag{i}{j}$ and $\VortexBag{i}{j'}$ into a new bag $\VortexBag{i}{j} \cup \VortexBag{i}{j'} \setminus \{v\}$ with society vertex $\SocietyVtx{i}{j'}$.
\item If $\vortexlength{i} = 1$, we change the vortex $\VortexA{i}$ into a dongle $\VortexA{i} \setminus \{v\}$ (keeping the same disc associated with it).
\end{itemize}

For $v \in V(\DongleA{i})$ for some dongle $\DongleA{i}$,
the operation of \emph{removing $v$ from $\DongleA{i}$} is defined analogously:
we remove $v$ from $\DongleA{i}$ (together with all incident edges in $\DongleA{i}$) and, furthermore,
   if $v \in V(G_0)$, we shrink a bit the disc associated with $\DongleA{i}$ so that $v$
   is no longer embedded in its boundary. 

The operation of removing $v$ from a vortex or a dongle 
will usually be conducted together with another
operation that places $v$ (and removed incident edges) somewhere else in the near-embedding.
In particular, 
the operation of moving a vertex $v \in V(G) \setminus \Apices$ to the apex set is defined
as follows: we add $v$ to $\Apices$, remove it from all the parts --- first vortices and
dongles in arbitrary order, and then $\MainA$ --- to which it belongs (triggering the processes described in the previous two paragraphs),
and remove it from $\embed$ if it belongs to $\MainA$. 
This operation results in a new near-embedding with one more apex,
where none of the following parameters increase: number of vortices, Euler genus, and the maximum adhesion size of a vortex.

\paragraph{Taking minors and near-embeddings.}
In the previous paragraphs we have presented how a near-embedding changes under vertex deletion.
Observe that in case of edge-deletion, there is a simple natural way of keeping the near-embedding:
just delete the edge from the corresponding part of the near embedding (unless it is incident
with an apex, in which case we do not do anything). 

However, if one wants to contract an edge $e = uv$ of $G$, a near-embedding $\nembed$ of 
$G$ can change in an unpredictable way if both $u$ and $v$ are society vertices
(of the same or distinct vortices). Therefore, we only define how contracting an edge $e=uv$
of $G$ affects the near-embedding $\nembed$ of $G$ if $u$ or $v$ is not a society vertex
of $\nembed$. More formally, if $\widetilde{G}$ is constructed from $G$ by contracting an edge
$e=uv$ into a new vertex $x_e$,
and $\nembed$ is a near-embedding of $G$ where at least one of $u$ and $v$ is not a society vertex,
then the induced near-embedding $\nembedB$ of $\widetilde{G}$ is defined as $\nembed$ with 
the following modifications:
\begin{itemize}
\item If either $u$ or $v$ is an apex of $\nembed$, $x_e$ is an apex of $\nembedB$;
that is, if the second vertex of $\{u,v\}$ is not an apex, we remove it from the near embedding,
     and set $\ApicesB \coloneqq \ApicesA \setminus \{u,v\} \cup \{x_e\}$. 
\item Otherwise, if $uv$ is an edge of a vortex $\VortexA{i}$, we contract this edge
in the graph $\VortexA{i}$, replace all occurences of $u$ and $v$ with $x_e$ in the path decomposition of $\VortexA{i}$ and, if $u$ or $v$ is a society vertex, we proclaim $x_e$ also a society
vertex with the same bag and placement on $\Sigma$ as $u$ or~$v$. (Recall that we exclude the
    case when both $u$ and $v$ are society vertices.)
\item Otherwise, if $uv$ is an edge of a dongle $\DongleA{j}$ and
$u \notin V(\MainA)$ or $v \notin V(\MainA)$, 
we contract this edge in the graph
$\DongleA{j}$ and put $x_e$ to $\MainB$ if and only if either $u$ or $v$ is in $\MainA$.
(Note that the exclusion of the case $u,v \in V(\MainA)$ 
 allows us to keep the same disc for $\DongleA{j}$.)
\item In the remaining case, $uv$ is an edge of $\MainA$ or $uv$ is an edge of a dongle $\DongleA{j_0}$
with $u,v \in V(\MainA) \cap V(\DongleA{j_0})$. Note that in the latter case, $\MainPlusA$ contains an edge $uv$
embedded in the disc of $\DongleA{j_0}$. In the graph $\MainPlusA$, we contract $v$ onto $u$ along this edge $uv$, obtaining a graph $\MainPlusB$ embedded in the same surface as $\MainPlusA$.
(The choice of $u$ and $v$ is arbitrary in this contraction.)
We remark that $\MainPlusA$ may contain multiple copies of the edge $uv$, possibly being virtual edges of other dongles
    or just edges of $\MainA$; in this process we contract the edge corresponding to the edge $uv$ we contract in $G$.

Starting with $\MainPlusB$ described as above, we complete the construction of the near-embedding $\nembedB$ from $\nembedA$ as follows:
\begin{itemize}
\item If $u$ or $v$ is a society vertex of a vortex $\VortexA{i}$, then $u$ becomes a society vertex of $\VortexB{i}$ instead.
In the case $u$ is a society vertex of $\VortexA{i}$, the disc of $\VortexB{i}$ equals the disc of $\VortexA{i}$. 
If $v$ is a society vertex of $\VortexA{i}$, modify the disc of $\VortexA{i}$ by extending it along the embedding of the edge $uv$ so catch $u$ on its boundary.
\item If exactly one of $u$ or $v$ is a vertex of $V(\MainA) \cap V(\DongleA{j})$ for a dongle $\DongleA{j}$, similarly $u$ becomes a vertex of $V(\MainB) \cap V(\DongleB{j})$.
If $u \in V(\MainA) \cap V(\DongleA{j})$, then the disc of $\DongleB{j}$ is the disc of $\DongleA{j}$, and 
if $v \in V(\MainA) \cap V(\DongleA{j})$, then the disc of $\DongleB{j}$ is the disc of $\DongleA{j}$, extended along the edge $uv$ to catch $u$ on its boundary.

There can be multiple dongles with $\DongleA{j}$ with $v \in V(\MainA) \cap V(\DongleA{j})$ and, additionally, one vortex $\VortexA{i}$ with $v$ as a society vertex. 
In $\nembed$, the edge $uv$ and the discs of these dongles and vortices are ordered cyclically around $v$; this cyclic order is also an order in which we extend the said discs to catch $u$
on their boundaries, so that these discs are pairwise disjoint. 
\item If $u,v \in V(\MainA) \cap V(\DongleA{j})$ for a dongle $\DongleA{j}$, we distinguish two cases.
\begin{itemize}
\item If $|V(\MainA) \cap V(\DongleA{j}| = 2$, then after the contraction of $uv$ in $\MainPlusA$ the virtual edge of $\DongleA{j}$ may still be present as a loop at $u$.
    If this is the case, delete it. Define the new disc of $\DongleA{j}$ to be a small disc with $u$ on its boundary and proclaim $u$ to be the new $\PlusDongleVtx{j}$. 
\item If $|V(\MainA) \cap V(\DongleA{j}| = 3$, then after the contraction of $uv$ in $\MainPlusA$ there is a double edge $\PlusDongleVtx{j} u$ and a single edge $\PlusDongleVtx{j} w$
  where $w$ is the third vertex of $V(\MainA) \cap V(\DongleA{j})$. Replace the virtual edges and vertices of $\DongleA{j}$ with an edge $wu$, drawn in the place
  of the edge $\PlusDongleVtx{j} w$ and one of the copies of $\PlusDongleVtx{j} u$, define the new disc of $\DongleA{j}$ to be a small open neighborhood of the interior of the new edge $wu$, 
  with $u$ and $w$ on its boundary, 
  and proclaim $u$ to be $\PlusDongleVtx{j}$. 
\end{itemize}
\end{itemize}
\end{itemize}

In particular, if $\widetilde{G}$ is a minor of $G$, and $\nembed$ is a near-embedding of $G$
without vortices, the above process describes a near-embedding $\nembedB$ of $\widetilde{G}$
that does not have vortices, its number of apices is not greater than the number of apices in $\nembedA$, and its Euler genus is not greater than the one of $\nembedA$.

\embedsection{Projecting a subgraph to $\MainPlusA$ and lifting a subgraph of $\MainPlusA$}
Consider a near-embedding $\nembed$ of $G$ and a subgraph $J$ of $G-\Apices$.
Observe that $J$ naturally projects to a subgraph $\SubProj_{\nembed}(J)$ of $\MainPlusA$
(called henceforth \emph{the projection of $J$ onto $\MainPlusA$}) as follows:
\begin{itemize}
\item $\SubProj_{\nembed}(J)$ contains all vertices and edges of $J$ that are present in $\MainA$;
\item for every dongle $\DongleA{i}$ that contains an edge of $J$, we add $\PlusDongleVtx{i}$ to $\SubProj_{\nembed}(J)$
together with all the edges between $\PlusDongleVtx{i}$ and the vertices of $V(\MainA) \cap V(\DongleA{i}) \cap V(J)$;
\item for every vortex $\VortexA{i}$ that contains an edge of $J$, we add $\PlusVortexVtx{i}$ to $\SubProj_{\nembed}(J)$ together with all the edges between $\PlusVortexVtx{i}$ and the vertices of $V(\MainA) \cap V(\VortexA{i}) \cap V(J)$.
\end{itemize}
We note that if two vertices of $V(J) \cap V(\MainA)$ are in the same connected component of $J$, then they are also in the same connected
component of $\SubProj_{\nembed}(J)$. However, the converse implication may not hold:
for example, two connected components of $J$ that both contain edges of the same vortex are projected to the same connected component of $\SubProj_{\nembed}(J)$. 

Slightly abusing the notation, we also extend the domain of the projection $\SubProj_{\nembed}$
to single vertices $v \in V(G) \setminus \Apices$. If $v \in V(\MainA)$, then
$\SubProj_\nembed(v) = v$.
If $v \in V(\DongleA{j}) \setminus V(\MainA)$, then $\SubProj_\nembed(v) = \PlusDongleVtx{j}$. 
If $v \in V(\VortexA{i}) \setminus V(\MainA)$, then $\SubProj_\nembed(v) = \PlusVortexVtx{i}$.

Observe that if $J$ is a path in $G-\Apices$ with two endpoints in $\MainA$, then $\SubProj_{\nembed}(J)$ may not be a path, or even a walk between the same two endpoints.
For this reason, we will also need another variant of projection, applicable only to paths with fixed orientation. 
For a path $J$ with a fixed orientation (i.e., indication which endpoint is the starting point and which endpoint is the ending point of the path $J$) and with both endpoints in $\MainA$, 
$J$ projects to a path $\PathProj_{\nembed}(J)$ in $\MainPlusA$
(called henceforth \emph{the path projection of $J$ onto $\MainPlusA$}) as follows:
\begin{itemize}
\item For every dongle $\DongleA{i}$ that contains an edge of $J$, we replace 
the maximal subpath of $J$ contained in $\DongleA{i}$
with a one- or two-edge path with the same endpoints and $\PlusDongleVtx{i}$ as one of its vertices.
(Note that as $|V(\MainA) \cap V(\DongleA{i})| \leq 3$, there may be at most one such path on $J$.)
\item We iteratively pick the first edge on $J$ that belongs to a vortex, say vortex $\VortexA{i}$, then we pick
the first and the last vertex of $\VortexA{i}$ on $J$, say
$\SocietyVtx{i}{a}$ and $\SocietyVtx{i}{b}$ (since the endpoints of $J$ are in $\MainA$,
    these vertices are society vertices of $\VortexA{i}$),
  and we replace the subpath of $J$ between $\SocietyVtx{i}{a}$ and $\SocietyVtx{i}{b}$
  with the two-edge path $\SocietyVtx{i}{a} - \PlusVortexVtx{i} - \SocietyVtx{i}{b}$. 
\end{itemize}
Note that $J$ may contain multiple subpaths that are maximal subpaths contained
in $\VortexA{i}$; the projection defined as above shortcuts them all through a single two-edge path
via $\PlusVortexVtx{i}$. In particular, we fix the orientation of $J$ and process
vortices in the order of their appearance on $J$ only to make the process deterministic
(that is, the path projection of an oriented path is defined uniquely). 
Note that the orientation of $J$ does not matter if $J$ contains edges of at most one vortex.

We also extend the notion of a path projection to cycles in $G-\Apices$ that are edge-disjoint
with vortices in the natural manner.

We will need the following observation.
\begin{lemma}\label{lem:proj-2-paths}
Let $\nembed$ be a near-embedding of $G$.
Let $P$ and $Q$ be two paths in $G-\Apices$ that have all endpoints in $\MainA$
and do not contain any edges of vortices.
If $P$ and $Q$ are vertex-disjoint,
then $\PathProj_\nembed(P)$ and $\PathProj_\nembed(Q)$ are vertex-disjoint, too.
\end{lemma}
\begin{proof}
Note that if a vertex $\PlusDongleVtx{j}$ lies on $\PathProj_\nembed(P)$, then
$P$ contains at least two vertices of $V(\DongleA{j}) \cap V(\MainA)$, and similarly for $Q$.
Since every dongle has at most three vertices in common with $\MainA$, we infer that
no dongle of $\nembed$ contains both an edge of $P$ and an edge of $Q$. 
The claim follows.
\end{proof}

In some restricted cases, one can reverse the operation of a (path) projection as follows.
Assume that $J'$ is a subgraph of $\MainPlusA$ such that (a) $J'$ does not contain any virtual vortex edge and (b) if $J'$ contains a virtual dongle edge added due to a dongle $\DongleA{i}$, then $\DongleA{i}$ contains a connected subgraph $F$
with $V(F) \cap V(\MainA) = V(J') \cap V(\DongleA{i}) \cap V(\MainA)$. 
Then $G-\Apices$ contains a \emph{lift} of $J'$: a subgraph $J$ consisting of the vertices and edges of $J'$ that are in $\MainA$
plus, for every dongle $\DongleA{i}$ that contains an edge of $J'$, a minimal connected subgraph $F$ with $V(F) \cap V(\MainA) = V(J') \cap V(\DongleA{i}) \cap V(\MainA)$. 
Observe that $\SubProj_{\nembed}(J) = J'$. Furthermore, 
if $J'$ is a path with both endpoints in $V(\MainA)$, then any lift $J$ of $J'$ is a path between the same endpoints in $G-\Apices$ and $\PathProj_{\nembed}(J) = J'$. 

Apart from paths, we will be often projecting walls. 
Observe that if $W$ is a $w \times h$ wall for $w,h \geq 4$, then every separation $(A,B)$
of $W$ of order at most $3$ satisfies the following property: either $A \setminus B$
or $B \setminus A$ contains at most $1$ vertex of degree $3$ of $W$. 
With this observation, a direct check shows the following:
\begin{lemma}\label{lem:project-wall}
Let $\nembed$ be a near-embedding of a graph $G$ into a surface $\Sigma$.
Let $W$ be a $w \times h$ (cylindrical) 
wall or a $\Sigma$-wall of order $h$ in $G-\Apices$, where $w,h \geq 4$, that does not contain
any edge of a vortex. 
Then, either all but at most one degree-3 vertices of $W$ are contained in a single dongle
of $\nembed$, or the projection $\SubProj_{\nembed}(W)$ is also a $w \times h$ (cylindrical)
  wall or a $\Sigma$-wall of order $h$ in $\MainPlusA$, respectively.
\end{lemma}
\begin{lemma}\label{lem:lift-wall}
Let $\nembed$ be a near-embedding of $G$ into a surface $\Sigma$.
Let $W'$ be a $w \times h$ (cylindrical) wall 
in $\MainPlusA$ for some $w,h \geq 2$ 
or a $\Sigma$-wall of order $h$ for some $h \geq 2$. Suppose further that 
$W'$ does not contain any virtual vortex edge
and, for every dongle $\DongleA{i}$ such that $W'$ contains a virtual dongle edge added due to $\DongleA{i}$, $\DongleA{i}$ contains a connected subgraph $F$
with $V(F) \cap V(\MainA) = V(W') \cap V(\DongleA{i}) \cap V(\MainA)$. 
Then, any lift $W$ of $W'$ is a $w \times h$ (cylindrical) wall 
or a $\Sigma$-wall of order $h$, respectively, in $G-\Apices$. Also,
$\SubProj_{\nembed}(W) = W'$. 
\end{lemma}

We will also need the following observation.
\begin{lemma}\label{lem:small-subgraph-is-local}
Let $\nembed$ be a near-embedding of a graph $G$ and let $H$ be a connected subgraph
of $G-\Apices$. 
Furthermore, assume that $H$ contains edges of at most $\ell^\circ$ vortices of $\nembed$.
Then, for every $u,v \in V(H)$, the distance (and thus also the radial distance) in $\MainPlusA$ between $\SubProj_{\nembed}(u)$ and $\SubProj_{\nembed}(v)$ 
is at most $|V(H)| + \ell^\circ$.
\end{lemma}
\begin{proof}
Without loss of generality, since $H$ is connected, we can assume that $H$ is a path with endpoints $u$ and $v$.
If the whole $H$ is contained in $\DongleA{j} \setminus V(\MainPlusA)$ for some dongle $\DongleA{j}$ or in 
$\VortexA{i} \setminus V(\MainPlusA)$, then $\SubProj_{\nembed}(u) = \SubProj_{\nembed}(v)$ and we are done.
Otherwise, let $Q$ be a maximal subpath of $H$ with internal vertices not in $\MainPlusA$ and let $u'$ and $v'$ be the endpoints of $Q$; we have $u' \in V(\MainPlusA) \cup \{u\}$
and $v' \in V(\MainPlusA) \cup \{v\}$. 
Observe that either $Q$ is a single edge of $\MainPlusA$, or a path contained in a single dongle, or a path contained in a single vortex.
We define a path $Q'$ between $\SubProj_{\nembed}(u')$ and $\SubProj_{\nembed}(v')$ as follows.
If $Q$ is a single edge of $\MainPlusA$, we take $Q' = Q$.
If $Q$ is contained in a single dongle, we take $Q'$ to be the single edge of $\MainPlusA$ (one of the virtual edges for the said dongle) connecting $\SubProj_{\nembed}(u')$ and $\SubProj_{\nembed}(v')$
If $Q$ is contained in a single vortex, we take $Q'$ to be a one- or two-edge path of $\MainPlusA$ connecting $\SubProj_{\nembed}(u')$ and $\SubProj_{\nembed}(v')$ via the virtual edges
and possibly the virtual vertex of the said vortex. Note that in this case, $Q'$ is a one-edge path if and only if $u' = u \notin V(\MainPlusA)$ or $v' = v \notin V(\MainPlusA)$. 
By replacing every such path $Q$ with $Q'$, we obtain a walk in $\MainPlusA$ between $\SubProj_{\nembed}(u)$ and $\SubProj_{\nembed}(v)$. If we shortcut the walk at virtual vortex vertices
visited more than once, the length of the resulting walk is at most $|V(H)| + \ell^\circ$, as desired.
\end{proof}

\embedsection{Preimages of parts of the surface}

In embedded graphs, vertices on a closed face-vertex curve form a separator between the parts
of the embedding drawn on the two sides of the curve (assuming cutting along it breaks the surface into two). 
We need some definitions to discuss a similar process in near-embeddings.

\begin{definition}[preimage]
Let $\nembed$ be a near-embedding of $G$ and let $X \subseteq V(\MainPlusA)$.
A \emph{preimage of $X$} is a vertex subset $\preimage^\nembed(X) \subseteq V(G)$ consisting of:
\begin{itemize}
\item all vertices of $\MainA$ that belong to $X$;
\item for every vertex $\PlusDongleVtx{i}$ in $X$, all vertices of $\DongleA{i}$;
\item for every society vertex $\SocietyVtx{i}{j}$ in $X$, all vertices of $\VortexBag{i}{j}$. 
\end{itemize}
The {\em{preimage}} of a subset $\Psi \subseteq \Sigma$ is defined as the preimage of all vertices
that are embedded in $\Psi$.
\end{definition}
We often drop the superscript $\nembed$ if the near-embedding is clear from the context.
We have the following two immediate observations.
\begin{lemma}\label{lem:preimage-size}
Let $\nembed$ be a near-embedding of $G$ and let $X \subseteq V(\MainPlusA)$.
Assume $\nembed$ has maximum dongle size $\donglesize$ and maximum vortex bag size $\vortexbagsize$. 
Then $|\preimage(X)| \leq |X|\cdot \max(1, \donglesize, \vortexbagsize)$. 
\end{lemma}

\begin{lemma}\label{lem:preimage-cut}
Let $\nembed$ be a near-embedding of $G$ and let $\gamma$ be a closed face-vertex curve
in $\MainPlusA$ without self-intersections
such that $\Sigma-\gamma$ consists of two components $\Delta_1$ and $\Delta_2$.
Then $\preimage(\gamma)$ separates $\preimage(\Delta_1)$ and $\preimage(\Delta_2)$ in $G-\Apices$,
     that is, the following pair is a separation of $G$:
     \[ \left(\preimage(\Delta_1) \cup \preimage(\gamma) \cup \Apices, \preimage(\Delta_2) \cup \preimage(\gamma) \cup \Apices\right). \]
\end{lemma}

Let $\nembed$ be a near-embedding of $G$ and let $\gamma$ be a face-vertex curve in $\MainPlusA$ without self-intersections. 
Lemma~\ref{lem:preimage-cut} motivates the following surgery, which we call \emph{cutting along $\gamma$}.
\begin{enumerate}
\item Move $\preimage(\gamma)$ to the apex set. 
\item Cut $\Sigma$ along $\gamma$. If $\gamma$ is closed, glue one or two discs (depending on whether $\gamma$ is 1-sided or 2-sided) into the resulting holes in the surface. Otherwise, 
glue a disc into the resulting hole in the surface. 
If $\gamma$ is closed and $\Sigma-\gamma$ is connected (this mandates $\gamma$ noncontractible),
   the result is a connected surface of strictly smaller Euler genus than $\Sigma$.
If $\Sigma-\gamma$ is disconnected (this happens if $\gamma$ is contractible, but can also happen
    if $\gamma$ is 2-sided and noncontractible), then this results in two connected surfaces:
a sphere and a surface isomorphic to $\Sigma$ if $\gamma$ is contractible, or two surfaces
of strictly smaller Euler genus than $\Sigma$ if $\gamma$ is noncontractible. 
In the process of cutting along $\gamma$, we temporarily allow at this point for the embedding
to be in two surfaces, as in all applications (thanks to unbreakability) 
  the part embedded in one of the surfaces will be small and subsequently turned into a dongle.
\item For every vortex $\VortexA{i}$ such that $\PlusVortexVtx{i}$ lies on $\gamma$ and is 
not the first nor last vertex of $\MainPlusA$ on~$\gamma$, note that the two adjacent
vertices of $\MainPlusA$ on $\gamma$ are two society vertices $\SocietyVtx{i}{a}$ and $\SocietyVtx{i}{b}$ for some $1 \leq a < b \leq \vortexlength{i}$. Split $\VortexA{i}$ into two vortices, one containing bags of society vertices $\SocietyVtx{i}{a+1}, \ldots, \SocietyVtx{i}{b-1}$ and the other containing the remaining society vertices. The new vortices inherit as their discs the corresponding two parts of $\VortexDisc{i}\setminus \gamma$. The bags of the society vertices of the new vortices are inherited from the bags of the society vertices of $\VortexA{i}$, except all the vertices of $\preimage(\gamma)$ are removed. 
\end{enumerate}

The above surgery increases the number of vortices by at most the number of vertices $\PlusVortexVtx{i}$ on $\gamma$, and increases the size of the apex set by $|\preimage(\gamma)|$. 
We will not discuss here the result of cutting along $\gamma$ in details, as 
in what follows we will use this operation in a number of slightly different settings, such as:
\begin{itemize}
\item If $\gamma$ is closed and noncontractible,
it can be used to simplify the surface $\Sigma$ at the cost 
of a slight increase in the number of apices and vortices. 
\item If $\gamma$ is closed and separating,
it can be used in conjuction with the unbreakability assumption to argue that on one side
of $\gamma$ there is only a bounded number of vertices of $\MainPlusA$. 
\item If $\gamma$ connects two vertices $\PlusVortexVtx{i}$, it can be used to merge
two vortices at the cost of a slight increase in the number of apices.
\end{itemize}

\embedsection{Existence of a near-embedding}

We say that a near-embedding $\nembed$ \emph{captures} a tangle $\mathcal{T}$ if for each separation of $\mathcal{T}-\Apices$, the big side of this separation
is contained neither in any dongle $\DongleA{i}$ nor in any vortex $\VortexA{i}$.

The following statement is the main structural result of the Graph Minors series and appears
as (3.1) in~\cite{GM16}. The details of the statement follow Theorem~1 of~\cite{DiestelKMW12}.
The computability of $k_H$ and $\alpha_H$ follows from the simplified proof of~\cite{KawarabayashiW11,KTW20}.

\begin{theorem}\label{thm:GM16}
For every non-planar graph $H$ there exist integers $k_H,\alpha_H \geq 0$ such that
for every $H$-minor-free graph $G$ and every tangle $\mathcal{T}$ of $G$ of order at least $k_H$, 
$G$ admits a near-embedding capturing $\mathcal{T}$ where the number of apices, the Euler genus of the surface,
the number of vortices, and the maximum adhesion size of a vortex are all bounded by $\alpha_H$.

Furthermore, the values of $k_H$ and $\alpha_H$ are computable given $H$. 
\end{theorem}

\section{Optimal near-embeddings}

For a graph $H$, we define $\funchain(H) = \funcut(H) = 100(k_H + \alpha_H + 1)$, where $k_H$ and $\alpha_H$ are the constants provided by Theorem~\ref{thm:GM16}. 
For the sake of readability, we will not explicitly define $\funstep(x)$, but resort to stating that it is a sufficiently quickly growing computable function emerging from the proof.
That is, in a number of places we will say that a particular expression is bounded by some value of $\funstep$ if $\funstep$ is growing quickly enough; in all such places, it will be evident that the assumed lower bound is computable.

A typical example of such a usage is as follows. We are focusing on an index $\iota \leq \zeta$ and analysing an expression that is a computable function of some $q_i$s and $k_i$s for $i < \iota$.
Then, we say that --- by choosing $\funstep$ to be sufficiently quickly growing --- the expression in question is bounded by $k_\iota$ as well.

With the above in mind, fix a graph $H$, and fix $\zeta \coloneqq \funchain(H)$ and
$k_0 \coloneqq \funcut(H)$.
Let $G$ be an $H$-minor-free graph with an unbreakability chain $((q_i,k_i))_{i=0}^\zeta$ of length $\zeta$ and step $\funstep$, starting at $k_0$. 
If $|V(G)| \leq 3q_\zeta$, then the algorithm can return $\Vtw = V(G)$, $\Vgenus = \emptyset$, and $\famgenus$ to consist only of the empty function. Hence, we shall assume that $|V(G)| > 3q_\zeta$.
Thus, the unbreakability tangle $\mathcal{T}$ for the $(q_\zeta,k_\zeta)$-unbreakable graph $G$ is well-defined.

Fix $0 \leq \iota \leq \zeta/10$. 
A near-embedding $\nembed$ that captures $\mathcal{T}$ is \emph{of level $\iota$} if $\nembed$ has at most $k_{3\iota}/100$ apices, the maximum adhesion size of a vortex is at most $k_{3\iota}/100$,
and $k^\circ(\nembed) + 9\eulerg(\nembed) \leq 10\alpha_H - \iota$. 
Note that Theorem~\ref{thm:GM16} asserts that $G$ admits a near-embedding of level $0$, and there are no near-embeddings of level $\iota>10\alpha_H$.
Let $\iota_\ast$ be the maximum level for which there exists a near-embedding of $G$ of level $\iota_\ast$.
Furthermore:
\begin{itemize}
\item let $\eulerg_\ast$ be the minimum possible Euler genus of a near-embedding of $G$ of level $\iota_\ast$;
\item let $\napices_\ast$ be the minimum possible number of apices of a near-embedding of $G$ of level $\iota_\ast$ with Euler genus $\eulerg_\ast$;
\item let $\nvortices_\ast$ be the minimum possible number of vortices of a near-embedding of $G$ of level $\iota_\ast$ with $\napices_\ast$ apices and of Euler genus $\eulerg_\ast$;
\end{itemize}
A near-embedding of $G$ that captures $\mathcal{T}$, is of level $\iota_\ast$, has $\napices_\ast$ apices, $\nvortices_\ast$ vortices, and Euler genus $\eulerg_\ast$, is
called an \emph{optimal near-embedding of $G$}.
We denote $(q,k) \coloneqq (q_{3\iota_\ast}, k_{3\iota_\ast})$.
Note that $\iota_\ast \leq 10\alpha_H < \zeta/10$, so we can also denote $(q',k') \coloneqq (q_{3\iota_\ast+1}, k_{3\iota_\ast+1})$, $(q'',k'') \coloneqq (q_{3\iota_\ast+2}, k_{3\iota_\ast+2})$, 
     and $(q''',k''') \coloneqq (q_{3\iota_\ast+3}, k_{3\iota_\ast+3})$.
See Figure~\ref{tb:parameters} for the scale of these parameters and some selected quantities
from the proof on the scale.

\begin{figure}[tb]
\begin{center}
\begin{tabular}{c|l}
$q'''$ & lower bound on the treewidth of $G$ and $\MainPlusA$ \\\hline
\multirow{3}{*}{$k'''$} & lower bound on the facewidth of $\MainPlusA$ \\
& lower bound on the radial distance between two vortices \\
& minimum vortex length \\\hline
\multirow{2}{*}{$q''$} & upper bound on the treewidth of radial discs (\S\ref{ss:radial-discs})\\
& upper bound on the side of a cut along a radial path in a radial disc (\S\ref{ss:radial-discs})\\\hline
\multirow{2}{*}{$k''$} & $\rmax$: maximum allowed radius of a radial disc (\S\ref{ss:radial-discs})\\& upper bound on a cut size if one cuts along a $\mathrm{poly}(q')$-sized curve in $\MainPlusA$ \\\hline
\multirow{2}{*}{$\mathrm{poly}(q')$} & $\ProxCoverRadius$: radius of discs in a cover of local apices and vortices\\
& \qquad\quad in other optimal near-embeddings with minimal vortices (\S\ref{ss:proximity}) \\\hline
\multirow{1}{*}{$q'$} & $\AttCoverRadius$: radius of discs in a cover of attachment points of local apices (\S\ref{ss:universal-apex}) \\\hline
$k'$ & upper bound on a cut size if one cuts along a $\mathrm{poly}(q)$-sized curve in $\MainPlusA$ \\\hline
\multirow{2}{*}{$\mathrm{poly}(q)$} & $\WallHitSize$ and $\WallHitRadius$: upper bounds on locality of vortices and apices in a wall (\S\ref{ss:large-grids}) \\
& size of a wall in an apex forcer forcing a universal apex (\S\ref{ss:universal-apex}) \\\hline
\multirow{2}{*}{$q$} & maximum size of a dongle \\
& maximum size of a vortex bag \\\hline
\multirow{4}{*}{$\mathrm{poly}(k)$} & number of walls in an apex forcer forcing a universal apex (\S\ref{ss:universal-apex}) \\
& $\AttCoverSize$: number of discs in a cover of attachment points of local apices (\S\ref{ss:universal-apex}) \\
& $\ProxCoverSize$: number of discs in a cover of local apices and vortices\\
& \qquad\quad of other optimal near-embeddings with minimal vortices (\S\ref{ss:proximity}) \\\hline
\multirow{3}{*}{$k$} & maximum number of apices \\
& maximum number of vortices \\
& maximum adhesion size of a vortex 
\end{tabular}
\caption{The scale of parameters and selected quantities at the corresponding level.}\label{tb:parameters}
\end{center}
\end{figure}

We have the following observation.

\begin{lemma}\label{lem:opt-ne-rep}
Let $\nembed$ be an optimal near-embedding of $G$. Then, the following holds:
\begin{itemize}
\item Every dongle and every bag of every vortex has size at most $q$.
\item Every vortex has length at least $k'''/(100(k+1)) - k$.
\item In $\MainPlusA$, the radial distance between two distinct vertices $\PlusVortexVtx{i}$ is at least $k'''/100 - 2q - k$. 
\item The face-width of the embedding of $\MainPlusA$ is at least $k'''/100-2q-k$. 
\end{itemize}
\end{lemma}
\begin{proof}
For the first claim, note that for every bag $\VortexBag{i}{j}$, the separation 
$(\VortexBag{i}{j} \cup \Apices, N_{G}[V(G) \setminus (\VortexBag{i}{j} \cup \Apices)])$ is 
of order at most $3 \cdot k_{3\iota_\ast}/100 \leq k$ and thus $|\VortexBag{i}{j}| + |\Apices| \leq q$ by the $(q,k)$-unbreakability of $G$ and the assumption that $\nembed$ captures $\mathcal{T}$. 
Similarly, for every dongle $\DongleA{i}$, the separation $(V(\DongleA{i}) \cup \Apices, V(G) \setminus (V(\DongleA{i}) \setminus V(\MainA)))$ is of order at most $k_{3\iota_\ast}/100 + 3 \leq k$, 
  and thus $|V(\DongleA{i})| + |\Apices| \leq q$.

For the second claim, consider a vortex $\VortexA{i}$. If we move all society vertices and the vertices of the adhesions of $\VortexA{i}$ to the apex set,
the remainder of $\VortexA{i}$ is disconnected from the remainder of $\MainA$ and can be changed into a single dongle (decreasing the number of vortices of the near-embedding).
This increases the size of the apex set by at most $\vortexlength{i} \cdot (k+1)$. 
By the choice of $\iota_\ast$, the length of $\VortexA{i}$ is lower bounded as claimed, for otherwise the near-embedding obtained as above would have level larger than $\iota_\ast$.

For the third claim, let $\gamma$ be a face-vertex curve in $\MainPlusA$ that is a shortest curve connecting two vertices from the set $\{\PlusVortexVtx{i}~|~1 \leq i \leq \nvortices_\ast\}$;
say the endpoints are $\PlusVortexVtx{i}$ and $\PlusVortexVtx{j}$. For contradiction suppose $\gamma$ has length smaller than $k'''/100-2q-k$.
Note that the penultimate vertices on $\gamma$ are society vertices of $\VortexA{i}$ and $\VortexA{j}$, respectively, and the minimality of $\gamma$ ensures that no other vertex
of $\gamma$ belongs to a vortex.
Consider the operation of cutting along $\gamma$. 
By the first claim, this increases the size of the apex
set by at most $(k'''/100-2q-k) + 2q$.
Then, it is straightforward to merge $\VortexA{i}$ and $\VortexA{j}$ into a single vortex without increasing its maximum adhesion size. 
This contradicts the assumption that $\nembed$ is optimum, as the obtained near-embedding would be of higher level.

For the last claim, assume that $\gamma$ is a noncontractible noose in the embedding of $\MainA$
that visits at most $k'''/100-2q-k$ vertices of $\MainA$.
By the third claim, $\gamma$ visits society vertices and the vertex $\PlusVortexVtx{i}$ 
of at most one vortex. Furthermore, 
without loss of generality we can assume that if $\gamma$ visits some society vertices of a vortex, say $\VortexA{i}$, then it visits at most two of them, as all of them are incident with $\PlusVortexVtx{i}$. 
Consequently, we can cut along $\gamma$.
This moves at most $(k'''/100-2q-k) + 2q$ vertices to the apex set.
Since $\gamma$ is noncontractible, the result is one or two surfaces of strictly smaller
Euler genera than $\Sigma$. 
Furthermore, if there are two surfaces $\Sigma_1$ and $\Sigma_2$, then, due to $(q''',k''')$-unbreakability, for one $j \in \{1,2\}$ 
  the set of vertices of $G$ with their projection in $\Sigma_j$ is of size at most $q'''$;
  we put all of them into a new dongle and delete them from other parts
  of the embedding. 
  We end up with a near-embedding into a connected surface with a strictly smaller Euler
  genus and with the number of vortices increased by at most one, a contradiction
to the optimality of $\nembed$.
\end{proof}

Recall that we have assumed that $G$ has large treewidth (larger than any function of $\sum_{i=0}^\zeta q_i$ and $H$ fixed in the context). 
We now study the implications of this assumption on the treewidth of the 3-connected components of $G-\Apices$ for a small set $\Apices$,
and the treewidth of $\MainPlusA$ in a near-embedding of $G$.

We start with analyzing the 3-connected components of $G-\Apices$.
\begin{lemma}\label{lem:3conn-apex}
For every set $\Apices \subseteq V(G)$ of size at most $k-2$, there exists a unique 
3-connected component $G(\Apices)$ of $G-\Apices$ that has more than $\funtw(H,\sum_{i=0}^\zeta q_i)-|\Apices|$ vertices, while all the other $3$-connected components of $G-\Apices$ have less than $q-\Apices$ vertices.
Furthermore, the treewidth of $G(\Apices)$ is larger than $\funtw(H,\sum_{i=0}^\zeta q_i)$.
\end{lemma}
\begin{proof}
Recall that the unbreakability tangle $\mathcal{T}$ of order $k+1$ in $G$ is well-defined.
Hence, $G-\Apices$ is $(q-|\Apices|, k-|\Apices|)$-unbreakable and $\mathcal{T}-\Apices$ is a tangle of order $3$ in $G-\Apices$. 
Since $G-\Apices$ is $(q-|\Apices|, k-|\Apices|)$-unbreakable and $k-|\Apices| > 3$, there is at most one $3$-connected component of $G-\Apices$
that has more than $q-|\Apices|$ vertices. 
Since since the treewidth of $G$ is larger than $\funtw(H,\sum_{i=0}^\zeta q_i)$, Lemma~\ref{lem:tutte-tw} implies that there 
exists a $3$-connected component of $G-\Apices$ that has treewidth at least $\funtw(H,\sum_{i=0}^\zeta q_i)-|\Apices|$;
in particular, it has more than $\funtw(H,\sum_{i=0}^\zeta q_i)-|\Apices|$ vertices.
We infer that $G-\Apices$ contains a unique $3$-connected component
with more than $\funtw(H,\sum_{i=0}^\zeta q_i)-|\Apices|$ vertices, which we denote $G(\Apices)$,
while all the other $3$-connected components have less than $q-|\Apices|$ vertices.
As discussed, the treewidth bound of $G(\Apices)$ follows from Lemma~\ref{lem:tutte-tw}.
\end{proof}

Lemma~\ref{lem:3conn-apex} will be used predominantly in the setting where $\Apices$ is the apex set of an optimal near-embedding of $G$. In particular, it implies that then, every $3$-connected component of $G-\Apices$ except $G(\Apices)$ is very small: has at most $q-|\Apices|$ vertices. As this observation will be used multiple times in the sequel, we will often apply it implicitly, without an explicit reference to Lemma~\ref{lem:3conn-apex}.

To relate the treewidth of $\MainPlusA$ with the treewidth of $G$, we essentially
replicate Lemma~22 of~\cite{DiestelKMW12}.
Since a few side details are different in our setting, we repeat the adjusted proof.
We rely on the following lemma proved by Demaine and Hajiaghayi~\cite{DemaineH08}.
\begin{lemma}[Lemma 1 of~\cite{DemaineH08}]\label{lem:DH08}
Let $G$ be a graph embedded in a surface of Euler genus $g$. 
Let $G'$ be constructed from $G$ by contracting one face of $G$ into a single vertex.
Then, provided $\tw(G) \geq 56(g+1)^2$, we have $\tw(G') \geq \tw(G) / (g+1)$. 
\end{lemma}

\begin{lemma}[see also Lemma 22 of~\cite{DiestelKMW12}]\label{lem:DKMW12}
Let $\nembed$ be a near-embedding of $G$.
Let $\MainA'$  be the subgraph of $\MainPlusA$ consisting of
$\MainA$ and additionally,
for every dongle $\DongleA{i}$, the set $V(\MainA) \cap V(\DongleA{i})$ turned into a clique.
Then,
\[ \tw(G) \leq (56 \cdot (\tw(\MainA') + \nvortices) \cdot (\eulerg+1)^{\nvortices+2} + \nvortices) \cdot (1 + \vortexadhwidth) + \vortexbagsize + \donglesize + \napices. \]
\end{lemma}
\begin{proof}
Let $K$ be the set of virtual vertices $\PlusDongleVtx{i}$ for all $i \in [\ndongles]$ such that $|V(\MainA) \cap V(\DongleA{i})| = 3$. 

Consider the following process. Start with $H := \MainPlusA - \{\PlusVortexVtx{i}~|~1 \leq i \leq \nvortices\} - K$.
Then, iteratively for each $i=1,2,\ldots,\nvortices$, contract the face containing $\PlusVortexVtx{i}$ into a single vertex, and then delete the resulting vertex. 
(Recall that the vortices are vertex-disjoint, thus those contractions affect disjoint sets of vertices and in fact can be executed in an arbitrary order.)

Observe that the resulting graph $H'$ is a subgraph of $\MainA'$. 
Consequently, by Lemma~\ref{lem:DH08}:
\[ \tw(\MainPlusA - K) \leq 56 \cdot (\eulerg+1)^{\nvortices + 2} \cdot (\tw(\MainA')+ \nvortices) + \nvortices. \]
To obtain the promised bound on the treewidth of $G$, observe that one can turn a tree decomposition $(T',\beta')$ of $\MainPlusA-K$ into a tree decomposition of $G$
using the following operations:
\begin{itemize}
\item For every $1 \leq i \leq \nvortices$, for every $1 \leq j < \vortexlength{i}$,
  for every node $t \in V(T')$ with $\SocietyVtx{i}{j}, \PlusVortexVtx{i} \in \beta'(t)$,
add the adhesion $\VortexBag{i}{j} \cap \VortexBag{i}{j+1}$ to the bag $\beta'(t)$. This increases the width of the decomposition by at most a multiplicative factor of $(1+\vortexadhwidth)$.
\item For every $1 \leq i \leq \nvortices$, for every $1 \leq j \leq \vortexlength{i}$,
  fix a node $t \in V(T')$ whose bag contains all three vertices $\PlusVortexVtx{i}, \SocietyVtx{i}{j}, \SocietyVtx{i}{j+1}$ (with indices modulo $\vortexlength{i}$)
and  create a new adjacent node $t'$ with bag $\beta'(t') = \VortexBag{i}{j}$. This increases the width of the decomposition by at most an additive factor of $\vortexbagsize$.
\item For every $1 \leq i \leq \ndongles$, fix a node $t \in V(T')$ whose bag contains $V(\MainA) \cap V(\DongleA{i})$ and create a new adjacent node $t'$ with bag $\beta'(t') = V(\DongleA{i})$. 
This increases the width of the decomposition by at most an additive factor of $\donglesize$.
\item Add all apices to every bag of the decomposition. This increases the width of the decomposition by at most an additive factor of $\napices$.
\end{itemize}
\end{proof}

Lemma~\ref{lem:DKMW12} and the assumption on the treewidth of $G$
allows us to assume that in every optimal near-embedding of $G$ the treewidth of $\MainPlusA$ is larger than any fixed computable function of $H$ and $\sum_{i=0}^\zeta q_i$.

\section{Canonizing dongles}\label{ss:tidy-dongles}

In this section we present a key statement that intuitively says the following: once the apices and vortices are selected, there is a canonical way to choose the dongles in ``the best possible manner''.

We start with a simple clean-up of the connectivity of dongles.
For a dongle $G_j^\triangle$, a \emph{component} of a dongle is subgraph of $G_j^\triangle$ that is either:
\begin{itemize}
\item an edge connecting two vertices of $V(G_0) \cap V(G_j^\triangle)$, or
\item a connected component $C$ of $G_j^\triangle -V(G_0)$, together with $N_{G-\Apices}(C)$ and all edges connecting $C$ with $N_{G-\Apices}(C)$. 
\end{itemize}
Let $\DongleComps(G_j^\triangle)$ be the family comprising all components of the dongle $G_j^\triangle$. Note that the components of $G_j^\triangle$
form a partition of $G_j^\triangle$ into edge-disjoint subgraphs that intersect only in subsets of $V(G_0) \cap V(G_j^\triangle)$.

\begin{definition}
Fix a near-embedding $\nembed$ of a graph $G$.
A dongle $G_j^\triangle$ is \emph{properly connected} if for every $D \in \DongleComps(G_j^\triangle)$ we have $V(D) \cap V(G_0) = V(G_j^\triangle) \cap V(G_0)$. 
A near-embedding has \emph{properly connected dongles} if every dongle is properly connected.
\end{definition}
\begin{lemma}\label{lem:split-dongle}
Given a near-embedding $\nembed$ of a graph $G$ and a dongle $G_j^\triangle$ of $\nembed$,
one can in polynomial time compute a near-embedding $\nembed'$ of $G$ that differs from $\nembed$ only in replacing $G_j^\triangle$ with possibly more than one new dongles, 
so that every newly created dongle is properly connected.
Furthermore, if $\nembed$ is optimal, so is $\nembed'$.
\end{lemma}
\begin{proof}
Group components $D \in \DongleComps(G_j^\triangle)$ by $V(D) \cap V(G_0)$ (which are subsets of $V(G_0) \cap V(G_j^\triangle)$).
For every $A \subseteq V(G_0) \cap V(G_j^\triangle)$, if there exists a component $D \in \DongleComps(G_j^\triangle)$ with $V(D) \cap V(G_0) = A$,
group all such components into a new dongle.
Observe that it is straightforward to draw the discs for the new dongles inside the former disc of $G_j^\triangle$ so that those discs have disjoint interiors.

Finally, observe that the above operation does not change any of the invariants taken into account in the definition of an optimal near-embedding.
\end{proof}
Note that the procedure presented in the proof of Lemma~\ref{lem:split-dongle} does not modify the embedding if the dongle in question is already properly connected.

By applying Lemma~\ref{lem:split-dongle} to every dongle of a near-embedding, we obtain a near-embedding with properly connected dongles that differs from the original one
only on dongles and every dongle of the new near-embedding is a subgraph of a dongle in the old near-embedding.
Thus, we may henceforth study optimal near-embeddings that additionally have properly connected dongles.

Consider now a set $\Apices \subseteq V(G)$ of size exactly $\napices_\ast$ and consider all optimal near-embeddings of $G$ 
with apex set $\Apices$.
Intuitively, Lemma~\ref{lem:3conn-apex} says that the majority of complexity of $G$ stays in one 3-connected component of $G-\Apices$, namely $G(\Apices)$. 

We start by aligning an optimal near-embedding with 3-connected components of the graph. 
\begin{definition}
Consider a near-embedding $\nembed$ of $G$ with apex set $\Apices$.
For a connected component $C$ of $G-\Apices-V(G(\Apices))$, let the subgraph of $G-\Apices$ consisting of $G[C]$, $N_{G-\Apices}(C)$ and all edges between $C$ and $N_{G-\Apices}(C)$ be denoted by $C^\ast$.
We say that $C$ is \emph{aligned in $\nembed$} if either:
\begin{itemize}
\item $C^\ast$ is contained in a single dongle $G_i^\triangle$ and, furthermore, $C \subseteq V(G_i^\triangle) \setminus V(G_0)$; or
\item $C^\ast$ is contained in a single vortex $G_i^\circ$ and, furthermore, $C \subseteq V(G_i^\circ) \setminus V(G_0)$. 
\end{itemize}
We say that $\nembed$ is \emph{aligned with 3-connected components} if every connected component $C$ of $G-\Apices-V(G(\Apices))$ is aligned in $\nembed$.
\end{definition}

\begin{lemma}\label{lem:align-3conn}
For every set $\Apices$, if there exists an optimal near-embedding with apex set $\Apices$, then there exists an optimal near-embedding with apex set $\Apices$
that is aligned with 3-connected components and has properly connected dongles. 
\end{lemma}
\begin{proof}
Let $\nembed$ be an optimal near-embedding with apex set $\Apices$ that minimizes the number of connected components $C$ of $G-\Apices-V(G(\Apices))$ that are not aligned in $\nembed$.

Consider applying Lemma~\ref{lem:split-dongle} to a dongle $G_j^\triangle$ in $\nembed$, obtaining a near-embedding $\nembed'$. 
Every connected component $C$ of $G-\Apices-V(G(\Apices))$ that is aligned in $G$ because $C^\ast$ is contained in $G_j^\triangle$ and $C \subseteq V(G_j^\triangle) \setminus V(G_0)$
remains aligned in $\nembed'$, because every connected component of $G_j^\triangle - V(G_0)$ remains within one dongle of $\nembed'$ and $C$ is connected. 
Clearly, for every other connected component $C$ of $G-\Apices-V(G(\Apices))$, the application of Lemma~\ref{lem:split-dongle} does not affect whether $C$ is aligned, that is, $C$ is aligned in $\nembed$
if and only if $C$ is aligned in $\nembed'$. 
Consequently, by applying Lemma~\ref{lem:split-dongle} to every dongle of $\nembed$, we can assume that $\nembed$ has properly connected dongles.

By contradiction, assume that there is a connected component $C$ of $G-\Apices-V(G(\Apices))$ that is not aligned in $\nembed$.

If $N_{G-\Apices}(C) = \emptyset$, then we can modify $\nembed$ by creating a new dongle equal to $G[C]$ (and removing vertices and edges of $G[C]$ from other parts of the near-embedding), contradicting
the choice of $\nembed$. The disc of the new dongle can be placed arbitrarily, provided it is disjoint with the embedding of $\MainA$ and other discs for vortices and dongles. 

If $N_{G-\Apices}(C) = \{v\}$, then we consider where the vertex $v$ lies. If $v \in V(G_0)$, then we modify $\nembed$ by creating a new dongle equal to $C^\ast$ and its disc defined so that $v$
is embedded on its boundary. 
If $v \in V(G_i^\circ)\setminus V(G_0)$ or $v \in V(G_j^\triangle)\setminus V(G_0)$, we can move the whole $G[N_G[C]]$ to $G_i^\circ$ or $G_j^\triangle$, respectively (removing also all vertices of $C$ and edges of $C^\ast$
    from other parts of the near-embedding). In the case $v \in V(G_i^\circ)$, $C$ is also added to one of the bags of the path decomposition of $G_i^\circ$ that contains $v$;
note that this step does not increase the maximum size of a vortex adhesion, as after it the vertices of $C$ are present only in one bag of the path decomposition of $G_i^\circ$.
In all cases, the modification contradicts the choice of $\nembed$.

In the final case, $N_{G-\Apices}(C) = \{v_1,v_2\}$.
Similarly as in the previous case, if $v_1,v_2$ are contained in a single dongle $G_j^\triangle$, we can move $C^\ast$ to $G_j^\triangle$, removing all its edges and vertices except for $\{v_1,v_2\}$
from other parts of the embedding. Henceforth we assume that this is not the case.

We would like to ensure now that $v_1$ and $v_2$ are not ``hidden'' in the interior of any dongle. 
Assume then that $v_1 \in V(G_j^\triangle) \setminus V(G_0)$ for some dongle $G_j^\triangle$. By the previous paragraph, $v_2 \notin V(G_j^\triangle)$. 
Since $C^\ast$ is connected, at least one vertex of $V(G_j^\triangle) \cap V(G_0)$ is contained in~$C^\ast$. 
Furthermore, not all vertices of $V(G_j^\triangle) \cap V(G_0)$ are contained in~$C^\ast$, as otherwise the 3-connectivity of $G(\Apices)$ would imply
$V(G(\Apices)) \subseteq V(G_j^\triangle)$, but this is a contradiction with $|V(G(\Apices))|>q>|V(G_j^\triangle)|$.
If exactly one vertex, say $u$, of $V(G_j^\triangle) \cap V(G_0)$ does not belong to $C^\ast$, then $(V(C^\ast)\cup V(G_j^\triangle),V(G)\setminus (\Apices\cup V(C)\cup (V(G_j^\triangle)\setminus V(G_0))))$ is a separation of order $2$ in $G-\Apices$ with $\{u,v_2\}$ being the separator. Since $v_1\in V(G(\Apices))$, it would follow that $V(G(\Apices))\subseteq V(G_j^\triangle)$, a contradiction to the fact that $|V(G(\Apices))|>q>|V(G_j^\triangle)|$.

Hence $V(G_j^\triangle) \cap V(G_0)$ consists of exactly three vertices: two of them, say $u_1$ and $u_2$, lie outside of $C^\ast$
and one, say $u_3$, lies in $C$.
Since $u_3 \in C$ and $v_2 \notin V(G_j^\triangle)$ and $\nembed$ has properly connected dongles,
 we infer that there exists exactly one connected component of $G_j^\triangle \setminus V(G_0)$ --- the one containing
$v_1$ --- for otherwise $N_{G-\Apices}(C)$ would necessarily contain more vertices than just $v_1$ and $v_2$. We modify the embedding $\nembed$ as follows: we add $v_1$ to $G_0$ and replace $G_j^\triangle$ with two dongles, one containing $G_j^\triangle \cap C^\ast$ and the other $G_j^\triangle \setminus (V(C^\ast) \setminus \{v_1\})$. Note that the first dongle contains three vertices of $G_0$ ($u_1$, $u_2$, and $v_1$) and the second contains two vertices of $G_0$ ($u_3$ and $v_1$) and it is straightforward
to split the disc for $G_j^\triangle$ into two discs for the two new dongles, placing $v_1$ inside the disc for $G_j^\triangle$. 

This concludes the description of the first step; we perform it also for $v_2$ if necessary, and then apply Lemma~\ref{lem:split-dongle} to every dongle that is not properly connected.
Hence, we assume that $\nembed$ has properly connected dongles and that each of $v_1$ and $v_2$ is either in $G_0$ or in a vortex. 
The original choice of $\nembed$ implies that $C$ is still not aligned in $\nembed$.

The first step implies the following property.
\begin{claim}\label{cl:Cast-dongles}
Every dongle $G_j^\triangle$ is either edge-disjoint with $C^\ast$ or contained in $C^\ast$.
\end{claim}
\begin{proof}
Assume $G_j^\triangle$ is not edge-disjoint with $C^\ast$ and let $D \in \DongleComps(G_j^\triangle)$ be such that $D$ contains an edge of $C^\ast$.
Then, as neither of $v_1$ and $v_2$ can belong to $V(D) \setminus V(G_0)$, we infer that $D$ is a subgraph of $C^\ast$. 
Thus, every component of $G_j^\triangle$ is either edge-disjoint with $C^\ast$ or contained in $C^\ast$.

Assume additionally that, apart from $D \in \DongleComps(G_j^\triangle)$ that is contained in $C^\ast$, 
there is another $D' \in \DongleComps(G_j^\triangle)$ that is edge-disjoint with $C^\ast$. 
Then, in $G-\Apices$, the $D'$ and $D$ are separated by $\{v_1,v_2\}$. Since $G_j^\triangle$ is properly connected, $V(D') \cap V(D) = V(G_j^\triangle) \cap V(G_0) \subseteq \{v_1,v_2\}$.
As $C$ is connected, we infer that $D = C^\ast$; in particular $C$ is aligned with $\nembed$, a contradiction.
\cqed\end{proof}

Assume now that there exists a path $Q$ in $C^\ast$ that connects $v_1$ and $v_2$ and does not contain any edge of a vortex (but may contain vertices of a vortex that are society vertices). 
Then we proceed as follows: create a new dongle containing $C^\ast$ and remove all edges of $C^\ast$ and vertices of $C$ from other parts of the near-embedding.
The disc associated with the newly created dongle is a small open neighborhood of the interior of the drawing of $Q$. Observe here that $Q$ might traverse dongles contained in $C^\ast$; then we draw the corresponding part of $Q$ within the the interior of the disc of the said dongle, and observe that this dongle is then anyway removed from the near-embedding. Recall here that, by Claim~\ref{cl:Cast-dongles}, every dongle is either contained in $C^\ast$ and removed by the process above, or edge-disjoint with $C^\ast$ and left intact. We have thus created a new near-embedding in which $C$ is aligned, and every component of $G-\Apices-V(G(\Apices))$ that was previously aligned stays aligned; this contradicts the choice of $\nembed$.

We are left with the case when no such path $Q$ exists. In particular, there exists a vortex $G_i^\circ$ that contains an edge of $C^\ast$. Since $C^\ast$ is connected and of size at most $q$,
by Lemma~\ref{lem:opt-ne-rep} $C^\ast$ does not contain a vertex or an edge of any other vortex. Intuitively, we would like now $G_i^\circ$ to ``swallow''~$C^\ast$, but the precise details of
this operation are delicate. 

Consider a path $P$ in $C^\ast$ that contains no edge of $G_i^\circ$, but whose endpoints are two society vertices of $G_i^\circ$, say $v_{i,a}^\circ$ and $v_{i,b}^\circ$, $a < b$.  
Let $\gamma_P$ be a closed curve defined as the drawing of $\PathProj_\nembed(P)$, 
    closed via a curve inside the disc for $\VortexA{i}$. 
If $\gamma_P$ is not contractible, then by adding $V(C^\ast) \cup \VortexBag{i}{a} \cup \VortexBag{i}{b}$ to apices and splitting $G_i^\circ$ into two vortices
(one with bags $\VortexBag{i}{a+1},\ldots,\VortexBag{i}{b-1}$ and the second with the remaining bags) one obtains a near-embedding of level larger than $\iota_\ast$,
a contradiction to the maximality of $\iota_\ast$.
Thus, $\gamma_P$ is contractible. In particular, $Y := \Apices \cup V(C^\ast) \cup \VortexBag{i}{a} \cup \VortexBag{i}{b}$ separates in $G$
parts embedded in different sides of $\gamma_P$.
Here, we include a vortex bag on the side where its society vertex is embedded and a dongle on the side where its disc lies; note that Claim~\ref{cl:Cast-dongles} and the fact
that dongles are properly connected
implies that the only dongles whose discs do not lie on one side of $\gamma_P$ are dongles contained in $C^\ast$. 
By unbreakability, as $|Y| < k'$, one of the sides has at most $q'$ vertices; let us call it \emph{the small side of $\gamma_P$} and let $Z_P$ be the set of vertices that are not in $Y$
but are embedded in the small side (again, including the vortex bags of all society vertices lying the small side, as well as all dongles whose discs are contained in the small side).
By the optimality of $\nembed$, moving $Y$ and $Z_P$ to $\Apices$ cannot decrease the Euler genus, hence
the small side is embedded in a disc side of $\gamma_P$; call this disc side $\Delta_P$. With $P$ we associate the cyclic interval $I_P = [a,b]$ if the society vertices $\SocietyVtx{i}{a}, \SocietyVtx{i}{a+1}, \ldots, \SocietyVtx{i}{b}$
are embedded in $\Delta_P$, and $I_P = [b,a]$ otherwise. Note that $|I_P|\leq |Z_P\cup Y|\leq q'$. Since the length of vortex $\VortexA{i}$ is larger than $q'$ (see Lemma~\ref{lem:opt-ne-rep}), it follows that $I_P$ is a strict subset of all the indices of the society vertices of $\VortexA{i}$. 

Let $\mathcal{P}$ be the family of all paths $P$ as above. We can choose the way curves $\gamma_P$ are realized within the disc of vortex $\VortexA{i}$ so that $I_P\subseteq I_{P'}$ 
 entails $\Delta_P \cap \VortexDisc{i} \subseteq \Delta_{P'} \cap \VortexDisc{i}$.
We need the following observation about the interaction between different intervals $I_P$ and corresponding discs $\Delta_P$.

\begin{claim}\label{cl:union-intervals}
 Suppose ${\cal Q}\subseteq {\cal P}$ is such that $\bigcup_{P\in {\cal Q}} I_P$ is a cyclic interval $J=[a,b]$, and $\cal Q$ contains all paths $P\in \cal P$ for which $I_P\subseteq J$. Then there exists $Q\in \cal Q$ such that $I_Q=J$ and $\Delta_Q\supseteq \Delta_{P}$ for each $P\in \cal Q$.
\end{claim}
\begin{proof}
We first show that there exists $P \in \mathcal{Q}$ with $I_P = J$. Assume the contrary; let $P \in \mathcal{Q}$ be such that $I_P$ is inclusion-wise and assume $I_P \neq J$.
By symmetry, let $I_P = [a',b']$ where $b' \neq b$. 
By the maximality of $I_P$, there exists $Q \in \mathcal{Q}$ such that $I_Q = [a'',b'']$ where $a'' \in [a',b']\setminus \{a'\}$ and $b'' \in [b',b]\setminus \{b'\}$. Now both $\gamma_P$ and $\gamma_Q$ bound disks --- $\Delta_P$ and $\Delta_Q$ respectively --- and they meet once within the disc of vortex $\VortexA{i}$. It follows that $\gamma_P$ and $\gamma_Q$ must intersect again outside of the disc of $\VortexA{i}$, so in particular $P$ and $Q$ intersect. As $P$ and $Q$ are edge-disjoint with vortices, we conclude that  $P\cup Q$ contains a path $R$ that is edge-disjoint with vortices and connects $\SocietyVtx{i}{a'}$ with $\SocietyVtx{i}{b''}$.
We have $I_R = [a',b'']$ and $R \in \mathcal{Q}$, so this is a contradiction to the choice of $P$.

Let $\mathcal{Q}' \subseteq \mathcal{Q}$ be the family of those paths $P$ for which $I_P = J$; we have already proven $\mathcal{Q}' \neq \emptyset$. 
By our assumption on the paths inside $\VortexDisc{i}$, all $P \in \mathcal{Q}'$ share the same set $\Delta_P \cap \VortexDisc{i}$. 
Our goal is to show that $\{\Delta_P~\colon~P \in \mathcal{Q}'\}$ contains only one inclusion-wise element (it may appear for a number of distinct paths $P$). 
Assume this is not the case, so let $P,Q \in \mathcal{Q}'$ be two elements with distinct and inclusion-wise maximal $\Delta_P$ and $\Delta_Q$. 

Since both $\Delta_P$ and $\Delta_Q$ are discs with common intersection with $\VortexDisc{i}$, we have that $\gamma_P$ and $\gamma_Q$ intersect. 
Let $\gamma_Q'$ be a maximal subcurve of $\gamma_Q$ that is outside $\Delta_P$ and let $\gamma_P'$ be the subcurve of $\gamma_P$ between the endpoints
of $\gamma_Q'$. By the definition of $\gamma_P$ and $\gamma_Q$, 
$\gamma_P'$ and $\gamma_Q'$ corresponds to a subpaths $P''$ of $\PathProj_\nembed(P)$ and $Q''$ of $\PathProj_{\nembed}(Q)$, respectively, between the same endpoints
in the intersection of $\PathProj_\nembed(P)$ and $\PathProj_\nembed(Q)$. 
Since both $P$ and $Q$ do not contain edges of vortices and have endpoints in $V(\MainA)$, $P''$ and $Q''$ correspond to subpath $P'$ and $Q'$ of $P$ and $Q$ respectively, between the same endpoints,
with the projections of endpoints being the endpoints of $P''$ and $Q''$, respectively. 

Consider now a path $R$ created from $P$ by replacing $P'$ with $Q'$. Clearly, we have $R \in \mathcal{Q}$ and $I_R = J$, so $R \in \mathcal{Q}'$.
Also, by construction $\gamma_R$ equals $\gamma_P$ with $\gamma_P'$ replaced with $\gamma_Q'$. Since $\gamma_Q'$ is disjoint with $\Delta_P$ except for the endpoints, we have
$\Delta_P \subseteq \Delta_R$. Since the closure of $\Delta_R$ contains internal points of $\gamma_Q'$, while the closure of $\Delta_P$ does not contain them, we
have $\Delta_P \neq \Delta_R$. This contradiction with the choice of $P$ finishes the proof of the claim.
\cqed\end{proof}

Let $J_1,\ldots,J_\ell$ be the sequence of consecutive disjoint intervals of $\bigcup_{P \in \mathcal{P}} I_P$
and intervals of the form $[a,a]$ for vertices $\SocietyVtx{i}{a}$ that are in $C^\ast$, but $a$ is in none of the segments $I_P$. By Claim~\ref{cl:union-intervals}, for each interval $J_j$ of size larger than $1$ there exists $Q_j\in {\cal P}$ such that $J_j=I_{Q_j}$ and the disc $\Delta_{Q_j}$ contains all discs $\Delta_P$, for $P\in \cal P$ such that $I_P\subseteq I_j$. Then $|J_j|\leq |I_{Q_j}|\leq q'$, hence each of the intervals $J_1,\ldots,J_\ell$ contains at most $q'$ indices. In particular, each $J_j$ is a strict subset of all the indices of all the society vertices of $\VortexA{i}$.

Consider one interval $J_j = [a,b]$, $1 \leq j \leq \ell$. Let $C_j$ be the subgraph of $C^\ast$ consisting of all edges and vertices that can be reached by a path from a vertex $\SocietyVtx{i}{\alpha}$, for some $\alpha \in J_j$, consisting only of edges in $E(C^\ast)\setminus E(\VortexA{i})$.

We perform the following face-vertex walk in $\MainPlusA$. We start from $w_1 := \SocietyVtx{i}{a}$ and then scan the edges of $\MainPlusA$ incident with $w_1$, starting
from the virtual edge $\SocietyVtx{i}{a}\SocietyVtx{i}{a-1}$, in the direction going outside of the disc of $\VortexA{i}$ until an edge of the projection of $C^\ast$ (i.e., of $C^\ast$ or a virtual edge of a dongle contained in $C^\ast$) is found. We move to $f_1$ being the face just before the found edge $e_1$ of the projection of $C^\ast$. Then, we start from $w_1$ and $e_1$ and continue along the facial
walk of $f_1$ until an edge $e_1''$ that is not in the projection of $C^\ast$ is encountered. Let $e_1'$ be the preceding edge on the said facial walk and $w_2$ be the vertex between $e_1'$ and $e_1''$.
We move to $w_2$ and scan the edges incident with $w_2$, starting from $f_1$ to $e_1''$, until an edge $e_2$ of the projection of $C^\ast$ is found. We move to the face $f_2$ that is just before $e_2$
in this scan and continue along the facial walk of $f_2$, starting from $w_2$ and $e_2$, in the same manner as we did with $w_1$, $f_1$, and $e_1$. 
The walk terminates when, while scanning faces along a vertex $w_\iota$, we enter a disc of a vortex (i.e., we cross a virtual edge of a vortex). 

Let us discuss a few properties of this process. First, the next steps are always defined: while scanning edges incident with a vertex $w_\iota$, the edge $e_{\iota-1}'$ is one of the candidates for $e_\iota$, so $e_\iota$ will eventually be found. Similarly, while going along the facial walk of $f_\iota$, eventually we will reach an edge not in the projection of $C^\ast$, as
the edge preceting $e_\iota$ in the scan of $w_\iota$ is one such edge (note that this is also true for $\iota=1$). 
Second, this process is reversible: knowing $e_\iota$, $w_\iota$, and $f_\iota$, the tuple $(e_{\iota-1},w_{\iota-1},f_{\iota-1})$ is uniquely defined (except for the case $\iota=1$).
In particular, this implies that the process terminates (does not loop). 
Third, all edges of the projection of $C^\ast$ seen in the process while traversing facial walks of faces $f_\iota$ are actually in projections of $C_j$, as the concatenation of the lifts
of the said traversed facial walks witness their membership in $C_j$. 
Fourth, since the process visits only vertices of the projection of $C^\ast$, it cannot reach a society vertex of a vortex different than $\VortexA{i}$. 

Fifth, for every $\iota > 1$, the edges incident with $w_\iota$ scanned between $f_{\iota-1}$ and $f_\iota$ are not in the projection of $C^\ast$.
This has a few consequences. The process cannot cross
any path $P \in \mathcal{P}$ and hence cannot go inside the disc $\Delta_{Q_j}$, hence it does not visit vertices $\SocietyVtx{i}{\alpha}$ for $a < \alpha < b$. 
Since every visited vertex $w_\iota$ has an incident edge in the projection of $C_j$ and not in the projection of $C_j$, we infer that every $w_\iota$ is actually one
of the vertices $\SocietyVtx{i}{a}$, $\SocietyVtx{i}{b}$, $v_1$, and $v_2$. 
Furthermore, while the defined face-vertex walk can visit the same vertex multiple times, it does not transversally cross itself.
That is, if $w := w_{\iota_1} = w_{\iota_2} = \ldots = w_{\iota_\eta}$
for some $\iota_1 < \iota_2 < \ldots \iota_\eta$, then the face-vertex incidences
where we entered from $f_{\iota_1-1}$ to $w_{\iota_1}$, left to $f_{\iota_1}$, entered from $f_{\iota_2-1}$ to $w_{\iota_2}$,
\ldots, left $w_{\iota_\eta}$ to $f_{\iota_\eta}$ lie in this order along $w$.
Similarly, since we follow a facial walk along projection of $C$ until we encounter an edge not in this projection, 
if $f : f_{\iota_1} = f_{\iota_2} = \ldots = f_{\iota_\eta}$ for some $\iota_1 < \ldots < \iota_\eta$, then the face-vertex incidences
where we entered $f$ from $w_{\iota_1}$, left to $w_{\iota_1+1}$, \ldots, left $f$ to $w_{\iota_\eta+1}$ lie in this order in the facial walk along $f$. 

We now argue that the connectivity of $C$ implies that this process actually does not visit the same vertex twice, except for the special case $a=b$
and $\SocietyVtx{i}{a} = \SocietyVtx{i}{b}$ is visited twice, as the first and last vertex of the walk.
Assume the contrary, let $w := w_\iota = w_{\iota'}$ for some $\iota < \iota'$ and pick such that $\iota'-\iota$ is minimum possible. 
We have $w \in \{\SocietyVtx{i}{a}, \SocietyVtx{i}{b}, v_1,v_2\}$. 
By the excluded special case, we have that either $w_\iota$ is not the first vertex of the walk or $w_{\iota'}$ is not the last vertex of the walk. 
In both cases, the process visited $w$ as neither the first nor the last vertex of the walk and hence there is an edge of $\MainPlusA$ incident with $w$
that is neither in $C^\ast$ nor a virtual edge of $\VortexA{i}$. We infer that actually $w \in \{v_1,v_2\}$. 
By minimality of $\iota'-\iota$, the walk between $w_\iota$ and $w_{\iota'}$ defines a closed face-vertex curve in $\MainPlusA$ without self-intersection $\gamma'$
that visits only a subset of $\{\SocietyVtx{i}{a},\SocietyVtx{i}{b},v_1,v_2\}$. By the same argumentation as for discs $\Delta_P$, $P \in \mathcal{P}$, 
$\gamma'$ is contractible and encloses a disc with projections of at most $q'$ vertices of $G$. This disc cannot contain $\VortexA{j}$, so it is in the same side
of $\gamma'$ as edges $e_{\iota}$ and $e_{\iota'-1}'$. However, that implies that $w$ is a cutvertex of $C^\ast$, a contradiction to the definition of $C^\ast$. 

We infer that the defined walk defines a face-vertex curve in $\MainPlusA$ from $\SocietyVtx{i}{a}$ to $\SocietyVtx{i}{b}$ that may visit $v_1$ or $v_2$ along the way.
This curve is closed if $a=b$; if $a \neq b$, we close this curve inside the disc of $\VortexA{i}$. Let $\gamma_j$ be the final closed curve. 
Using the same argumentation as for discs $\Delta_P$ for $P \in \mathcal{P}$, we infer that $\gamma_j$ is contractibe and bounds a disc $\Delta_j$ that contains the projections
of $C_j$ and in total at most $q'$ vertices of $G$. Furthermore, we can assume that $\gamma_j$ is closed via the disc of $\VortexA{i}$ in such a manner that $\Delta_j$ contains $\Delta_{Q_j}$.

Let $Z_j$ be the vertices of $G-\Apices$ whose projection lies in $\Delta_j$. Recall that $|Z_j| \leq q'$. Let $H_j = G[Z_j]$. 
Let $\mathbf{Z} = \bigcup_{j} Z_j$. 
We would like to move the subgraph induced by $\mathbf{Z} \cup V(C^\ast)$ to the vortex $\VortexA{i}$, but there are two issues that we need to take care of: (i) if any of $v_1$ or $v_2$ is not currently
in $\VortexA{i}$, it needs to become a society vertex, and (ii) we need to adjust the path decomposition to accommodate the new vertices without increasing the maximum adhesion size. 

Summarizing the construction, the disc $\Delta_j$ contains:
\begin{itemize}[nosep]
 \item all society vertices $\SocietyVtx{i}{\alpha}$, $\alpha \in J_j$;
 \item all edges and vertices of $H_j$ that are in $\MainA$ (possibly with $v_1$, $v_2$, $\SocietyVtx{i}{a}$, and $\SocietyVtx{i}{b}$ on the boundary of $\Delta_j$), and
 \item all discs of all dongles contained in $H_j$.
\end{itemize}
Except for the above, $\Delta_j$ is disjoint with $\MainA$ and with all discs of dongles and vortices of $\nembed$. 

Since the intervals $J_j$ are pairwise disjoint, the same argument as the one used in the proof of Claim~\ref{cl:union-intervals} shows that the discs $\Delta_j$ are also pairwise disjoint. This implies that the union of the disc of vortex $\VortexA{i}$ and all the discs $\Delta_j$, for $1 \leq j \leq \ell$, is a disc as well. Call this disc $\Delta$.

Fix an index $j$, $1 \leq j \leq \ell$. Recall that $|V(H_j)|\leq q'$. 
Moreover, if $v_1$ is in $C_j$ and is incident to an edge not in $H_j$, then $v_1$ is necessarily embedded on the boundary of $\Delta_j$. Same applies to $v_2$. 
Finally, we observe that it is impossible for both $v_1$ and $v_2$ to be on the boundary of $\Delta_j$, as this would imply there are connected by a path consisting only of edges in $E(C_j)$, contradicting the assumption of the current case. 

We modify the near-embedding $\nembed$ as follows.
For the disc of the modified vortex $\VortexA{i}$ we take~$\Delta$. 
Perform the operation iteratively for each $j = 1,2,\ldots,\ell$. 
We add the whole graph $H_j$ to $\VortexA{i}$ and remove it, except for possibly $v_1$ or $v_2$, from $\MainA$ (and delete all dongles contained in $H_j$). 
In particular, all society vertices contained in $H_j$ stop to be society vertices.
If $C_j$ contains $v_1$ or $v_2$ that is incident to an edge not in $\VortexA{i}$ nor $H_j$, (w.l.o.g., let it be $v_1$) or $H_j$ contains the whole $C^\ast$, add $v_1$ to 
$\MainA$, embedding it on the boundary of $\Delta_j$ (back in its previous place if it was there before) and make the vortex bag of $v_1$ be $V(H_j)$. 
Otherwise, since $C^\ast$ is connected, there exists a society vertex $\SocietyVtx{i}{\alpha}$ immediately preceeding or succeeding $J_j$ (in the current near-embedding, i.e., after intervals $J_1,J_2,\ldots,J_{j-1}$ are processed) whose vortex bag contains a vertex of $C^\ast$;
add the whole $H_j$ to the vortex bag $\VortexBag{i}{\alpha}$. 
(Note that as $|J_j| \leq q'$, the number of society vertices of $\VortexA{i}$ is much higher than $|J_j|$.)

It is straightforward to verify that the operation presented above constructs a correctly defined path decomposition of the modified $\VortexA{i}$; note here that that $H_j$ includes vertices
of all vortex bags $\VortexBag{i}{\alpha}$, $\alpha \in J_j$. The less trivial claim is that the adhesion width of this path decomosition does not increase during the construction. To this end, consider the modification applied for one interval $J_j$. 
If in the process of this modification, $v_1$ or $v_2$ was designated as a new society vertex, say $v_1$, then if $v_{i,\alpha}^\circ$ is the society vertices immediately preceeding $J_j$,
   then $V(H_j) \cap \VortexBag{i}{\alpha} = \VortexBag{i}{\alpha+1} \cap \VortexBag{i}{\alpha}$ and similarly for the society vertex immediately succeeding $J_j$. 
 Otherwise, if $V(H_j)$ has been added to a vortex bag $\VortexBag{i}{\alpha}$ immediately preceeding $J_j$ (the argument for succeeding vortex bag $\VortexBag{i}{\beta}$ is analogous),
 then $(\VortexBag{i}{\alpha} \cup V(H_j)) \cap \VortexBag{i}{\alpha-1} = \VortexBag{i}{\alpha} \cap \VortexBag{i}{\alpha-1}$ 
 and $(\VortexBag{i}{\alpha} \cup V(H_j)) \cap \VortexBag{i}{\beta} = \VortexBag{i}{\beta-1} \cap \VortexBag{i}{\beta}$. 

We infer that in all cases $\nembed$ can be modified so that the number of components $C$ that are not aligned in $\nembed$ decreases, while keeping $\nembed$ optimal and with properly connected
dongles. This finishes the proof of the lemma.
\end{proof}

Hence, we may focus on optimal near-embeddings that are additionally aligned with 3-connected components and have properly connected dongles. 
Observe that a near-embedding $\nembed$ with apex set $\Apices$ that is aligned with 3-connected components naturally induces a near-embedding
$\nembed(G(\Apices))$ of $G(\Apices)$ that has no apices:
for every connected component $C$ of $G-\Apices-V(G(\Apices))$ with $|N_{G-\Apices}(C)| = 2$, either:
\begin{itemize}
\item if $C^\ast = \DongleA{j}$ for a dongle $\DongleA{j}$, replace it with an edge connecting the vertices of $N_{G-\Apices}(C)$, added to $\MainA$
and drawn in the disc for $\DongleA{j}$;
\item otherwise, if $C^\ast$ is contained in a dongle $\DongleA{j}$ or a vortex $\VortexA{i}$, replace it with an edge connecting the vertices of $N_{G-\Apices}(C)$,
  adjusting the path decomposition in the vortex case by replacing all vertices of $C$ with one fixed element of $N_{G - \Apices}(C)$.
\end{itemize}
The near-embedding $\nembed(G(\Apices))$ has no apices, and the same Euler genus and number of vortices as $\nembed$, and no larger maximum adhesion size than $\nembed$. 

In the other direction, given $\Apices \subseteq V(G)$ and a near-embedding $\nembedT$ of $G(\Apices)$, one can turn $\nembedT$ into a near-embedding of $G$ by, for every connected
component $C$ of $G-\Apices-V(G(\Apices))$, perform the following:
\begin{itemize}
\item If $N_{G-\Apices}(C)$ is contained in a dongle or a vortex, add $C^\ast$ to the said dongle or vortex and, in the vortex case,
  add all vertices of $C^\ast$ to one bag of the path decomposition of the vortex that contains $N_{G-\Apices}(C)$.
\item Otherwise, all vertices of $N_{G-\Apices}(C)$ are in $\MainT$. Turn $N_{G-\Apices}[C]$ into a dongle with a disc that contains all elements of $N_{G-\Apices}(C)$
on its boundary (which is possible even if there are two elements of $N_{G-\Apices}(C)$, as the edge connecting them in $G(\Apices)$ is in the $\MainT$ part of $\nembedT$).
\end{itemize}
Note that the resulting near-embedding $\nembed$ has apex set $\Apices$, the same Euler genus, number of vortices, and maximum adhesion size as $\nembedT$, and
is aligned with 3-connected components. Thus, we may interchangeably think of an optimal near-embedding of $G$ with apex set $\Apices$ that is aligned with 3-connected components
as a near-embedding of $G(\Apices)$ with no apices. 

In particular, in a near-embedding that is aligned with 3-connected components, the graph $\MainPlusA$ is 3-connected.
\begin{lemma}\label{lem:aligned-is-3conn}
Let $\nembed$ be a near-embedding of $G$ that is aligned with 3-connected components.
Then, the graph $\MainPlusA$ is 3-connected. 
\end{lemma}
\begin{proof}
Let $(A,B)$ be a separation of $\MainPlusA$ of order at most $2$; our goal is to show that $A \setminus B = \emptyset$ or $B \setminus A = \emptyset$. 

To this end, we construct a separation $(A',B')$ of $G-\Apices$ as follows. Start with $(A',B') = (A,B)$.
For every dongle $\DongleA{i}$, the vertices of $V(\DongleA{i}) \cap V(\MainA)$ form a clique in $\MainPlusA$, so they are contained either in $A$ or in $B$;
add $V(\DongleA{i})$ to $A'$ or $B'$, respectively (and remove $\PlusDongleVtx{i}$).
Similarly, for every vortex $\VortexA{i}$, its society vertices together with $\PlusVortexVtx{i}$ form a wheel that is 3-connected; hence, all of them lie in $A$ or all of them lie in $B$;
add $V(\VortexA{i})$ to $A'$ or $B'$, respectively (and remove $\PlusVortexVtx{i}$). 

We obtain a separation $(A',B')$ of $G-\Apices$ of order at most $2$. Hence, either $A' \setminus B'$ or $B' \setminus A'$ consists only of connected components
of $G-V(G(\Apices))$; without loss of generality, assume that it is $A' \setminus B'$. 
Since $\nembed$ is aligned with 3-connected components, we infer that $A \setminus B = \emptyset$, as desired.
\end{proof}

We now proceed to defining a canonical choice of dongles. Let us make the following definition.

\begin{definition}[predongle]
Let $G$ be a graph and let $\mathcal{T}$ be a tangle of order at least $4$ in $G$.
A \emph{$\mathcal{T}$-predongle} is a set $D \subseteq V(G)$ such that 
$|N(D)| \leq 3$ and $(N[D], V(G) \setminus D) \in \mathcal{T}$, that is,
 $(N[D], V(G) \setminus D)$ is a separation of order at most $3$ with $N[D]$ being the small side.
\end{definition}

A $\mathcal{T}$-predongle is \emph{maximal} if it is inclusion-wise maximal.
We often drop the parameter $\mathcal{T}$ if the tangle is clear from the context.
We need the following observation:
\begin{lemma}\label{lem:max-dongles}
Assume $G$ is a $3$-connected graph and $\mathcal{T}$ is a tangle in $G$ of order at least $4$.
Then, for every two distinct maximal $\mathcal{T}$-predongles $D_1$ and $D_2$,
we have $D_1 \cap D_2 = \emptyset$ and, furthermore, $|E(D_1,D_2)| \leq 1$.
\end{lemma}
\begin{proof}
Let $D = D_1 \cup D_2$. We consider the separation
$(N[D], V(G) \setminus D)$.
If $(N[D], V(G) \setminus D)$ is of order at most $3$, then it belongs to $\mathcal{T}$
and $D$ is a potential $\mathcal{T}$-dongle, a contradiction to the maximality of $D_1$ and $D_2$.
Thus, $|N(D)| \geq 4$.

Note that by $3$-connectivity, $|N(D_1)|=|N(D_2)|=3$.
Therefore,
   \[ 6 = |N(D_1)| + |N(D_2)| = |N(D)| + |N(D_1) \cap N(D_2)| + |N(D_1) \cap D_2| + |N(D_2) \cap D_1|. \]
Hence,
  \begin{equation}\label{eq:nd1d2} 2 \geq |N(D_1) \cap N(D_2)| + |N(D_1) \cap D_2| + |N(D_2) \cap D_1|\geq |N(D_1\cap D_2)|.
  \end{equation}
However, by 3-connectivity of $G$, $|N(D_1 \cap D_2)| \geq 3$ unless $D_1 \cap D_2 = \emptyset$. 
So $D_1 \cap D_2 = \emptyset$. 

Note that $N(D_1) \cap D_2 = \emptyset$ if and only if $N(D_2) \cap D_1 = \emptyset$.
By~\eqref{eq:nd1d2}, we infer that either both these sets are empty (which is equivalent to $E(D_1,D_2) = \emptyset$)
or both are of size exactly one (which is equivalent to $|E(D_1,D_2)| = 1$). 
\end{proof}

Consider now a set $\Apices \subseteq V(G)$ of size exactly $\napices_\ast$ and consider all optimal near-embeddings of $G$ 
with apex set $\Apices$ that are aligned with 3-connected components and have properly
connected dongles.
As we discussed, being aligned with 3-connected components means that we can equivalently think of a near-embedding of $G(\Apices)$ (with no further apices). 
Informally, we would like now to say the following.
In $G(\Apices)$, whenever we want to make a dongle, we can safely use a {\em{tidy dongle}}, which roughly is a maximal predongle with respect to a properly projected unbreakability tangle. There are two caveats, however.
First, if we want to use two maximal predongles that have an edge between
them, we need to make a minor surgery around that edge. Second, a maximal predongle can be partially in a nearby vortex. While this case can be tediously resolved using a reasoning similar to that in the proof of Lemma~\ref{lem:align-3conn}, we omit it, because later we will not need tidy dongles close to vortices. 

However, in our proof we will need the process of tidying dongles applied not only to $G(\Apices)$, but also to other 3-connected graphs embedded into $\Sigma$. Thus, we need to state
it in a bit more general form. Let us now proceed with formal arguments. 

For a 3-connected graph $H$ and a tangle $\mathcal{T}$ of order at least 4 in $H$, 
by $\Predongles_0(H,\mathcal{T})$ we denote the family of all maximal $\mathcal{T}$-predongles in $H$. 
Further, let
\[ \Predongles(H,\mathcal{T}) = \left \{D \setminus \bigcup_{D' \in \Predongles_0(H,\mathcal{T}) \setminus \{D\}} N_{H}[D']~\colon~D \in \Predongles_0(H,\mathcal{T}) \right \} \setminus \{\emptyset\}. \]

The next lemma summarizes the main properties of elements of $\Predongles(H,\mathcal{T})$.
\begin{lemma}\label{lem:predongles-properties}
Let $H$ be a 3-connected graph and let $\mathcal{T}$ be a tangle of order at least 4 in $H$. 
The elements of $\Predongles(H,\mathcal{T})$ are pairwise disjoint, nonadjacent, and for every $D \in \Predongles(H,\mathcal{T})$ the separation $(N_{H}[D], V(H) \setminus D)$
is of order $3$ in $H$. 
\end{lemma}
\begin{proof}
The fact that the elements of $\Predongles(H,\mathcal{T})$ are pairwise disjoint and nonadjacent follows from the definition and Lemma~\ref{lem:max-dongles}. 
For the last claim, fix $D \in \Predongles(H,\mathcal{T})$ and let
$D_0 \in \Predongles_0(H,\mathcal{T})$ be such that
\[ D = D_0 \setminus \bigcup_{D' \in \Predongles_0(H,\mathcal{T}) \setminus \{D_0\}} N_{H}[D']. \]
Since $H$ is $3$-connected and $D \neq V(H)$ (because it is contained in $D_0$ and $N_H[D_0]$ is a small side of a separation of $\mathcal{T}$), $|N(D)| \geq 3$.
It suffices to show that $|N(D)| \leq 3$.

To this end, recall that $|N(D_0)| = 3$. Fix $v \in N(D) \setminus N(D_0)$. By Lemma~\ref{lem:max-dongles} and the definition of $D$,
there exists $D_v \in \Predongles_0(H,\mathcal{T}) \setminus \{D_0\}$
with $v \in N(D_v)$. Thus, there is $u_v \in D_v$ with $u_v \in N(D_0)$.
Lemma~\ref{lem:max-dongles} implies that there is exactly one edge between $D_0$ and $D_v$, and this is the edge $vu_v$. 
Hence, $u_v \notin N(D)$.
Furthermore, for distinct $v \in N(D) \setminus N(D_0)$ the sets $D_v$ are pairwise distinct, and hence the vertices $u_v$ are pairwise distinct.
We infer that $|N(D)| \leq |N(D_0)| = 3$, as desired.
\end{proof}

\begin{definition}[canonical dongles]\label{def:canonical-dongles}
Let $H$ be a 3-connected graph and let $\mathcal{T}$ be a tangle of order at least 4 in $H$. 
Denote $h \coloneqq \max \{|A|~\colon~(A,B) \in \mathcal{T}\}$.
A near-embedding $\nembed$ of $H$ with no apices \emph{has canonical dongles (w.r.t. $\mathcal{T}$)}
if for every dongle $\DongleA{j}$ at least one of the following holds:
\begin{itemize}
\item for some vortex $\VortexA{i}$, the radial distance in $\MainPlusA$ between $\PlusDongleVtx{j}$ and $\PlusVortexVtx{i}$
is at most $h+1$; or
\item $V(\DongleA{j}) \setminus V(\MainA)$ belongs to $\Predongles(H,\mathcal{T})$.
\end{itemize}
\end{definition}

\begin{lemma}\label{lem:canonize-dongles}
Let $H$ be a 3-connected graph and let $\mathcal{T}$ be a tangle of order at least 4 in $H$. 
For every near-embedding $\nembed$ of $H$ with no apices such that for every dongle $\DongleA{j}$ 
the separation $(V(\DongleA{j}), V(H) \setminus (V(\DongleA{j}) \setminus V(\MainA)))$ is in $\mathcal{T}$, 
there exists a near-embedding $\nembedB$ of $H$ that differs from $\nembed$ only in the main embedded part and dongles (i.e., with the same vortices and into the same surface) 
and that has canonical dongles.
\end{lemma}
\begin{proof}
If $\nembed$ has canonical dongles, then we are done, so assume otherwise. 
Let $\DongleA{j}$ be a dongle not satisfying the properties of the definition of canonical dongles. 
We would like to adjust $\nembedA$ to ``canonize'' $\DongleA{j}$.

As $H$ is 3-connected, $V(\DongleA{j}) \cap V(\MainA)$ is of size three. By assumption, $D \coloneqq V(\DongleA{j}) \setminus V(\MainA)$ is a $\mathcal{T}$-predongle. 
Let $C$ be the maximal predongle containing $D$; note that $C$ is uniquely defined, as maximal predongles are pairwise disjoint. 
Since the disc of $\DongleA{j}$ is further than $h+1$ in the radial distance from any vortex in $\nembedA$, while $|N_{H}[C]|\leq h$ and $H[N_H[C]]$ is connected, we infer that
$C$ and $N_{H}(C)$ are disjoint from the vortices. 
Furthermore, for every dongle $\DongleA{j'}$, the predongle $\DongleA{j'} \setminus V(\MainA)$
is either contained in $C$ or disjoint with $C$. Let $J$ be the set of indices
$j'$ for which $\DongleA{j'} \setminus V(\MainA)$ is contained in $C$.

Let $U$ be the set  of those vertices of $N_{H}(C)$ that are contained in another maximal predongle of $H$.
For every vertex $u \in U$, Lemma~\ref{lem:max-dongles} asserts that $u$ has a unique neighbor
$v_u$ in $C$. Let $C' = C \setminus \{v_u~|~u \in U\}$. 

The 3-connectivity of $H$ implies that for every connected component
$C_1$ of $H[C]$, we have $N_{H}(C_1) = N_{H}(C)$. 
Consider for the remainder of this paragraph the case $U \neq \emptyset$. 
Since every vertex of $U$ has exactly one neighbor in $C$ we have that $H[C]$ is in fact connected. 
If $|C| > 1$, then from the $3$-connectivity of $H$ we infer that the vertices $\{v_u~|~u \in U\}$ are pairwise distinct
and $H[C]$ contains no cutvertices. 

We now exclude the corner case of $C' = \emptyset$.
From the previous paragraph we infer that this can happen only if $U \neq \emptyset$ 
and $C$ is an isolated vertex or $|U|=3$, $H[C]$ is a triangle with three vertices $\{v_u~|~u \in U\}$.
In both cases, we can just add $C$ and its indicent vertices to $\MainA$, drawing it exactly in the spot
where $\PlusDongleVtx{j}$ and its indicent edges are drawn, removing the dongle $\DongleA{j}$ completely
from the near-embedding.
Thus, we can assume $C' \neq \emptyset$ and hence $C' \in \Predongles(H,\mathcal{T})$. 

We replace $C$ in $\MainA$ (and all dongles $\DongleA{j'}$ for $j' \in J$)
with a new dongle $N_{H}[C']$, where
the vertices of $N_{H}(C')$ are added to $\MainA$ if they were not there before.
The remaining challenge is to define the disc associated with the new dongle 
and how to embed the vertices of $N_{H}(C')$ that are new to $\MainA$.

Recall that the 3-connectivity of $H$ implies that for every connected component
$C_1$ of $H[C]$, we have $N_{H}(C_1) = N_{H}(C)$. 
Consequently, there exists a tree $T$ in $H[N_H[C]]$ with exactly three leaves being $N_{H}(C)$
and all internal vertices in a single connected component of $H[C]$.

Observe that a vertex $u \in N_H(C)$ can be of three types:
\begin{itemize}
\item $u \notin U$. Then $u$ is not part of any predongle, hence is not in $V(\DongleA{j'}) \setminus V(\MainA)$ for any dongle $\DongleA{j'}$, and hence $u \in V(\MainA)$.
\item $u \in U$ and $u \notin V(\MainA)$. Then, as no vertex of $N_H[C]$ is part of a vortex, $u \in V(\DongleA{j'}) \setminus V(\MainA)$ for a dongle $\DongleA{j'}$.
Since every predongle is either contained in $C$ or disjoint with $C$, we have that $V(\DongleA{j'}) \setminus V(\MainA)$ is disjoint with $C$. Consequently,
      $v_u \in V(\MainA) \cap V(\DongleA{j'})$. 
\item $u \in U$ and $u \in V(\MainA)$. 
\end{itemize}
Let $U' = U \setminus V(\MainA)$, that is, $U'$ is the set of vertices of the second type above. 

Consider now the tree $T-U'$ and its projection $\dddot{T}\coloneqq \SubProj_{\nembedA}(T-U')$.
Note that the vertices $B \coloneqq (N_H(C) \setminus U') \cup \{v_u~|~u \in U'\}$ are parts 
of both $V(\MainA)$ and $T-U'$. 
Furthermore, note that, as the vertices $\{v_u~|~u \in U\}$ are pairwise distinct,
for every $u \in U'$, the edges of $\MainPlusA$ incident with $v_u$
are all part of the projection of $C$, except for the three virtual edges 
of the dongle $\DongleA{j'}$ that contains $u$. 
All these three edges are contained in the disc of $\DongleA{j'}$. 
Consequently, one can choose an area $D$ homeomorphic to a disc around $\dddot{T}$
that contains the drawing of $\dddot{T}$ in its inside,
except for the elements of $B$ that are on the boundary of $D$. 
Furthermore, $D$ can be chosen so that $D$ is disjoint with all features of $\nembedA$
insident with the vertices of $B$ that are not contained in $C$ nor being the discs
of the dongles $\DongleA{j'}$, $j' \in J$. 

For every $u \in U \setminus U'$, proceed as follows.
Recall that $u$ is a leaf of $T$, and hence a leaf of $\dddot{T}$.
Subvidide the edge of $\dddot{T}$ incident to $u$ with a new vertex $v_u'$ and proclaim
the new vertex $v_u'$ being the drawing of $v_u$ (even if $v_u$ is already embedded in $\MainA$).
Shrink $D$ a bit along the said edge so that $u$ is no longer on the boundary of $D$,
but $v_u'$. Add the edge $uv_u'$ to $\MainA$ and its drawing in $\MainPlusA$. 
This constructs the disc $D$ for the new dongle $N_H[C']$.
See Figure~\ref{fig:tidy-dongle} for an illustration of this operation.

This finishes the description of the modification step. By iteratively applying it as long as there exists a dongle
not satisfying the definition of canonical dongles, we obtain a near-embedding of $H$ that has canonical dongles. This finishes the proof.
\end{proof}

\begin{figure}[tb]
\begin{center}
\includegraphics{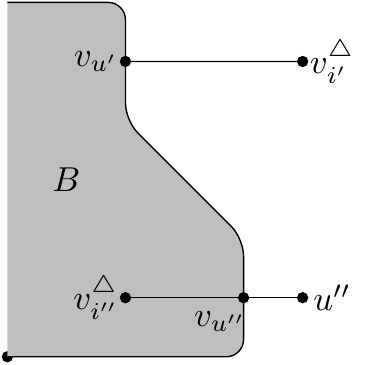}
\caption{Defining the disc in the proof of Lemma~\ref{lem:canonize-dongles}. 
In the top row, we have $u' \in U'$ and $u' \in V(\DongleA{i'}) \setminus V(\MainA)$, so $v_{u'} \in V(\MainA)$. 
In the bottom row, we have $u'' \in U \setminus U'$ but $v_{u''} \in V(\DongleA{i''}) \setminus V(\MainA)$, so we need to find a new place for $v_{u''}$ in $\MainA$.}\label{fig:tidy-dongle}
\end{center}
\end{figure}

Observe that $G(\Apices)$ is $(q-|\Apices|, k-|\Apices|)$-unbreakable similarly as $G-\Apices$, and let $\mathcal{T}_\Apices$ be the associated unbreakability tangle.
If $\Apices$ is clear from the context, we denote
$\Predongles_0 = \Predongles_0(G(\Apices),\mathcal{T}_\Apices)$
and $\Predongles = \Predongles(G(\Apices),\mathcal{T}_\Apices)$.

\begin{definition}[tidy dongles]\label{def:tidy-dongles}
We say that a near-embedding $\nembed$ with apex set $\Apices$ has \emph{tidy dongles} if
it has properly connected dongles, is aligned with 3-connected components, every dongle contains at least one vertex that is not in $\MainA$, and the corresponding near-embedding $\nembedT = \nembed(G(\Apices))$ has canonical dongles (w.r.t. $\mathcal{T}_\Apices$). 
\end{definition}
Observe that for every $(A,B) \in \mathcal{T}_\Apices$, we have $|A| \leq q$, so the value $h$ in the definition of canonical dongles is at most $q$. 

\begin{lemma}\label{lem:tidy-dongles}
For every set $\Apices$, if there exists an optimal near-embedding of $G$ with apex set $\Apices$, then there exists an optimal near-embedding of $G$ with apex set $\Apices$
that has tidy dongles.
\end{lemma}
\begin{proof}
By Lemma~\ref{lem:align-3conn}, there exists an optimal near-embedding with apex set $\Apices$ that has properly connected dongles and that is aligned with 3-connected components.
Furthermore, observe that if $\DongleA{i}$ is a dongle with $V(\DongleA{i}) \subseteq V(\MainA)$,
then we can move $\DongleA{i}$ to $\MainA$ as 
it is straightforward to embed $\DongleA{i}$ into the disc $\DongleDisc{i}$ and this 
operation does not break the property of being aligned with 3-connected components. 

Finally, we can use Lemma~\ref{lem:canonize-dongles} to modify the induced embedding $\nembedT$ of $G(\Apices)$ so that it has canonical dongles, and project this embedding of $G(\Apices)$
to a near-embedding of $G$ that has all the required properties. 
Note that, by the definition of $\mathcal{T}_\Apices$, every dongle $\DongleT{j}$ of $\nembedT$ satisfies $|V(\DongleT{j})| \leq q$ (so $(V(\DongleT{j}), V(G(\Apices)) \setminus (V(\DongleT{j}) \setminus V(\MainT))) \in \mathcal{T}_\Apices$), so the application of Lemma~\ref{lem:canonize-dongles} is justified. 
\end{proof}

\section{Large grids}\label{ss:large-grids}
In this section we prove a statement that intuitively says the following: if in the graph $G$ we find a huge (cylindrical) wall $W$, and on the other hand we work with some near-embedding $\nembed$ of $G$, then the apices and the vortices of $\nembed$ can affect $W$ only locally. In particular, $W$ will contain a large subwall $W'$ that is edge-disjoint with vortices and does not contain any apex.

We start with fixing two constants for a fixed near-embedding $\nembed$:
\begin{align*}
\WallHitSize(\nembed) &= 4((2\vortexbagsize+6)(\eulerg+2)\nvortices + \napices + 1 )^2+\napices+1=\Oh(q^2k^4), \\
\WallHitRadius(\nembed) &= 
  2\cdot((56 \cdot ((2\vortexbagsize+6)(\eulerg+2)\nvortices + \napices + \nvortices + 1) \cdot (\eulerg+1)^{\nvortices+2} + \nvortices) \cdot (1 + \vortexadhwidth) \\ &\qquad + \vortexbagsize + \donglesize + \napices) = \Oh(qk^{k+5}).
  \end{align*}
We often drop the argument if the near-embedding is clear from the context.
We also use $\WallHitSize_\ast$ and $\WallHitRadius_\ast$ for the maximum of respective values
for an optimal near-embedding $\nembed$ (i.e., using maximum allowed values of $\vortexbagsize$, $\vortexadhwidth$, and $\donglesize$ as in Lemma~\ref{lem:opt-ne-rep}). 

The main result of this section is the following.

\begin{lemma}\label{lem:wall-survives}
Let $\nembed$ be an optimal near-embedding of $G$ and let $W$ be a $w \times h$ wall that is a subgraph of $G$.
Then there exist a set $X$ consisting of at most $\WallHitSize$
vertices of $W$
such that every vertex of $W$ that is an apex of $\nembed$ and every edge
of $W$ that is in a vortex of $\nembed$ is within radial distance at most
$\WallHitRadius$
from a vertex of $X$ in the natural planar embedding of $W$. 
\end{lemma}

A weaker version of Lemma~\ref{lem:wall-survives} has been proven as~\cite[Lemma 22]{DiestelKMW12} and recalled as Lemma~\ref{lem:DKMW12}.
Lemma~\ref{lem:DKMW12} says that if $G$ contains a large subdivided wall (which has large treewidth),
then still a large-treewidth part of this wall needs to be contained in $\MainPlusA$. 
We use the following statement to bootstrap this result. Here, recall that a {\em{linkage}} from $A$ to $B$ is a family of vertex-disjoint paths, each with one endpoint in $A$ and the other in $B$.

\begin{lemma}\label{lem:AtoB}
Let $\nembed$ be a near-embedding of $G$ into a surface $\Sigma$ of Euler genus $\eulerg$,
with $\nvortices$ vortices, maximum vortex bag size $\vortexbagsize$, and $\napices$ apices.
Assume in $G$ there are two vertex-disjoint connected subgraphs $A$ and $B$
that contain no apex nor an edge of a vortex. 
Then, for every linkage $\linkage$ from $A$ to $B$, at most 
\[(2\vortexbagsize+6)(\eulerg+2)\nvortices + \napices\]
paths of $\linkage$ contain an apex or an edge of a vortex.
\end{lemma}
\begin{proof}
First, note that the statement is trivial if $|V(A)| \leq 1$ or $|V(B)| \leq 1$, so assume otherwise.
In particular, the connectedness of $A$ and $B$ imply that $A$ and $B$ do not contain vertices of vortices
that are not society vertices.

Without loss of generality we may assume that every path in $\linkage$ does not contain any internal vertex in $A \cup B$. 
At most $\napices$ paths of $\linkage$ contain an apex.
Thus, it suffices to prove the following: if no path of $\linkage$ contains an apex
and every path of $\linkage$ contains an edge of a vortex, then 
\[ |\linkage| \leq (2\vortexbagsize+6)(\eulerg+2)\nvortices.\]

For every $P \in \linkage$, follow $P$ from the endpoint in $A$ and indicate the first edge 
$e_P$ that is in a vortex. For $1 \leq i \leq \nvortices$, let $\linkage_i$ be the set
of those paths $P \in \linkage$ for which $e_P$ lies in $\VortexA{i}$. 
Fix $1 \leq i \leq \nvortices$. It suffices to prove the following:
\[ |\linkage_i| \leq (2\vortexbagsize+6)(\eulerg+2). \]
Let $\linkage_i'$ be the set of subpaths of the paths of $\linkage_i$ between the endpoint in $A$ (inclusive) and $e_P$ (exclusive). Note that $\linkage_i'$ is a linkage of the same size
as $\linkage_i$, leading from $A$ to the society vertices of $\VortexA{i}$. 

Let $\ell = |\linkage_i'|$ and assume $\ell > (2\vortexbagsize+6)(\eulerg+2)$.
Let $y_1,y_2,\ldots,y_\ell$ be the endpoints of $\linkage_i'$ in society vertices of $\VortexA{i}$,
    in the order along $\VortexDisc{i}$.
Let $Q_j \in \linkage_i'$ be the path ending in $y_j$, let $x_j \in A$ be the other endpoint of $Q_j$, and let $P_j \in \linkage_i$ be the path of which $Q_j$ is a subpath.

Let $H_A$ be a spanning tree of $A$. 
For indices $1 \leq j < j' \leq \ell$, let $R_{j,j'}$ be a concatenation
of $Q_j$, a path in $H_A$ between $x_j$ and $x_{j'}$, and $Q_{j'}$, 
and let $C_{j,j'}$ be a cycle in $\MainPlusA$ consisting of $\PathProj_{\nembedA}(R_{j,j'})$
and the path $y_j-\PlusVortexVtx{i}-y_{j'}$.
Note that as $A$ does not contain an edge of a vortex, $C_{j,j'}$ is indeed a simple cycle. 

Consider now cycles $C_{(\iota-1)\cdot \lfloor \ell/(\eulerg+2) \rfloor + 1, \iota \cdot \lfloor \ell/(\eulerg+2) \rfloor}$ for $\iota \in [\eulerg+2]$.
These cycles may intersect in $H_A$ and in $\PlusVortexVtx{i}$, but they correspond to noncrossing
(possibly touching) closed curves $\{\gamma_\iota~|~\iota \in [\eulerg+2]\}$
in $\Sigma$, and every curve $\gamma_\iota$ passes through $\PlusVortexVtx{i}$. 
Each curve $\gamma_\iota$ splits the disc $\VortexDisc{i}$ into two parts, one containing society vertices $y_j$ for $(\iota-1) \cdot \lfloor \ell/(\eulerg+2) \rfloor < j \leq \iota \cdot \lfloor \ell/(\eulerg+2) \rfloor$ on its boundary, 
and the other containing the remaining society vertices on its boundary; we refer to the first part as the \emph{inside} of $\gamma_\iota$. 
Modify each curve $\gamma_\iota$ a bit around $\PlusVortexVtx{i}$, shrinking the part inside $\gamma_\iota$, so that the curves
$\gamma_\iota$ are disjoint within $\VortexDisc{i}$.
Now, the intersection of the outside components of $\VortexDisc{i} \setminus \gamma_\iota$ over $\iota \in [\eulerg+2]$
is a component of $\VortexDisc{i} \setminus \bigcup_{\iota \in [\eulerg+2]} \gamma_\iota$ that contains $\PlusVortexVtx{i}$ and all curves $\gamma_\iota \cap \VortexDisc{i}$ in its closure. 
Lemma~\ref{lem:diestelB6} now implies that 
among the curves $\{\gamma_\iota\colon \iota\in [\eulerg+2]\}$, at least two enclose a disc that contains the inside of $\gamma_\iota$. 
Let $\gamma^1=\gamma_{\iota_1}$ and $\gamma^2=\gamma_{\iota_2}$ be two such curves, and let $D_1$ and $D_2$ be the respective two discs, as explained above. For $\alpha\in \{1,2\}$, let 
$$j_\alpha=(\iota_\alpha-1)\cdot \lfloor \ell/(\eulerg+2) \rfloor+1\qquad\textrm{and}\qquad j_\alpha'=\iota_\alpha\cdot \lfloor \ell/(\eulerg+2) \rfloor,$$
so that $\gamma^\alpha$ corresponds to the cycle  $C_{j_\alpha,j_\alpha'}$. Note that we have $j_\alpha'-j_\alpha \geq \ell/(\eulerg+2)-2$.

We have the following crucial observation:
\begin{claim}\label{cl:AtoB:crux}
Fix $\alpha \in \{1,2\}$. 
If a path $P_j$ contains a vertex of $\MainA$ inside $D_\alpha$ and a vertex of $\MainA$ that is not in the closure of $D_\alpha$, 
   then $P_j$ contains a vertex of the vortex bag of $y_{j_\alpha}$ or of $y_{j_\alpha'}$. 
\end{claim}
\begin{proof}
Let $P_j$ be a path satisfying the assumptions. 
If $j \in \{j_\alpha, j_\alpha'\}$, then we are done, as $y_{j_\alpha}$ lies on $P_{j_\alpha}$ and $y_{j_\alpha'}$ lies on $P_{j_\alpha'}$, so assume otherwise.
Let $R$ be a minimal subpath of $P_j$ with one endpoint in $\MainA$ inside $D_\alpha$ and one endpoint in $\MainA$ not in the closure of $D_\alpha$. 

By minimality, no internal vertex of $R$ belongs to $\MainA$, so $R$ is either a single edge of $\MainA$, is contained in a single dongle, or is contained in a single vortex. 
Observe that the projection $\PathProj_\nembed(R)$ needs to intersect $\gamma_\iota$. 
However, since $R$ is vertex-disjoint with $R_{j_\alpha,j_\alpha'}$ (as $R$ does not contain
    a vertex of $A$), Lemma~\ref{lem:proj-2-paths} excludes the first two cases. 
For the remaining case,
since $C_{j,j'}$ does not pass through any vertex $\PlusVortexVtx{i'}$ for $i' \neq i$, it follows that if $R$ is contained in a single vortex, then it is contained in $\VortexA{i}$. 
However, then the endpoints of $R$ are society vertices of $\VortexA{i}$ that lie on different arcs of the boundary of $\VortexDisc{i}$ delimited by $y_{j_\alpha}$ and $y_{j_\alpha'}$.
This implies that $R$ needs to contain a vertex of the vortex bag at $y_{j_\alpha}$ or at $y_{j_\alpha'}$, as desired. 
\cqed\end{proof}

Recall that $D_\alpha$ contains the inside of $\gamma^\alpha$, so in particular, $D_\alpha$ contains all society vertices $y_j$ for $j_\alpha < j < j_\alpha'$ in its interior or on the boundary.
Since $j_\alpha'-j_\alpha \geq \ell/(\eulerg+2)-2 > 2(\vortexbagsize+1) + 2$, 
at least one path $P_{j^\alpha}$ for some $j_\alpha < j^\alpha < j_\alpha'$ does not contain any vertex contained in the vortex bags at $y_{j_\alpha}$ and $y_{j_\alpha'}$. 
Recall that the next edge on $P_{j^\alpha}$, the one after $y_{j^\alpha}$ in the direction of $B$, is an edge $e_{P_{j^\alpha}}$ in the vortex $\VortexA{i}$. 
Consequently, between $y_{j^\alpha}$ and $B$, $P^{j^\alpha}$ passes through another society vertex $w_\alpha$
that is on the same arc of the boundary of $\VortexDisc{i}$ between $y_{j_\alpha}$ and $y_{j_\alpha'}$ as $y_{j^\alpha}$. 
Observe that $w_\alpha$ is not in $A$ nor on $P_{j_\alpha}$ nor on $P_{j_\alpha'}$, and thus $w_\alpha$ is strictly inside $D_\alpha$ (i.e., not on its boundary). 
By Claim~\ref{cl:AtoB:crux}, $P_{j^\alpha}$ does not contain a vertex of $\MainA$ that is not in the closure of $D_\alpha$. 
Furthermore, since all vertices of $C_{j_\alpha,j_\alpha}$ that are in $\MainA$ are in $A$, on $P_{j_\alpha}$, or on $P_{j_\alpha'}$, 
all vertices of $V(P_{j^\alpha}) \cap V(\MainA)$ are in $D_\alpha$, except for possibly the endpoint $x_{j^\alpha}$ that may be on the boundary of $D_\alpha$.

Let $z_\alpha$ be the closest-to-$B$ vertex of $\MainA$ on $P_{j^\alpha}$. Note that (thanks to the existence of $w_\alpha$) $z_\alpha$ is in $D_\alpha$. 
Consider the path $R$ defined as the concatenation of the subpaths of $P_{j^\alpha}$ between $z_\alpha$ and the endpoint in $B$, for $\alpha=1,2$, and a path contained in $B$
between the said endpoints. 
We observe that $R$ does not contain any edge of a vortex: $B$ does not contain any edge of a vortex by assumption, and
the subpath of $P_{j^\alpha}$ between $z_\alpha$ and the endpoint in $B$ does not contain an edge of a vortex, because $B$ does not contain vertices of vortices that are not society vertices
and, by the choice of $z_\alpha$, $z_\alpha$ is the only vertex of $\MainA$ on the said subpath.

Also, note that $R$ is vertex-disjoint with $R_{j_\alpha,j_\alpha'}$. 
Since $\PlusVortexVtx{i}$ does not lie on $R$, 
Lemma~\ref{lem:proj-2-paths} implies that $\PathProj_\nembed(R)$ is vertex-disjoint
with $C_{j_\alpha,j_\alpha'}$ for $\alpha=1,2$. This is a contradiction with $R$
having one endpoint in $D_1$ and the other endpoint in $D_2$. 
\end{proof}

We are now ready to prove Lemma~\ref{lem:wall-survives}.

\begin{proof}[Proof of Lemma~\ref{lem:wall-survives}.]
Let 
\begin{align*}
a &= (2\vortexbagsize+6)(\eulerg+2)\nvortices + \napices + 1,\\
b &= (56 \cdot (4a + \nvortices) \cdot (\eulerg+1)^{\nvortices+2} + \nvortices) \cdot (1 + \vortexadhwidth) + \vortexbagsize + \donglesize + \napices. 
\end{align*}
Note that $\WallHitSize(\nembed)=4a^2+\napices+1$ and $\WallHitRadius(\nembed)=2b$.
We say that an edge of $W$ is a \emph{deep edge} if it is not within radial distance
at most $2b$ from the outerface of $W$ in the natural plane embedding. 
It suffices to expose a set $X$ of at most $4a^2$ vertices of $W$ such that every edge of $W$
that is deep and in a vortex of $\nembed$, lies within radial distance at most $b$ from a vertex
of $X$ in the natural plane embedding of $W$. For the lemma statement, we return
$X$, all apices of $\nembed$ that are in $W$, and one vertex in the first row of $W$. 

If $w \leq 4b$ or $h \leq 4b$, then there are no deep edges of the grid, so assume otherwise. 
Let $W_\alpha$, $\alpha \in [4]$, be the four $b \times b$ walls that are subwalls
of $W$ in its four corners. 

\begin{claim}\label{cl:wall-survives:A}
Fix $\alpha \in [4]$. There exists a connected subgraph $A_\alpha$ of $W_\alpha$
that is edge-disjoint with vortices, does not contain apices, and contains vertices
in at least $a$ distinct rows and at least $a$ distinct columns of $W_\alpha$.
\end{claim}
\begin{proof}
We restrict the near-embedding $\nembed$ to a near-embedding $\nembedB$ of $W_\alpha$
as follows. First, the apices of $\nembedB$ are $\Apices \cap V(W_\alpha)$. 
Second, $\MainB$ consists of vertices and edges of $\MainA$ that are present in $W_\alpha$.
Third, the vortices of $\nembedB$ are the vortices of $\nembed$, with the vertices
not in $W_\alpha$ iteratively deleted.
Finally, for every dongle of $\nembed$, we first restrict it to its intersection with
$W_\alpha$, and then apply Lemma~\ref{lem:split-dongle} to it. 
Thus, $\nembedB$ is a near-embedding of $W_\alpha$ into a surface of Euler genus at most
$\eulerg$ with properly connected dongles. 

Let $W_\alpha'$ be equal to $\SubProj_{\nembedB}(W_\alpha-\Apices)$
   with all virtual vortex vertices and edges removed.
Let $W_\alpha''$ be $W_\alpha'$ with additionally, for every dongle
$\DongleB{j}$ that contains an edge of $W_\alpha$, the vertices of $V(W_\alpha) \cap V(\DongleB{j}) \cap V(\MainB)$ turned into a clique. 
As the treewidth of $W_\alpha$ is $b$,
by Lemma~\ref{lem:DKMW12}, applied to $\nembedB$, we infer
that the treewidth of $W_\alpha''$ is at least $4a$. 

Observe that one can turn a tree decomposition of $W_\alpha'$ into
a tree decomposition of $W_\alpha''$ by adding, for every dongle $\DongleB{j}$ with
$\PlusDongleVtx{j} \in V(W_\alpha')$, all vertices of
$V(W_\alpha) \cap V(\DongleB{j}) \cap V(\MainB)$ to all bags containing $\PlusDongleVtx{j}$. 
Consequently, the treewidth of $W_\alpha'$ is at least $a$.

Let $A_\alpha'$ be any lift of $W_\alpha'$ in $\nembedB$ (which is well-defined as 
$W_\alpha'$ does not contain any virtual vortex edge
and $\nembedB$ has properly connected dongles).  
An important observation is that since $\nembedB$ has properly connected dongles, 
$A_\alpha'$ can be obtained from $W_\alpha'$ by performing the following operations a number of times: (a) subdivision of an edge; (b) suppression
of a degree-2 vertex, (c) addition or deletion of a vertex of degree at most $1$. 

As these operations do not change the treewidth of a graph, we infer that the treewidth of $A_\alpha'$ is also at least $a$. Hence, $A_\alpha'$ contains a connected component $A_\alpha$ whose treewidth is also at least $a$. Note that $A_\alpha$ has to contain
vertices from at least $a$ rows and at least $a$ columns of $W_\alpha$, for otherwise its treewidth would be smaller than $a$. Hence, $A_\alpha$ has all the desired properties.
\cqed\end{proof}

Call a path in $G$ {\em{clean}} if it is edge-disjoint with vortices of $\nembed$ and does not contain any apex of $\nembed$.

Let $C_1,\ldots,C_w$ be the consecutive columns of $W$ and $R_1,\ldots,R_h$ be the consecutive rows of $W$. 
Without loss of generality, assume that $W_1$ and $W_2$ are in the top-left and top-right
corners of~$W$, respectively, that is, they both intersect rows $R_1\ldots,R_b$. Since $A_1$ and $A_2$ intersect at least $a$ different rows of $W$, within $W$ there exists a linkage $\cal L$ of size $a$ from $A_1$ to $A_2$ that is disjoint from rows $R_{b+1},\ldots,R_h$. By Lemma~\ref{lem:AtoB}, there exists a clean path $P_1\in {\cal L}$. Note that $P$ intersects each of the columns $C_{b+1},\ldots,C_{w-b}$ and is disjoint with rows $R_{b+1},\ldots,R_h$. By applying an analogous construction to subwalls $W_3$ and $W_4$ we find a clean path $P_2$ that is disjoint with rows $R_1,\ldots,R_{h-b}$ and intersects each of the columns $C_{b+1},\ldots,C_{w-b}$.

For $i\in \{b+1,\ldots,w-b-1\}$, let $S_i$ be the set of edges of $W$ that either are contained in the column $C_i$, or lie on a degree-$2$ path in $W$ connecting a vertex of $C_i$ with a vertex of $C_{i+1}$. Thus, sets $S_i$ are pairwise disjoint and each deep edge in $W$ belongs to one of them.
Let $F$ be the set of all deep edges of $W$ that are contained in vortices of $\nembed$. 
Call an index $i\in \{b+1,\ldots,w-b-1\}$ {\em{column-affected}} if $F\cap S_i\neq \emptyset$. Let $I\subseteq \{b+1,\ldots,w-b-1\}$ be the set of all column-affected indices that are even. Observe that for each $i\in I$ we can find a path $Q_i$ such that (i) $Q_i$ starts in a vertex of $P_1$, ends in a vertex of $P_2$, and otherwise is disjoint with $P_1\cup P_2$, (ii) $E(Q_i)\subseteq S_i\cup S_{i+1}$, and (iii) $Q_i$ traverses an edge of $F$. Moreover, since all indices $i\in I$ are even, paths $Q_i$ can be chosen to be vertex-disjoint. Thus, $\{Q_i\colon i\in I\}$ is a linkage connecting $P_1$ and $P_2$ that would contradict the statement of Lemma~\ref{lem:AtoB} unless $|I|\leq a$. A symmetric reasoning shows that there are at most $a$ odd column-affected indices. Therefore, the total number of column-affected indices is at most $2a$.

By switching the role of rows and columns we can perform a symmetric analysis, yielding a conclusion that there are at most $2a$ row-affected indices. Now construct $X$ by selecting one vertex from the intersection of the $i$th row and the $j$th column, for every row-affected index $i$ and column-affected index $j$. Then $|X|\leq 4a^2$ and every edge of $F$ is at radial distance at most $1$ from a vertex of $X$ in the natural plane embedding of $W$.
\end{proof}

We now make two corollaries of Lemma~\ref{lem:wall-survives}. 
The first one is immediate.
\begin{lemma}\label{lem:cylinder-survives}
Let $\nembed$ be a near-embedding of $G$ and let $W$ be a $w \times h$ cylindrical wall that is a subgraph of $G$ where $h > 4\WallHitRadius$. 
Then there exist a family $X$ of at most $2\WallHitSize$
vertices of $W$
such that every vertex of $W$ that is an apex of $\nembed$ and every edge
of $W$ that is in a vortex of $\nembed$ is within radial distance at most $\WallHitRadius$
from a vertex of $X$ in the natural plane embedding of $W$. 
\end{lemma}
\begin{proof}
Cover $W$ with two $w \times h$ walls $W_1$ and $W_2$, shifted by $\lfloor h/2 \rfloor$ columns with respect to each other.
Apply Lemma~\ref{lem:wall-survives} to each $W_\alpha$, $\alpha \in \{1,2\}$, obtaining
a set $X_\alpha$. It is evident that that $X \coloneqq X_1 \cup X_2$ satisfies the requirements of the lemma.
\end{proof}

The second one requires a bit more care, but the main combinatorial observation is provided by Lemma~\ref{lem:sigma-wall-cover}.

\begin{lemma}\label{lem:Sigma-wall-survives}
There exists a universal constant $\WallHitC$ such that the following holds.
Let $\nembed$ be a near-embedding of $G$ and let $W$ be a subdivision of a $\Sigma$-wall
of order $h$ that is a subgraph of $G$, where $h > 4\WallHitRadius$.
Then there exists a family $X$ of at most $\WallHitC (\eulerg+1)^2 \WallHitSize$ vertices of $W$
such that every vertex of $W$ that is an apex of $\nembed$ or every edge of $W$ that
is in a vortex of $\nembed$ is within radial distance at most $\WallHitRadius$
from a vertex of $X$ in the natural embedding of $W$ in $\Sigma$.
\end{lemma}
\begin{proof}
Let ${\cal C}(W)$ be the family of subwalls of $W$ provided by Lemma~\ref{lem:sigma-wall-cover}. Recall that $|{\cal C}(W)|\leq \Oh((\eulerg+1)^2)$. For each $U\in {\cal C}(W)$, apply Lemma~\ref{lem:wall-survives} yielding a set $X_U$. It now suffices to take $X\coloneqq Y\cup \bigcup_{U\in {\cal C}(W)} X_U$, where $Y$ comprises the first and the last vertices in the $h$th, the $(2h)$th, $\ldots$, and the $(\ell h)$th row of the elementary wall underlying $W$. Recalling that $|Y|=\Oh(\ell)=\Oh(\eulerg)$, that $X$ satisfies the required conditions follows directly from Lemmas~\ref{lem:sigma-wall-cover} and~\ref{lem:wall-survives}.
\end{proof}

\section{Radial discs}\label{ss:radial-discs}
We will need a notion similar to a radial ball, called a \emph{radial disc}.
Fix an optimal near-embedding $\nembed$ of the graph $G$ with tidy dongles. 
Let $\rballX{\nembed}{r}{v}$ denote the ball $\rball{r}{v}$ in the graph $\MainPlusA$.
For a vertex $v \in V(\MainPlusA)$
and a radius parameter
$$r \leq \rmax \coloneqq \lfloor k''/(200(q+k+1)) \rfloor - 2q - k - 2,$$
we define the \emph{radial disc} $\rdisc{\nembed}{r}{v}$, which is a subgraph
of $\MainPlusA$ containing $\rballX{\nembed}{r}{v}$ and a bit more. 
Note here that the facewidth of 
the embedding of $\MainPlusA$ is much larger than $2\rmax + 3$.
Formally, the radial disc is introduced through the following statement.
\begin{lemma}\label{lem:radial-disc}
For every optimal near-embedding $\nembed$ with tidy dongles,
vertex $v \in V(\MainPlusA)$
and radius $2 \leq r \leq \rmax$, 
there exists a simple contractible cycle $\rcycle{\nembed}{r}{v} \subseteq \partial \rballX{\nembed}{r}{v}$
such that the following holds. There is a closed disc $\Delta$ with boundary $\rcycle{\nembed}{r}{v}$ that contains $v$ together with the whole graph $\rballX{\nembed}{r}{v}$. Let $\rdisc{\nembed}{r}{v}$ be the subgraph of $\MainPlusA$ comprised of all vertices and edges embedded in $\Delta$.
Then $\rdisc{\nembed}{r}{v}$ is of treewidth at most $2q''+10r$ and every vertex
of $\rdisc{\nembed}{r}{v}$ is within radial distance at most $2q''+10r$ from $v$. Further,
the subgraph of $\MainPlusA$ consisting of all edges and vertices embedded
not in the interior $\Delta$
has Euler genus $\eulerg_\ast$ and contains more than $2q''$ vertices.

Furthermore, $\rdisc{\nembed}{r}{v}$ contains a sequence of cycles $(\rintcycle{\nembed}{r}{i}{v})_{i=1}^{r-1}$, all vertex-disjoint and disjoint from $\rcycle{\nembed}{r}{v}$ and $v$,
  such that $\rintcycle{\nembed}{r}{i}{v}$ encloses a disc containing $v$ and $\rintcycle{\nembed}{r}{i'}{v}$ for $i' < i$ and not containing $\rcycle{\nembed}{r}{v}$ nor $\rintcycle{\nembed}{r}{i'}{v}$ for $i' > i$. 
Finally,
$\rdisc{\nembed}{r}{v}$ also contains three paths that are vertex-disjoint except for $v$ and connect $v$ with $\rcycle{\nembed}{r}{v}$.
\end{lemma}

\begin{proof}
Assume $v \in V(\MainPlusA)$ and $2 \leq r \leq \rmax$ are given. 
By Corollary~\ref{cor:ltw}, there exists a vertex $u \in V(\MainPlusA)$ that lies within
radial distance larger than $2q''+r+2$ from $v$ in $\MainPlusA$, as otherwise the treewidth of $\MainPlusA$
is bounded by $f_\mathrm{ltw}(\eulerg_\ast) \cdot (2q'' + r+2)$.
Together with Lemma~\ref{lem:DKMW12}, this would contradict the assumption on the treewidth of $G$.
Fix one such vertex $u$; note that $q'' > r$, so $u \notin V(\rballX{\nembed}{r+1}{v})$. 
Let $D_u$ be the connected component of $\MainPlusA-V(\rballX{\nembed}{r}{v})$ that contains~$u$
and let $D_u'$ be the connected component of $\MainPlusA-V(\rballX{\nembed}{r+1}{v})$ that contains
$u$. Clearly $D_u'$ is a subgraph of $D_u$.
Note that $|V(D_u)| > 2q''+2$ and $|V(D_u')| > 2q''$.
Furthermore, since no two vertices of $V(\MainPlusA) \setminus V(\MainA)$ lie in the interior of the same
face of $\MainA$, $|V(\MainA) \cap V(D_u)| > q''+1$ and $|V(\MainA) \cap V(D_u')| > q''$.

If $\eulerg_\ast = 0$ (i.e., $\MainPlusA$ is embedded in the sphere),
then apply Lemma~\ref{lem:rballs:planar}, obtaining a sequence $C_1,C_2,\ldots,C_{r+1}$
of cycles separating $v$ and $u$.
If $\eulerg_\ast > 0$, then apply Lemma~\ref{lem:rballs}, obtaining a sequence $C_1,C_2,\ldots,C_{r+1}$
of cycles around $v$. 
In both cases, the component of $\Sigma-C_{i}$ that contains $v$ (and is isomorphic to a disc)
  is called the \emph{inside of $C_{i}$}. 
Also in both cases,
fix a shortest-path tree $T$ in the face-vertex graph of $\MainPlusA$ rooted in $v$
with leaves $V(C_{r+1})$; for $w \in V(C_{r+1})$, let $\sigma(w)$ be the shortest path in $T$ from $w$ to $v$.
Let $w_1,w_2$ be two consecutive vertices on the cycle $C_{r+1}$. 
Then, $\sigma(w_1)$, $\sigma(w_2)$, and the edge $w_1w_2$ form a closed walk $C(w_1w_2)$
of length $4r+3$ in the union
of $\MainPlusA$ and the face-vertex graph of $\MainPlusA$. Since $C_{r+1}$ is contractible
and $\rballX{\nembed}{r+1}{v}$ is contained in the inside of $C_{r+1}$, $C(w_1w_2)$ is contractible as well.
By the \emph{inside of $C(w_1w_2)$} we mean the connected component of $\Sigma-C(w_1w_2)$
that is homeomorphic to a disc and contained in the inside $C_{r+1}$.

Note that $C(w_1w_2)$ contains $2r+3$ vertices of $\MainPlusA$
and we can treat $C(w_1w_2)$ as a closed face-vertex curve of $\MainPlusA$ (following 
 $w_1w_2$ on either of its sides). 
By Lemmas~\ref{lem:preimage-size} and~\ref{lem:opt-ne-rep}, 
   $|\preimage(C(w_1w_2))| \leq (2r+3)(q+1) < k''$. 
Furthermore, Lemma~\ref{lem:opt-ne-rep} implies that $C(w_1w_2)$ may contain society
vertices or the vertex $\PlusVortexVtx{i}$ of at most one vortex. 

Let $Y'$ be the set of vertices of $\MainPlusA$ that are inside $C(w_1w_2)$ 
and let $Y$ be the preimage of the inside of $C(w_1w_2)$. 
By Lemma~\ref{lem:preimage-cut}, 
$\preimage(C(w_1w_2))$ separates $Y$ from the remainder of the graph $G-\Apices$. 
By $(q'',k'')$-unbreakability of $G$, either $|Y| + |\preimage(C(w_1w_2))| \leq q''$ or $|V(G) \setminus Y| \leq q''$.

Assume first that the latter holds. 
This in particular implies that $V(D_u') \cap V(\MainA) \subseteq Y$.
However, in the case $\eulerg_\ast = 0$ Lemma~\ref{lem:rballs:planar} asserts that $v$ and $u$ (and thus $D_u'$) are on the opposite sides of $C_{r+1}$, which is a contradiction.
For the case $\eulerg_\ast > 0$, we modify $\nembed$ by cutting along 
$\sigma(w_1)$ and $\sigma(w_2)$ and declaring $G-(\Apices \cup Y)$ 
a dongle. As discussed, $C(w_1w_2)$ may pass through society vertices of at most one vortex. 
Furthermore, since $C(w_1w_2)$ consists of two shortest paths with a common endpoint in the radial graph and an edge of $\MainPlusA$, 
  $C(w_1w_2)$ visits at most six society vertices of the said vortex and the cutting operation splits it into at most 3 vortices. 
However, the resulting near-embedding is planar, that is, has strictly smaller Euler genus
than $\nembed$. This is a contradiction with the optimality of $\nembed$.

Hence, we have $|Y|+|\preimage(C(w_1w_2))| \leq q''$. 
Note that for every vertex $\PlusDongleVtx{i} \in Y'$, there is at least
one vertex of $V(\DongleA{i}) \setminus V(\VortexA{i})$ in $Y$, as $\nembed$ has tidy dongles.
For every vertex $\PlusVortexVtx{i} \in Y'$, all vertices of $\VortexA{i}$ are in
$Y \cup \preimage(C(w_1w_2))$. Hence, $|Y'| \leq 2q''$. 
In particular, we have that in the case $\eulerg_\ast > 0$, $D_u'$ is not inside $C(w_1w_2)$. The bound on the number of vertices inside $C(w_1w_2)$ immediately implies the promised bound on the radial distance between those vertices and $v$.

Observe that the discs inside the closed walks  $C(w_1w_2)$ for $w_1w_2 \in E(C_{r+1})$ 
form a partition of the disc inside $C_{r+1}$.
Noting that $D_u'$ has size larger than $q''+1$ and is disjoint with those closed walks, we conclude that $D_u'$ is not inside $C_{r+1}$ in the case $\eulerg_\ast > 0$
(recall that this conclusion has been already obtained for the case $\eulerg_\ast = 0$ as it is immediate from Lemma~\ref{lem:rballs:planar}).
Consequently, $C_{r+1}$ is contained in $D_u$. In particular, $D_u$ is not in the inside of $C_r$.

We define $\rcycle{\nembed}{r}{v} \coloneqq C_r$ and $\Delta$ to be the closed disc inside.
Then $\rdisc{\nembed}{r}{v}$ is the subgraph of $\MainPlusA$
contained in the $\Delta$.
Furthermore, we define $\rintcycle{\nembed}{r}{i}{v} = C_i$; the separation properties of these cycles follow from Lemma~\ref{lem:rballs} or Lemma~\ref{lem:rballs:planar}.
The existence of the three paths connecting $v$ and $\rcycle{\nembed}{r}{v}$ follows from the $3$-connectivity of $\MainPlusA$ and Menger's theorem.

Let us show that the treewidth of $\rdisc{\nembed}{r}{v}$ is bounded by $q'' + 10r$
and that every vertex of the said disc is within radial distance $2q''+10r$ from $v$.
This follows from the claim that the inside of every $C(w_1w_2)$ contains at most $2q''$
vertices of $\MainPlusA$.
More precisely, enumerate vertices on $C_{r+1}$ as $w_0, w_1, \ldots, w_\ell$ (in this order)
  and for every $0 \leq i < \ell$ define a bag $W_i$ consisting of all vertices of $\rdisc{\nembed}{r}{v}$
  enclosed by $C(w_iw_{i+1})$ and all vertices on the path $\sigma(w_0)$. 
  Then, $W_0, \ldots, W_{\ell-1}$ constitute a path decomposition of $\rdisc{\nembed}{r}{v}$ of width at most $2q'' + 10r$.
  
Finally, let $G'$ be the subgraph of $\MainPlusA$ comprising all vertices and edges not embedded in the interior of $\Delta$. Then $G'$ has more than $q''$ vertices, as it contains $D_u$.
Further, removing all edges and vertices from the interior of $\Delta$ decreases the face-width of $\MainPlusA$ by at most $2r$, hence the embedding of $G'$ inherited from $\MainPlusA$ has face-width at least $k'''/100-2q-k-2r$. It follows that $G'$ has Euler genus $\eulerg_\ast$.
\end{proof}

We make the following simple observation that will be implicitly used in many further reasonings.
\begin{lemma}\label{lem:eat-disc}
Let $\nembed$ be an optimal near-embedding with tidy dongles, $u,v \in V(\MainPlusA)$, and $r_u,r_v$ be such that $1 \leq r_u < r_v \leq \rmax$ and 
the radial distance between $u$ and $v$ is at most $r_v-r_u$. Then
$\rdisc{\nembed}{r_u}{u} \subseteq \rdisc{\nembed}{r_v}{v}$. 
\end{lemma}
\begin{proof}
The condition on the radial distance implies that $\rballX{\nembed}{r_u}{u} \subseteq \rballX{\nembed}{r_v}{v}$. 
Consequently, the disc enclosed by $\rcycle{\nembed}{r_v}{v}$ that contains $\rdisc{\nembed}{r_v}{v}$ encloses also
$\rballX{\nembed}{r_u}{u}$. Therefore, this disc also contains the disc enclosed by $\rcycle{\nembed}{r_u}{u}$ that contains $u$,
and hence the whole $\rdisc{\nembed}{r_u}{u}$.
\end{proof}

We will need the following statement on touching radial discs.
\begin{lemma}\label{lem:packdiscs}
Let $\nembed$ be an optimal near-embedding with tidy dongles, $u,v \in V(\MainPlusA)$, and $2 \leq r_u,r_v$, $r_u+r_v \leq \rmax$ be such that 
$\rdisc{\nembed}{r_v}{v}$ and $\rdisc{\nembed}{r_u}{u}$ are not vertex-disjoint.
Then either $\rdisc{\nembed}{r_u}{u} \subseteq \rdisc{\nembed}{r_v}{v}$, $\rdisc{\nembed}{r_v}{v} \subseteq \rdisc{\nembed}{r_u}{u}$, or the radial distance between $u$ and $v$ is at most $r_u+r_v$.
Furthermore, in all cases there exists $w \in V(G)$ such that $\rdisc{\nembed}{r_v}{v}, \rdisc{\nembed}{r_u}{u} \subseteq \rdisc{\nembed}{r_u+r_v}{w}$.
\end{lemma}
\begin{proof}
If $\rdisc{\nembed}{r_u}{u} \subseteq \rdisc{\nembed}{r_v}{v}$ or $\rdisc{\nembed}{r_v}{v} \subseteq \rdisc{\nembed}{r_u}{u}$, then we can set in the last statement $w =v$ or $w =u$, respectively.
Assume then that neither $\rdisc{\nembed}{r_u}{u} \subseteq \rdisc{\nembed}{r_v}{v}$ nor $\rdisc{\nembed}{r_v}{v} \subseteq \rdisc{\nembed}{r_u}{u}$.
In other words, the disc enclosed by $\rcycle{\nembed}{r_v}{v}$ is not contained in the disc enclosed by $\rcycle{\nembed}{r_u}{u}$ and vice-versa.
On the other hand, these discs need to intersect to accommodate a vertex in the intersection of $\rdisc{\nembed}{r_v}{v}$ and $\rdisc{\nembed}{r_u}{u}$.
This is only possible if $\rcycle{\nembed}{r_v}{v}$ and $\rcycle{\nembed}{r_u}{u}$ intersect, which implies that $u$ and $v$ are within radial distance $r_u+r_v$ from each other.
This proves the first claim. 

For the second claim, consider a shortest path $P$ from $u$ to $v$ in the face-vertex graph of $\tilde{G}_0$ and let $w$ be a vertex on this path within radial distance at most $r_v$ from $u$ and within radial
distance at most $r_u$ from $v$. Then, $\rballX{\nembed}{r_v}{v}$ and $\rballX{\nembed}{r_u}{u}$ are contained in $\rballX{\nembed}{r_u+r_v}{w}$. 
Since $\rcycle{\nembed}{r_v}{v}$ encloses the disc with $\rdisc{\nembed}{r_v}{v}$ and is contained in $\rballX{\nembed}{r_v}{v}$, and a similar statement holds for $u$ and $r_u$,
      we infer that $\rdisc{\nembed}{r_v}{v}$ and $\rdisc{\nembed}{r_u}{u}$ are contained in $\rdisc{\nembed}{r_u+r_v}{w}$, as desired.
\end{proof}
The statement of Lemma~\ref{lem:packdiscs} may seem overcomplicated at first glance, but we remark that it covers an important special case
when $\rdisc{\nembed}{r_v}{v}$ is completely contained in $\rdisc{\nembed}{r_u}{u}$, but disjoint from $\rballX{\nembed}{r_u}{u}$. Then, the radial distance from 
$v$ and $u$ can be as large as $\Omega(q'')$ (but bounded by $\Oh(q'')$ in order not to break the unbreakability assumption) and we can have $q'' \gg r_u,r_v$. 

We will need the operations of \emph{clearing} and \emph{wheeling} a radial disc $\rdisc{\nembed}{r}{v}$. 
By \emph{clearing} the radial disc $\rdisc{\nembed}{r}{v}$ we mean deleting all edges and vertices strictly inside $\rcycle{\nembed}{r}{v}$, 
i.e., the whole subgraph $\rdisc{\nembed}{r}{v}$ except for $\rcycle{\nembed}{r}{v}$. 
By \emph{wheeling} the radial disc $\rdisc{\nembed}{r}{v}$ we mean first clearing it, but leaving the vertex $v$ itself, 
and then making $v$ adjacent to all the vertices of $\rcycle{\nembed}{r}{v}$. 
By \emph{contracting the inside} of the radial disc $\rdisc{\nembed}{r}{v}$ 
we mean contracting the connected component of $\rdisc{\nembed}{r}{v} - \rcycle{\nembed}{r}{v}$
that contains $v$ onto $v$ (but keeping everything else strictly inside $\rcycle{\nembed}{r}{v}$ intact). 

If one clears or wheels a number of disjoint radial discs, then the graph keeps most of its connectivity properties, including the facewidth. This is formalized in the next lemma.
\begin{lemma}\label{lem:genus:punctured-fw}
Let $\nembed$ be an optimal near-embedding with tidy dongles.
Let $\mathcal{D}$ be a family of pairs $(v,r_v)$, $v \in V(\MainPlusA)$, $2 \leq r_v \leq \rmax$ such that
the radial discs $\rdisc{\nembed}{r_v}{v}$ are pairwise vertex disjoint. 
Let $H$ be the result of, for every $(v,r_v) \in \mathcal{D}$, either clearing or wheeling the radial disc $\rdisc{\nembed}{r_v}{v}$.
Let $f$ be the face-width of $\MainPlusA$. 
Then the face-width of the induced embedding of $H$ is at least $f - 2\sum_{(v,r_v) \in \mathcal{D}} r_v$. 
\end{lemma}
\begin{proof}
We prove the claim for $H$ being the result of clearing all discs $\rdisc{\nembed}{r_v}{v}$; the face-width in case of wheeling is clearly not smaller. 
Consider a noncontractible cycle $C$ in the face-vertex graph of $H$. 
Whenever $C$ visits a face $f_v$ enclosed by $\rcycle{\nembed}{r_v}{v}$ for some $(v,r_v) \in \mathcal{D}$, replace
the two edges $f_vu_1$ and $f_vu_2$ of $C$ incident with $f_v$ it with a shortest path in the face-vertex graph of $\MainPlusA$ between $u_1$ and $u_2$ with all internal vertices
in $\rdisc{\nembed}{r_v}{v}$. Since $u_1,u_2 \in \rballX{\nembed}{r_v}{v}$, this path has length at most $2r_v$. 
Hence, the resulting closed walk $C'$ is of length at most $|C| + 2\sum_{(v,r_v) \in \mathcal{D}} r_v$.
Furthermore, $C'$ is noncontractible as well, because $C'$ can be obtained from $C$ by local continuous transformations inside disjoint discs and hence $C$ and $C'$ are homotopic.
The lemma follows. 
\end{proof}

\begin{lemma}\label{lem:genus:clear-2conn}
Let $\nembed$ be an optimal near-embedding with tidy dongles.
Let $\mathcal{D}$ be a family of pairs $(v,r_v)$, $v \in V(\MainPlusA)$, $2 \leq r_v \leq \rmax - 1$ such that
the radial discs $\rdisc{\nembed}{r_v+1}{v}$ are pairwise vertex disjoint. 
Let:
\begin{itemize}
\item $H_1$ be created from $\MainPlusA$ by iteratively clearing discs $\rdisc{\nembed}{r_v}{v}$ for $(v,r_v) \in \mathcal{D}$;
\item $H_2$ be created from $\MainPlusA$ by iteratively wheeling discs $\rdisc{\nembed}{r_v}{v}$ for $(v,r_v) \in \mathcal{D}$;
\item $H_3$ be created from $\MainPlusA$ by contracting the inside of $\rdisc{\nembed}{r_v}{v}$
for $(v,r_v) \in \mathcal{D}$;
\item $H_1'$ be created from $\MainPlusA$ by iteratively clearing discs $\rdisc{\nembed}{r_v+1}{v}$ for $(v,r_v) \in \mathcal{D}$.
\end{itemize}
Then:
\begin{enumerate}
\item $H_1$ is 2-connected and $H_2$ and $H_3$ are both 3-connected. %
\item If
$(A,B)$ is a separation of $H_1$ of order $2$ with both $A \setminus B$ and $B \setminus A$ nonempty, then
there exists $(v,r_v) \in \mathcal{D}$ such that $A \cap B \subseteq V(\rcycle{\nembed}{r_v}{v})$ and 
either $A \setminus B$ or $B \setminus A$ is contained in $\rdisc{\nembed}{r_v+1}{v} \setminus V(\rcycle{\nembed}{r_v+1}{v})$.
\item There is a unique 3-connected component $H_1^\ast$ of $H_1$ that has more than $2q''$ vertices and
this component contains $H_1'$.
\item The treewidth of $H_1$ is larger than $2q''$. 
\end{enumerate}
\end{lemma}
\begin{proof}
Let $(A,B)$ be a separation of $H_i$ ($i \in \{1,2,3\}$) of order at most $2$ with both $A\setminus B$ and $B \setminus A$ nonempty.
Since $\MainPlusA$ is 3-connected (Lemma~\ref{lem:aligned-is-3conn}), %
there exists a connected component $C$
of $\MainPlusA - V(H_i)$ such that $N_{\MainPlusA}(C)$ intersects both $A \setminus B$ and $B \setminus A$. 
This is only possible if there exists $(v,r_v) \in \mathcal{D}$ such that $\rcycle{\nembed}{r_v}{v}$ contains both a vertex of $A \setminus B$ and a vertex of $B \setminus A$, and these two vertices are adjacent to a connected component $C$ of $\MainPlusA-V(H_i)$.

Consequently, $\rcycle{\nembed}{r_v}{v}$ contains at least two vertices of $A \cap B$. Furthermore, if $i=2$, then $v \in A \cap B$, a contradiction with the assumption $|A \cap B| \leq 2$. Also,
if $i=3$, then in $H_3$ the component $C$ is either contracted onto the vertex $v$ or is left intact, but in any case the contracted $C$ has to contain a vertex of $A\cap B$ that does not belong to $\rcycle{\nembed}{r_v}{v}$, giving at least three vertices in $A\cap B$ in total.   
This proves the first claim.

For second claim, we are left with the case $i=1$ and $A \cap B$ consists of two vertices of $\rcycle{\nembed}{r_v}{v}$. 
Then, $\rcycle{\nembed}{r_v+1}{v}$ is completely contained either in $A \setminus B$ or $B \setminus A$, and the connectivity of $\MainPlusA$ yields the second claim.

For the third claim, note that 
if $|V(H_1')| \leq 2q''$, then the treewidth of $\MainPlusA$ is bounded by $4q''+10k''$ by Lemma~\ref{lem:radial-disc} as
every connected component of $\MainPlusA-V(H_1')$ is contained in a single disc $\rdisc{\nembed}{r_v+1}{v}$. 
This is a contradiction with Lemma~\ref{lem:DKMW12} and the assumption on the treewidth of $G$.

Thus, we have that $|V(H_1')| > 2q''$. It follows that $|V(H_2)| > 2q''$, $|V(H_1)| > 2q''$, and, by the analysis of separations of order $2$ in $H_1$ presented above, there is a 3-connected component $H_1^\ast$ of $H_1$ that contains $H_1'$.
Consider now a separation $(A,B)$ of $H_1$ of order at most $2$; as discussed, we may assume that
we have $V(H_1') \subseteq A \setminus B$ while $B$ is contained in $\rdisc{\nembed}{r_v+1}{v}$ for some $(v,r_v) \in \mathcal{D}$. 
Let $A \cap B = \{v_1,v_2\}$ and let $P_\alpha$ be a shortest path in the radial graph from $v$ to $v_\alpha$ for $\alpha=1,2$; without loss of generality, assume that $P_1$ and $P_2$ share a prefix and then are vertex-disjoint. 
We have $|V(P_\alpha)| \leq r_v+1$ while 
$P_1$ and $P_2$ separate $A \setminus B$ from $B\setminus A$ in $\MainPlusA$. 
Hence, the union of $\Apices$ and the preimages of $P_1$ and $P_2$ is of size at most $k''$. 
Since $A$ contains $H_1$, by the $(q'',k'')$-unbreakability of $G$, the size
of the preimage of $B$ is at most $q''$. 
Since every vertex $\PlusDongleVtx{i} \in B$ corresponds to at least one vertex
of $V(\DongleA{i}) \setminus V(\MainA)$ in the preimage of $B$
and every vertex $\PlusVortexVtx{i} \in B$ corresponds to at least one society
vertex of $\VortexA{i}$ in $B$, we have $|B| \leq 2q''$. 

Consequently, $H_1^\ast$ is the only 3-connected component of $H_1$ that has more than $2q''$ vertices, as desired.

We proceed to the fourth claim.
Assume now that there is a tree decomposition $(T_1,\beta_1)$ of $H_1$ of width $\ell$.
For every $(v,r_v) \in \mathcal{D}$, proceed as follows.
Fix a shortest path tree $S_v$ rooted at $v$ in the face-vertex graph of $\rdisc{\nembed}{r_v}{v}$ enclosed by $\rcycle{\nembed}{r_v}{v}$. For $w \in V(\rcycle{\nembed}{r_v}{v})$, let $\sigma_v(w)$ be the path from $w$ to $v$ in $S_v$. 
For every $w \in V(\rcycle{\nembed}{r_v}{v})$ and every $t \in V(T_1)$, if $w \in \beta_1(t)$, add all vertices of $V(\MainPlusA) \cap V(\sigma_v(w))$ to $\beta_1(t)$.
Furthermore, for every $w_1w_2 \in E(\rcycle{\nembed}{r_v}{v})$, fix one node $t(w_1w_2) \in V(T_1)$ such that $w_1,w_2 \in \beta_1(t(w_1w_2))$
and create a new node $t'(w_1w_2)$, adjacent to $t(w_1w_2)$, with a bag comprising all vertices of $\MainPlusA$ that lie in the disc contained in $\rdisc{\nembed}{r_v}{v}$ enclosed by the concatenation of
$\sigma_v(w_1)$, $\sigma_v(w_2)$, and the edge $w_1w_2$. 
Similarly as argued in the previous paragraphs, the $(q'',k'')$-unbreakability of $G$ implies that the size of the bag at $t'(w_1w_2)$ is bounded by $2q''$. 

In the end, we obtain a tree decomposition of $\MainPlusA$ of width at most $\rmax \cdot \ell + 2q''$, a contradiction to Lemma~\ref{lem:DKMW12} and the assumption on the treewidth of $G$.
\end{proof}

\section{Universal apices}\label{ss:universal-apex}
For integers $a,b \geq 2$, an \emph{$(a,b)$-apex-forcer} for a vertex $v \in V(G)$ is a subgraph $H$ of $G$ consisting of:
\begin{itemize}
\item $b$ pairs $\{W_i^1, W_i^2\}_{i=1}^b$ of cylindrical $a \times a$ walls; we denote the extremal cycles of $W_i^j$ as $C_i^j$ and $D_i^j$;
\item for every $1 \leq i \leq b$ and $1 \leq j \leq 2$, a path $P_i^j$ connecting $v$ with $C_i^j$;
\item for every $1 \leq i \leq b$, a path $Q_i$ connecting $D_i^1$ with $D_i^2$.
\end{itemize}
Note that all objects above are vertex-disjoint (except for the obvious overlaps: all paths $P_i^j$ share the vertex $v$ and paths $P_i^j$ and $Q_i$ have an endpoint in $W_i^j$). 
See Figure~\ref{fig:apex-forcer} for an illustration.

\begin{figure}[tb]
\begin{center}
\includegraphics{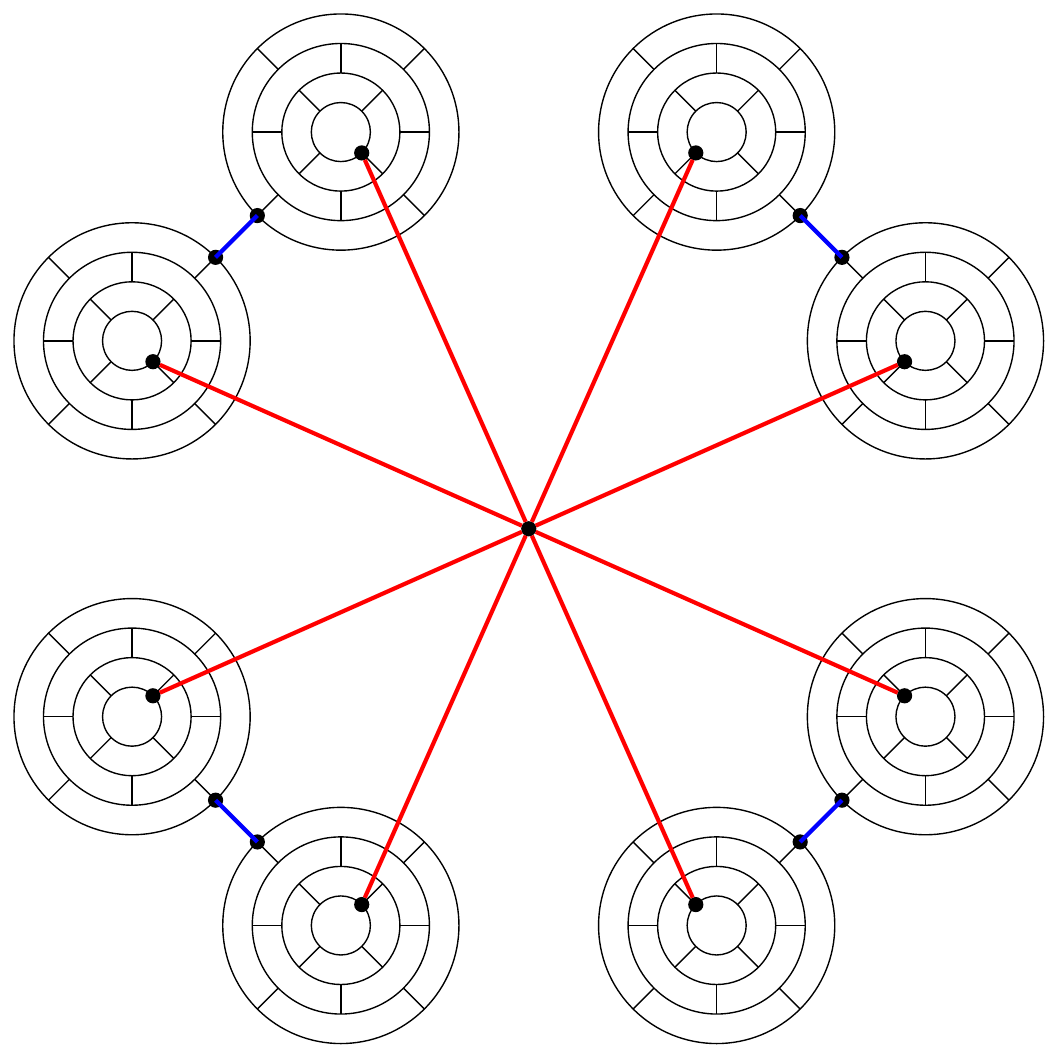}
\caption{A $(4,4)$-apex forcer. Paths $P_i^j$ are depicted in red, paths $Q_i$ are depicted in blue. Each cylindrical wall $W_i^j$ is depicted so that $C_i^j$ is the innermost cycle and $D_i^j$ is the outermost cycle.}\label{fig:apex-forcer}
\end{center}
\end{figure}

The main result of this section is the following statement, which intuitively says that a large apex forcer for $v$ is a certificate that $v$ has to be an apex in every optimal near-embedding.

\begin{lemma}\label{lem:apex-forcer}
Let $a$ and $b$ be integers satisfying
\begin{align*}
a &\geq q + 100(\eulerg_\ast + 3) + 2\WallHitSize_\ast \cdot (2\WallHitRadius_\ast+1) = 
\Oh(q^3 k^{k+9}), \\
b &\geq k/100 + \eulerg_\ast + 1 = \Oh(k).
\end{align*}
Assume that $G$ contains an $(a,b)$-apex-forcer for a vertex $v \in V(G)$. Then $v$ is an apex in every optimal near-embedding of $G$. 
\end{lemma}
\begin{proof}
Let $\nembed$ be an optimal near-embedding of $G$ where $v$ is not an apex. By Lemma~\ref{lem:tidy-dongles}, we may assume that $\nembed$ has tidy dongles. We aim at a contradiction.

Let $b' = b-k/100 \geq \eulerg_\ast + 1$ and note that for at least $b'$ values of $i \in [b]$, the objects $W_i^j$, $P_i^j$ for $j=1,2$, and $Q_i$ do not contain any apex of $\nembed$. 
Furthermore, let $a' = a - 2\WallHitSize_\ast \cdot (2\WallHitRadius_\ast+1) \geq q + 100(\eulerg_\ast + 3)$.
By Lemma~\ref{lem:cylinder-survives}, every $W_i^j$ contains an $a' \times a'$ cylindrical wall that is edge-disjoint from vortices. 
Note that both paths $Q_i$ and $P_i^j$ can be prolonged via $W_i^j$ to the corresponding extremal cycle of the said $a' \times a'$ cylindrical wall. 
Consequently, $G$ contains an $(a',b')$-apex-forcer for $v$ that does not contain an apex of $\nembed$ and, furthermore, whose every cylindrical wall is edge-disjoint with vortices of
$\nembed$. By slightly abusing the notation, we will denote the elements of this apex-forcer still with $W_i^j$, $P_i^j$, and $Q_i$, and the extremal cycles of $W_i^j$ as $C_i^j$ and $D_i^j$. 

Let $E_i^j$ be any row of $W_i^j$ that is not one of the first three rows nor one of the
last three rows. Consider $\PathProj_\nembed(E_i^j)$. 
A few remarks are in place. Recall that $E_i^j$s are edge-disjoint with vortices, so
the path projection is well-defined. 
Furthermore, $E_i^j$ has at least $a' \geq q + 100(\eulerg_\ast + 3)$ vertices, so in particular it is not contained in a single dongle and contains vertices of $\MainA$.
Consequently, $\PathProj_\nembed(E_i^j)$ corresponds to a closed curve without self-intersections 
$\gamma_i^j$ and the curves $(\gamma_i^j)_{i \in [b'], j \in [2]}$ are pairwise disjoint. 
Finally, note that since $E_i^j$ is not among of the first or last three rows of $W_i^j$,
no path within $W_i^j$ connecting a vertex of $E^i_j$ with a vertex of $C_i^j\cup D_i^j$ can be entirely contained in a single dongle of $\nembed$, for otherwise either the size of the neighborhood of this dongle would be larger than $3$, or almost the whole wall $W_i^j$ would be contained in the dongle, contradicting $a'\ \gg q$.

Obtain a path $R_i^j$ as the prolonged 
path $P_i^j$ through $C_i^j$ and $W_i^j$ to a vertex of $V(E_i^j)$, and then along $E_i^j$ to a first vertex of $V(E_i^j) \cap V(\MainA)$.
Let $R = \bigcup_{i \in [b'], j \in [2]} R_i^j$. 
Observe that, as paths $R_i^j$ share the endpoint $v$, 
$\PathProj_{\nembed}(R)$ is contained in a single connected
component $\Delta$ of $\Sigma - \bigcup_{i \in [b'], j \in [2]} \gamma_i^j$,
except for the endpoint and possibly a suffix if the first intersection of $R_i^j$ and $E_i^j$ lies in a dongle; 
then the said suffix goes along $\gamma_i^j$. 
Furthermore, the closure of $\Delta$ contains all curves $\gamma_i^j$. 
By Lemma~\ref{lem:diestelB6}, every but at most $\eulerg_\ast$ curves $\gamma_i^j$ is two-sided, with $\Delta$ on one side and a disc on the other side; furthermore, these 
discs are pairwise disjoint. 
Without loss of generality, assume that this is the case for every $1 \leq i \leq b' - \eulerg_\ast$ and $1 \leq j \leq 2$ and let us denote the corresponding disc by $\Delta_i^j$. 
Note that $b' - \eulerg_\ast \geq 1$.

Fix $1 \leq i \leq b' - \eulerg_\ast$ and $1 \leq j \leq 2$. The curve $\gamma_i^j$ splits the surface into two parts, one $\Delta_i^j$ and the other one containing $\Delta$
and all discs $\Delta_{i'}^{j'}$ for $(i,j) \neq (i',j')$. 
On the other hand, the removal of $E_i^j$ from $W_i^j$ splits $W_i^j$ into two connected components, one containing $C_i^j$ and the other containing $D_i^j$. 
Since $W_i^j$ is edge-disjoint with vortices, the projection of each part
  needs to be fully contained on one side of $\gamma_i^j$. 
Since each path $R_i^j$ contains a vertex of $C_i^j$ as an internal vertex, 
and no dongle contains both a vertex of $C_i^j$ and $E_i^j$, 
the image of the connected component of $W_i^j - V(E_i^j)$ that contains $C_i^j$ needs to be contained in the side of $\gamma_i^j$ containing $\Delta$. 

Consider now $\SubProj_{\nembed}(Q_i)$. 
Since $Q_i$ is vertex-disjoint with $W_{i'}^{j'}$ for every $1 \leq i,i' \leq b'$ and $1 \leq j' \leq 2$ except for the endpoints, and no dongle contains both a vertex of $E_i^j$ and of $D_i^j$, 
the projection $\SubProj_\nembed(Q_i)$ is disjoint with all the curves $\gamma_{i'}^{j'}$. Consequently, for every $1 \leq i \leq b'-\eulerg_\ast$ and $1 \leq j \leq 2$, $\SubProj_\nembed(Q_i)$ has its endpoint corresponding
to the endpoint of $Q_i$ in $W_i^{3-j}$ outside of $\Delta_i^j$, so it is entirely contained outside of                                                                                                                                                                                               $\Delta_i^j$. We infer that the part of $W_i^j-V(E_i^j)$ that contains $D_i^j$ needs to
be contained in the side of $\gamma_i^j$ containing $\Delta$ as well.

Since $a' \geq q + 100(\eulerg_\ast+3)$, 
Lemma~\ref{lem:project-wall} implies that $\SubProj_{\nembed}(W_i^j)$ is also a cylindrical
$a \times a$ wall in $\MainPlusA$. 
By the discussion of the previous paragraph, $\PathProj_{\nembed}(E_i^j)$ bounds a face of $\SubProj_{\nembed}(W_i^j)$whose interior is $\Delta_i^j$. 

Augment $\SubProj_{\nembed}(W_i^j)$ by adding to $W_i^j$ a vertex $x$ adjacent to all vertices of $V(E_i^j) \cap V(\MainA)$, embedding $x$ and its incident edges into $\Delta_i^j$. 
Now, the cylindrical wall $\SubProj_{\nembed}(W_i^j)$ with $x$ contains $\eulerg_\ast+3$ minor models of $K_5$ that contain $\{x\}$ as one of the branch sets, and otherwise are vertex-disjoint;
see Figure~\ref{fig:K5} for a self-explanatory figure how to construct these models. This is a contradiction with the fact that the Euler genus of the considered
surface equals $\eulerg_\ast$ and Corollary~\ref{cor:genus-manyK5s}.
\end{proof}

\begin{figure}[tb]
\begin{center}
\includegraphics[width=0.5\linewidth]{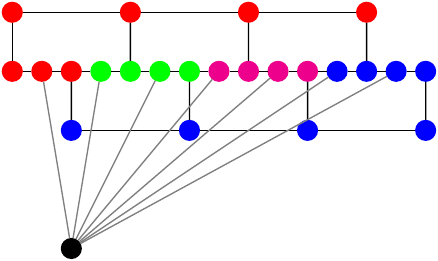}
\caption{Construction of a $K_5$ minor in Lemma~\ref{lem:apex-forcer}.
  The black vertex is the vertex $x$, the middle row belongs to $E_i^j$, the top row to $C_i^j$,
      the bottom row to $D_i^j$. The colors of vertices correspond to branch sets 
of the $K_5$ minor.}\label{fig:K5}
\end{center}
\end{figure}

Consider now a near-embedding $\nembed$ and an apex $v$ in $\nembed$. 
We define the \emph{attachment points} of $v$ as the following vertices:
\begin{itemize}
\item all neighbors of $v$ that belong to $\MainA$;
\item for every dongle $\DongleA{i}$ that contains a neighbor of $v$ in
$V(\DongleA{i}) \setminus V(\MainA)$, the vertex $\PlusDongleVtx{i}$;
\item for every vortex $\VortexA{i}$ that contains a neighbor of $v$ that is 
not a society vertex, the vertex $\PlusVortexVtx{i}$. 
\end{itemize}

To find the cylindrical walls of an apex-forcer, we will need the following variant of a well-known construction
(cf. Grigoriev's proof of a linear grid minor theorem in planar graphs~\cite{Grigoriev11}). 
\begin{lemma}\label{lem:find-cylinder}
Let $h \geq 2$ be an integer and let $H$ be a plane graph that contains:
\begin{itemize}
\item $h$ concentric cycles $C_1, C_2, \ldots, C_{h}$ (i.e., the cycles are embedded in the plane so that $C_i$ separates $C_{i'}$ from $C_{i''}$ for all $i' < i < i''$); and
\item a linkage $\mathcal{P}$ of size $2h$ from $C_1$ to $C_h$.
\end{itemize}
Then, $H$ contains an $h \times h$ cylindrical wall with $C_1$ and $C_h$ as the two extremal cycles and the remainder of the wall embedded in the ring between $C_1$ and $C_h$.
\end{lemma}
\begin{proof}
 We proceed by contradiction, so let us suppose that $H$ is a counterexample to the lemma statement with a minimal number of edges. By minimality, every edge of $H$ participates in a cycle $C_i$ or in a path $P\in \cal P$. Further, we may assume that each path $P\in {\cal P}$ starts at a vertex of $C_1$, ends at a vertex of $C_h$, and otherwise is disjoint with $C_1\cup C_h$. In particular, $P$ is entirely contained in the ring in the plane between $C_1$ and $C_h$. 
 
Also, by minimality one cannot find cycles $C_2',\ldots,C_{h-1}'$ such that $C_1,C_2',\ldots,C_{h-1}',C_h$ is a concentric family of vertex-disjoint cycles and there is an edge $e$ that belongs to neither of the cycles $C_1,C_2',\ldots,C_{h-1}',C_h$ nor to any of the paths $P\in \cal P$. Indeed, then removing the edge $e$ from $H$ and replacing cycles $C_2,\ldots,C_{h-1}$ with $C_2',\ldots,C_{h-1}'$ would yield a counterexample with a smaller number of edges. Similarly, there is no linkage $\cal P'$ from $C_1$ to $C_h$ of size $2h$ for which there would be an edge $e$ that would belong neither to any of the cycles $C_1,\ldots,C_h$ nor to any of the paths $P\in \cal P'$. 
 
 The main combinatorial observation is embedded in the following claim.
 
 \begin{claim}\label{cl:connected-intersection}
  For every $P\in {\cal P}$, $P\cap C_2$ is a subpath of $P$.
 \end{claim}
 \begin{proof}
 For any $P\in {\cal P}$, an {\em{excursion}} on $P$ is a subpath of $P$ that has both endpoints on $C_2$ but otherwise is vertex-disjoint and edge-disjoint with $P$. Observe that the claim is equivalent to verifying that there are no excursions on any $P\in \cal P$. Note also that every excursion is either entirely contained in the ring between $C_1$ and $C_2$, or entirely contained in the ring between $C_2$ and $C_h$. Excursions of these two types will be called {\em{outbound}} and {\em{inbound}}, respectively.
 
 First, we observe that there are no outbound excursions on any $P\in \cal P$. Suppose otherwise: let $E$ be an outbound excursion on $P\in \cal P$, say with endpoints $a,b\in V(C_2)$. Then $E$ together with one of the two paths into which $C_2$ is split by $a$ and $b$ forms a cycle $C_2'$ that separates $C_1$ from $C_3,\ldots,C_h$. Since $E$ is an outbound excursion, $E$ is vertex-disjoint with $C_1,C_3,\ldots,C_h$, hence $C_1,C_2',C_3,\ldots,C_h$ form a concentric family of vertex-disjoint cycles. Now the edge of $E(C_2)\setminus E(C_2')$ incident to $a$ is neither on any of the cycles $C_1,C_2',\ldots,C_h$ nor on any of the paths $P\in \cal P$. This is a contradiction with the minimality of the counterexample.
 
 Second, we observe that there are no outbound excursions. Suppose otherwise: let $E$ be an outbound excursion on $P\in \cal P$. Let $\Delta$ be the disc enclosed by $C_2$, that is, the part of the plane with $C_2$ removed that contains $C_h$. Then $E$ splits $\Delta$ into two discs, and let $\Delta'$ be the one that does not contain $C_h$. We now observe that no path $Q\in \cal P$, $Q\neq P$, intersects $C_2\cap \partial \Delta'$, for this would require the existence of an outbound excursion on $Q$, which we excluded in the previous paragraph. Consequently, if we construct $P'$ from $P$ by replacing the excursion $E$ with $C_2\cap \partial \Delta'$, and obtain $\cal P'$ from $\cal P$ by replacing $P$ with $P'$, then $\cal P'$ is a linkage of size $2h$ from $C_1$ to $C_h$. However, now the edges of $E$ incident to the endpoints of $E$ belong neither to the cycles $C_1,\ldots,C_h$ nor to any path in $\cal P'$. This is a contradiction with the minimality of the counterexample.
 \cqed\end{proof}
 
 Having Claim~\ref{cl:connected-intersection} in place, we can apply the same reasoning to cycles $C_2,C_3,\ldots,C_h$ and the linkage $\cal Q$ from $C_2$ to $C_3$ obtained from $\cal P$ by restricting every $P\in \cal P$ to a suffix starting from the last vertex of $V(P)\cap V(C_2)$ on $P$. Thus, we infer that for every $P\in \cal P$, $E(C_3)\cap E(P)$ is a subpath of $P$. Proceeding with this reasoning further in an analogous manner, we conclude the following.

 \begin{claim}\label{cl:connected-intersection2}
  For every $P\in \cal P$ and $i\in \{1,\ldots,h\}$, $P\cap C_i$ is a subpath of $P$.
 \end{claim}

 Note that by planarity, for every $P\in \cal P$ the subpaths $P\cap C_1,P\cap C_2,\ldots,P\cap C_h$ must appear in this order on $P$. Also, we can enumerate the paths of $\cal P$ as $P_1,\ldots,P_{2h}$ so that the paths $P_1\cap C_i,P_2\cap C_i,\ldots,P_{2h}\cap C_i$ appear in this cyclic order on the cycle $C_i$, for each $i\in \{1,\ldots,h\}$. It now remains to observe that one can obtain a subdivision $W$ of an $h\times h$ cylindrical wall as follows:
 \begin{itemize}
  \item For $i,j\in \{1,\ldots,h\}$, pick $v_{i,j}$ to be the first or last vertex of $C_i\cap P_j$ on $P_j$, depending on whether $i+j$ is odd or even.
  \item The rows of the wall are cycles $C_1,\ldots C_h$.
  \item The remaining edges of the wall, that is edges $v_{i,j}v_{i+1,j}$ where $i+j$ is even, are modelled as subpaths of $P_i$ between $v_{i,j}$ and $v_{i+1,j}$.
 \end{itemize}
 Thus, $W$ is a subdivision of a cylindrical wall that satisfies all the required properties.
\end{proof}

The next lemma shows that if an apex has many attachment points that are sufficiently spread
in $\MainPlusA$, then there is an apex-forcer for the said apex.
\begin{lemma}\label{lem:exists-apex-forcer}
Let $\nembed$ be an optimal near-embedding with tidy dongles and let $v$ be an apex in $\nembed$.
Let $a \geq b \geq 2$ be integers, let $r \coloneqq 3(a+4)+1 + 2q'$ and assume that
\[ 2r+11 < k''/(200(q+1)) - k \quad \mathrm{and} \quad (3a+1)(q+1) + k < k'. \]
Furthermore, assume there  exists a set $L$ of attachment points of $v$ of size at least
$\nvortices_\ast + 100(b + \nvortices_\ast + \eulerg_\ast)$ 
such that the radial discs $\rdisc{\nembed}{r+5}{w}$ for $w \in L$ are pairwise disjoint.
Then $G$ contains an $(a,b)$-apex-forcer for $v$. 
\end{lemma}
\begin{proof}
Consider radial discs $\rdisc{\nembed}{i}{w}$ for $w \in L$ and $1 \leq i \leq r+4$. 
By Lemma~\ref{lem:genus:clear-2conn},
clearing the discs $\{\rdisc{\nembed}{r+4}{w}~|~w \in L\}$ leaves a 2-connected subgraph $H'$ of $\MainPlusA$ that is contained in one connected
component $H$ of the graph $\MainPlusA - \bigcup_{w \in L} V(\rdisc{\nembed}{r}{w})$.

Note that for every vortex $\VortexA{i}$, for at most one $w \in L$ the radial disc $\rdisc{\nembed}{r+2}{w}$ may contain a society vertex of $\VortexA{i}$ or the vertex $\PlusVortexVtx{i}$. 
Hence, we can restrict $L$ to only those $w \in L$ for which $\rdisc{\nembed}{r+2}{w}$
does not contain any society vertex nor a vertex $\PlusVortexVtx{i}$; we have $|L| \geq 100(b+\eulerg_\ast)$. 

Let $a' = 3a+1$. 
Fix $w \in L$ and consider cycles $\rintcycle{\nembed}{r}{i}{w}$ (see Lemma~\ref{lem:radial-disc}).
The first $q'$ cycles and the last $q'$ cycles are separated by all cycles $\rintcycle{\nembed}{r}{i}{w}$, $q' < i \leq q' + a'$. 
Furthermore, every cycle $\rintcycle{\nembed}{r}{i}{w}$ contains at least one vertex of $\MainA$. 
If $\MainPlusA$ contains a separation of order less than $a'$ that separates $\rintcycle{\nembed}{r}{q'+1}{w}$ from $\rintcycle{\nembed}{r}{q'+a'}{w}$, then 
$\MainPlusA$ contains a set $X \subseteq V(\MainPlusA)$ of size less than $a'$ that separates all cycles $\rintcycle{\nembed}{r}{i}{w}$ for $i \leq q'+1$ from all cycles $\rintcycle{\nembed}{r}{i}{w}$ for $i \geq q'+a'$.
Then, $\preimage(X) \cup \Apices$ is a separator of size less than $a'(q+1)+k < k'$
in $G$ that has more than $q'$ vertices of $\MainA$ on both sides (as each cycle $\rintcycle{\nembed}{r}{i}{w}$ contains at least one vertex of $\MainA$),
a contradiction to the $(q',k')$-unbreakability of $G$. 

Hence, $\MainPlusA$ contains a family $\mathcal{P}_w$ of $a'$ vertex-disjoint paths from $\rintcycle{\nembed}{r}{q'+1}{w}$ to $\rintcycle{\nembed}{r}{q'+a'}{w}$.
Without loss of generality, assume that only the endpoints of paths of $\mathcal{P}_w$ lie on those two cycles; in particular, these paths are embedded in the ring between them. 
Lemma~\ref{lem:find-cylinder} implies that the subgraph of $\MainPlusA$
consisting of $\mathcal{P}_w$ and cycles $\rintcycle{\nembed}{r}{q'+i}{w}$ for $i \in [a']$ 
contains an $a \times a$ cylindrical wall $W_w$
with $C_w \coloneqq \rintcycle{\nembed}{r}{q'+1}{w}$
and $D_w \coloneqq \rintcycle{\nembed}{r}{q'+a}{w}$ as external cycles. 
Note that, as $\rdisc{\nembed}{r+2}{w}$ does not contain any society vertex of a vortex
nor a vertex $\PlusVortexVtx{i}$, $W_w$ does not contain any virtual vortex edges.

Now consider the graph $G'$ defined as $\MainPlusA$ after contracting
the insides of all radial discs $\rdisc{\nembed}{q'+a+1}{w}$ for $w \in L$
and all radial discs $\rdisc{\nembed}{2}{\PlusVortexVtx{i}}$ for $i \in [\nvortices_\ast]$;
note that these radial discs are pairwise vertex-disjoint
due to the assumption that $\rdisc{\nembed}{r+4}{w}$ does not contain
any vertex $\PlusVortexVtx{i}$ for $w \in L$ and
due to Lemma~\ref{lem:opt-ne-rep}. 
Lemma~\ref{lem:genus:clear-2conn} asserts that $G'$ is 3-connected.
By Corollary~\ref{cor:gallai}, $G'$ admits a family $\mathcal{Q}_0$ of
$b+\nvortices_\ast$ vertex-disjoint paths with endpoints in $L$. 
Let $\mathcal{Q} \subseteq \mathcal{Q}_0$ be a subfamily of $b$ paths
that do not contain any vertex $\PlusVortexVtx{i}$.
A path $Q \in \mathcal{Q}$ connecting $w_1$ and $w_2$ can be lifted
to a path $Q'$ between corresponding cycles $D_{w_1}$ and $D_{w_2}$
in $\MainPlusA$ that (because we contracted $\rdisc{\nembed}{2}{\PlusVortexVtx{i}}$ for $i \in [\nvortices_\ast]$) does not pass through any society vertex or any vertex $\PlusVortexVtx{i}$.

Recall that $\nembed$ has tidy dongles, in particular, it has properly connected dongles.
Furthermore, we have now that all objects $W_w$, $\mathcal{P}_w$, and $Q'$ for $Q \in \mathcal{Q}$
do not contain society vertices nor vertices $\PlusVortexVtx{i}$. 
By Lemma~\ref{lem:lift-wall}, any lift of $W_w$ is an $a \times a$ wall. Note
that (again thanks to the fact that dongles are properly connected) $P_w$ can be lifted to a path
from $v$ to the extremal cycle of $W_w$ that corresponds to $C_w$, while
for any $Q \in \mathcal{Q}$, $Q'$ with endpoints on $D_{w_1}$ and $D_{w_2}$
can be lifted to a path in $G$ between the extremal cycle of the lift of $W_{w_1}$ that corresponds
to $D_{w_1}$ and the extremal cycle of the lift of $W_{w_2}$ that corresponds to $D_{w_2}$. 
Furthermore, the fact that neither of the lifted objects contain a society vertex
nor a $\PlusVortexVtx{i}$ vertex ensures that the lifts are vertex-disjoint (except for the
obvious required overlaps).
This gives the desired $(a,b)$-apex-forcer for $v$. 
\end{proof}

Let $\nembed$ be an optimal near-embedding and let $v$ be an apex of $\nembed$.
The vertex $v$ is an \emph{universal apex} if \emph{for every} optimal near-embedding
$v$ is an apex, and a \emph{local apex} otherwise. 
We have the following corollary of Lemmas~\ref{lem:apex-forcer} and~\ref{lem:exists-apex-forcer}.

\begin{corollary}\label{cor:univ-apex}
There are integers
\begin{align*}
\AttCoverSize &\in \Oh(\napices_\ast(k + \eulerg_\ast + 1)), \\
\AttCoverRadius &\in \Oh(q'+\eulerg_\ast + \WallHitSize \cdot \WallHitRadius + 1)
\end{align*}
such that the following holds. 
Let $\nembed$ be an optimal near-embedding with tidy dongles.
Then there exists a set $L \subseteq V(\MainPlusA)$ of size 
at most $\AttCoverSize$ 
such that every attachment point of a local apex of $\nembed$ lies
in $\rdisc{\nembed}{\AttCoverRadius}{w}$ for some $w \in L$.
\end{corollary}

Note that, if we choose $\funstep$ sufficiently quickly-growing, we have
\[ \AttCoverSize = \Oh(k^2) \quad\mathrm{and}\quad \AttCoverRadius = \Oh(q'). \]

\section{Proximity of attachment points}\label{ss:proximity}
We say that an optimal near-embedding $\nembedA$ with tidy dongles has \emph{minimal vortices}
if there is no other optimal near-embedding $\nembedB$ with tidy dongles and exactly the same set of apices
such that $\bigcup_{i \in [\nvortices_\ast]} V(\VortexB{i}) \subsetneq \bigcup_{i \in [\nvortices_\ast]} V(\VortexA{i})$. 
In this section we prove the following structural result about optimal near-embeddings, which is the cornerstone of our approach. The result intuitively says that if we look at two optimal near-embedding with tidy dongles, say $\nembedA$ and $\nembedB$ where $\nembedB$ has minimal vortices, then the attachment points of apices of $\nembedB$ that are not apices of $\nembedA$ are somewhat local in $\nembedA$: they can be covered by a bounded number of radial discs of bounded radius. Moreover, this cover can be defined based only on $\nembedA$ so that it works for every possible other embedding $\nembedB$.
Later, this statement will allow us to relate different optimal embeddings with each other, as it shows that in terms of apices, they may  differ only locally and in a controlled manner.

Recall that we use the convention that objects related to embedding $\nembedA$ are named $\Apices$, $\MainPlusA$, etc., while objects related to embedding $\nembedB$ are named $\ApicesB$, $\MainPlusB$, etc.

\begin{lemma}\label{lem:proximity}
There are integers
\begin{align*}
\ProxCoverSize &= \AttCoverSize + \nvortices_\ast = \Oh(\napices_\ast(k + \eulerg_\ast+1) + \nvortices_\ast) = \Oh(k^2), \\
\ProxCoverRadius &= 4(\AttCoverRadius + 
 3(\ProxCoverSize + \napices_\ast+1)(q+\AttCoverRadius+3\\
     &\qquad +4f_{\mathrm{fwf}}(\Sigma(\eulerg_\ast,\eulerg_\ast,\ProxCoverSize + \napices_\ast+1),\\
     &\qquad\qquad\qquad 5+q+2f_{\mathrm{ue}}(\eulerg_\ast) + (2q+7)\nvortices_\ast + \WallHitC(\eulerg_\ast+1)^2\WallHitSize_\ast(2\WallHitRadius_\ast+1)))) \\
 &\quad = \Oh(q'k^2) \qquad\qquad \textrm{(assuming }\funstep\textrm{ is sufficiently quickly growing)}
\end{align*}
such that the following holds. 
For every optimal near-embedding $\nembedA$ of $G$ with tidy dongles, 
there exists a set $R \subseteq V(\MainPlusA)$ of size at most $\ProxCoverSize$ such that
for every optimal near-embedding $\nembedB$ that has tidy dongles and minimal vortices,
all the following vertices lie in $\bigcup_{w \in R} \rdisc{\nembed}{\ProxCoverRadius}{w}$:
\begin{enumerate}
\item every attachment point of a local apex of $\nembedA$;
\item every vertex $\PlusVortexVtx{i}$ and society vertex of $\nembedA$;
\item the vertex $\SubProj_{\nembedA}(v)$ for every vertex $v \in (\ApicesB \cup \bigcup_{j \in [\nvortices_\ast]} V(\VortexB{j})) \setminus \ApicesA$.
\end{enumerate}
\end{lemma}
\begin{proof}
We invoke Corollary~\ref{cor:univ-apex} on $\nembedA$, obtaining a set $L$, and add the vertices $\{\PlusVortexVtx{i}~|~i \in [\nvortices_\ast]\}$
to $L$. 
Clearly, $|L| \leq \AttCoverSize + \nvortices_\ast$. By Corollary~\ref{cor:univ-apex},
every attachment point of a local apex of $\nembedA$ and every 
society vertex of $\nembedA$ lies in $\bigcup_{w \in L}\rdisc{\nembed}{\AttCoverRadius}{w}$.
The challenge is to show that, for any choice of $\nembedB$ as in the lemma statement, 
the apices and vertices of vortices
of $\nembedB$ are either apices in $\nembedA$ or (their projections in $\nembedA$) also lie close to the aforementioned radial balls.

To this end, observe the following.
\begin{claim}\label{cl:proximity-finish}
Assume that for every 
\[ v \in \left(\ApicesB \cup \bigcup_{j \in [\nvortices_\ast]} V(\VortexB{j})\right) \setminus \ApicesA \]
there exists $u \in L$ and $w \in V(\MainPlusA)$ such that both $\SubProj_{\nembedA}(v)$ and $u$ lie in $\rdisc{\nembedA}{\ProxCoverRadius/4}{w}$.
Then for every $u \in L$ there exists $w(u) \in V(\MainPlusA)$ such that $R \coloneqq \{w(u)~|~u \in L\}$ satisfies the statement of the lemma.
\end{claim}
\begin{proof}
For every $u \in L$ let $W(u) \subseteq V(\MainPlusA)$ be the family of vertices $w$
such that $\rdisc{\nembedA}{w}{\ProxCoverRadius/4}$ contains $u$ and is inclusion-wise maximal. 
Then, Lemma~\ref{lem:packdiscs} implies that the radial distance (in $\MainPlusA$)
between any two vertices of $W(u)$ is at most $\ProxCoverRadius/2$ and then Lemma~\ref{lem:eat-disc}
implies that for every $w,w' \in W(u)$, the radial disc
$\rdisc{\nembed}{w}{\ProxCoverRadius}$ contains the radial disc $\rdisc{\nembed}{w'}{\ProxCoverRadius/4}$. Hence, any choice $w(u) \in W(u)$ for every $u \in L$ satisfies the statement of the claim.
\cqed\end{proof}

Assume the contrary; let
$v_0$ be such that 
\begin{equation}\label{eq:v0assumption}
v_0 \in \left(\ApicesB \cup \bigcup_{j \in [\nvortices_\ast]} V(\VortexB{j})\right) \setminus \ApicesA \qquad \mathrm{and} \qquad \neg \exists_{w \in R} \exists_{u \in L} 
\{\SubProj_\nembedA(v_0),u\} \subseteq \rdisc{\nembed}{\ProxCoverRadius/4}{w}.
\end{equation}
Our goal is to reach a contradiction by showing that $\nembedB$ cannot be optimal
with minimal vortices. Intuitively, the ``area around $v_0$'' can be redrawn using a part of the embedding $\nembedA$ so that either the number of apices  strictly decreases, or the set of vertices
in vortices gets strictly smaller. 

Assume $\Sigma = \Sigma(a,b)$. 
Set 
\[\mu := 2f_{\mathrm{fwf}}(\Sigma(a,b,|L|+\napices_\ast+1), h_1)\]
where $f_{\mathrm{fwf}}$ comes from Theorem~\ref{thm:GM7} and
\[ h = 5+q+2f_{\mathrm{ue}}(\eulerg_\ast), \qquad h_3 = h + (2q+7)\nvortices_\ast, \qquad h_1 = h_3 + \WallHitC(\eulerg_\ast+1)^2\WallHitSize_\ast(2\WallHitRadius_\ast+1). \]
Here, $f_{\mathrm{ue}}$ comes from Theorem~\ref{thm:fw-unique}
and $\WallHitC$ comes from Lemma~\ref{lem:Sigma-wall-survives}.

Consider now the following process. 
Start with $R_0 \coloneqq L \cup \{\SubProj_{\nembedA}(u)~|~u \in \{v_0\} \cup \ApicesB \setminus \ApicesA\}$ and assign radius $r_u \coloneqq \AttCoverRadius + q + 3$ for every $u \in R_0$.
Exhaustively, apply one of the steps below if applicable.
\begin{itemize}
\item If there exists $u \in R_0$ such that for some $v \in V(\MainPlusA)$ we have that the radial disc $\rdisc{\nembed}{r_u}{u}$ is a proper subset of
the radial dist $\rdisc{\nembed}{r_u}{v}$, replace $u$ with $v$ in $R_0$ and set $r_v \coloneqq r_u$. 
\item If there exist $u_1,u_2 \in R_0$ such that the radial discs $\rdisc{\nembed}{r_{u_1}+\mu}{u_1}$ and $\rdisc{\nembed}{r_{u_2}+\mu}{u_2}$ are not disjoint,
      apply Lemma~\ref{lem:packdiscs} to these two discs, obtaining a vertex~$w$;
set $R_0 \coloneqq (R_0 \setminus \{u_1,u_2\}) \cup \{w\}$ and set $r_w = r_{u_1}+r_{u_2} + 2\mu$. 
\end{itemize}
Let $R$ be $R_0$ at the end of the process.

Observe that when the process ends, 
we have 
\[ \sum_{w \in R} r_w \leq (q+\AttCoverRadius+3+2\mu)(|L|+\napices_\ast+1) \leq \ProxCoverRadius/4 = \Oh(q'k^2) \]
assuming $\funstep$ is sufficiently quickly growing. 
In particular, no $r_w$ ever exceeds $\rmax/2$ and the application of Lemma~\ref{lem:packdiscs} is justified. 
The assumption~\eqref{eq:v0assumption} implies
that $\SubProj_{\nembedA}(v_0)$ is in a radial disc $\rdisc{\nembedA}{r_w}{w}$ for some $w \in R$
that does not contain any vertex of $L$. 

Let $H_1$ be the graph $\MainPlusA$ with all radial discs $\rdisc{\nembedA}{r_u}{u}$ for $u \in R$
cleared.
Furthermore, let $\Sigma^\bullet$ be $\Sigma$ with a small cuff cut out around every vertex $u \in R$
so that this cuff is completely inside the face of $H_1$ that contains $u$ and
$H_1$ is embedded into $\Sigma^\bullet$ without any vertex on a cuff. 

We now make the following construction.
\begin{claim}\label{cl:proximity:W}
$G$ contains a $\Sigma$-wall $W$ of order $h$ with the following properties:
\begin{itemize}[nosep]
\item $W$ is vertex-disjoint with apices and vortices of both $\nembedA$ and $\nembedB$;
\item $\SubProj_\nembed(W)$ is a $\Sigma^\bullet$-wall of order $h$ with the natural embedding that is a subgraph of $H_1$;
\item in $\nembedA$, every vertex of $\SubProj_{\nembed}(W)$ is within radial distance at least $q + 3$ from any virtual vortex vertex, any attachment point of a local apex, and any projection
of a vertex of $\ApicesB \setminus \ApicesA$;
\item $\SubProj_{\nembedB}(W)$ is a $\Sigma$-wall of order $h$ with the natural embedding;
\item in $\nembedB$, every vertex of $\SubProj_{\nembedB}(W)$ is within radial distance at least $q+3$ from any virtual vortex vertex;
\item all degree-3 vertices of $W$ are contained in both $G(\ApicesA)$ and $G(\ApicesB)$ (that is, $W$ projects to a $\Sigma$-wall $W(\ApicesA)$ of order $h$ in $G(\ApicesA)$
 and to a $\Sigma$-wall $W(\ApicesB)$ of order $h$ in $G(\ApicesB)$). 
\end{itemize}
\end{claim}
\begin{proof}
Lemma~\ref{lem:genus:punctured-fw} asserts that $H_1$ is still embedded
in $\Sigma^\bullet$ with high facewidth, namely at least 
\[ k'''/100-2q-k - 2\sum_{w \in R} r_w \geq k'''/100-2q-k - 2(\AttCoverRadius+3)(\AttCoverSize+1).\]
By setting $\funstep$ sufficiently quickly-growing, we can assume that that the said facewidth
is at least $\mu$.
By the stopping condition of the process, the radial distance in $H_1$ between distinct cuffs
is at least $2\mu$.
As $\Sigma^\bullet$ has at most $|L|+\napices_\ast+1$ cuffs, Theorem~\ref{thm:GM7} implies that 
$H_1$ contains a $\Sigma^\bullet$-wall $W_1$ of order $h_1$ as a surface minor. 

Note that $W_1$ does not contain any society vertex of $\nembedA$ nor any vertex $\PlusVortexVtx{i}$. 
Recall also that $\nembedA$ has tidy dongles, in particular, it has properly connected dongles.
Hence, Lemma~\ref{lem:lift-wall} implies that $G-\ApicesA$ contains a 
$\Sigma$-wall $W_2$ of order $h_1$ with $\SubProj_{\nembedA}(W_2) = W_1$.

Lemma~\ref{lem:Sigma-wall-survives} (and Lemma~\ref{lem:sigma-wall-holes}; recall that Lemma~\ref{lem:sigma-wall-holes} accommodates surfaces with a boundary)
implies in turn that $W_2$ contains a $\Sigma$-wall $W_3$
of order $h_3$ that does not contain any apex of $\ApicesB$ nor any edge of a vortex of $\nembedB$
and such that $\SubProj_{\nembedA}(W_3)$ is in fact a $\Sigma^\bullet$-wall of
order $h_3$ embedded naturally into $\Sigma^\bullet$
(i.e., informally speaking, the cuffs of $\Sigma^\bullet$, that correspond
to the vertices of $R$, lie in the analogous places of $\SubProj_{\nembedA}(W_3)$
as they lie in $W_1$). 

By the choice of $h_3$, the embedding of $\SubProj_{\nembedA}(W_3)$ into $\Sigma$ is of facewidth
larger than $f_{\mathrm{ue}}(\eulerg_\ast)$. By Theorem~\ref{thm:fw-unique},
it is an embedding of this graph of minimum Euler genus and is the unique embedding
into $\Sigma$. 

By Lemma~\ref{lem:project-wall}, $\WB_3 \coloneqq \SubProj_{\nembedB}(W_3)$ is a $\Sigma$-wall of order $h_3$.
By Theorem~\ref{thm:fw-unique}, the embedding of $\WB_3$ in $\nembedB$ is the natural
embedding of a $\Sigma$-wall of order $h_3$ into $\Sigma$.
By the choice of $h_3$ and Lemma~\ref{lem:sigma-wall-holes}, $\WB_3$ contains a $\Sigma$-wall $\WB$ of order $h$ as a subwall,
with the natural embedding, such that every vertex of $\WB$ is within radial distance
larger than $q+3$ from any virtual vortex vertex of $\nembedB$. 
By Lemma~\ref{lem:lift-wall}, $W_3$ contains a $\Sigma$-wall $W$ of order $h$ 
with $\SubProj_{\nembedB}(W) = \WB$. 
Again, $\SubProj_{\nembedA}(W)$ is not only a $\Sigma$-wall of order $h$ with its natural
embedding, but also a $\Sigma^\bullet$-wall of order $h$ with its natural embedding.

The above verifies that $W$ satisfies all required properties but the last one.
The last property follows in turn from $h > q$ and Lemma~\ref{lem:lift-wall}.
\cqed\end{proof}

We now leverage how we distantiated $\WB$ from the apices of both $\nembedA$ and $\nembedB$ in the following claim. In what follows, $\Predongles$ and $\PredonglesB$ are the families
considered in the definition of tidy dongles (Section~\ref{ss:tidy-dongles}), constructed
 for graphs $G-\ApicesA$ and $G-\ApicesB$, respectively,
 and the corresponding unbreakability tangle.
\begin{claim}\label{cl:prox:predongles-the-same}
The following holds:
\[V(W(\Apices)) \setminus \bigcup_{D \in \Predongles} D = 
  V(W(\ApicesB)) \setminus \bigcup_{\widetilde{D} \in \PredonglesB} \widetilde{D}. \]
\end{claim}
\begin{proof}
Let $\Predongles_0$ be the family of maximal $\mathcal{T}_{\Apices}$-predongles in $G(\Apices)$ and
let $\PredonglesB_0$ be the family of maximal $\mathcal{T}_{\ApicesB}$-predongles in $G(\ApicesB)$.

Consider an element $D_0 \in \Predongles_0$ such that $N_{G(\Apices)}[D_0]$ contains a vertex $v_1 \in W(\Apices)$. Let $X_0 = N_{G(\Apices)}(D_0)$; note that $|X_0| = 3$. 
Let $G_1 = G(\Apices)[D_0 \cup X_0]$
and let $G_2$ be the preimage of $G_1$ in the graph $G$, that is: start with $G_2=G[V(G_1)]$ and for every connected component $C$ of $G-\Apices-V(G(\Apices))$ whose neighborhood in $G-\Apices$ is nonempty and contained in $V(G_1)$, add $C$ and all edges between $C$ and $V(G_1)$ to $G_2$. 
By the definition of $\Predongles_0$, 
\[ (N_G[V(G_2) \setminus X_0], N_G[V(G) \setminus N_G[V(G_2) \setminus X_0]]) \]
is a separation of order at most $k$ with $N_G[V(G_2) \setminus X_0] \subseteq V(G_2) \cup \Apices$ being the side of size at most $q$.

By the 3-connectivity of $G(\Apices)$, $G_1$ is connected. 
By the definition of $G(\Apices)$, $G_2$ is connected as well. 

Recall that the radial distance in $\nembedA$ between any vertex of $\SubProj_{\nembedA}(W)$
and an attachment point of a local apex or a virtual vortex vertex 
of $\nembedA$ is larger than $q+1$. 
By the connectivity of $G_2$, $G_2$ is edge-disjoint with vortices of $\nembedA$.
Hence, Lemma~\ref{lem:small-subgraph-is-local} implies that 
$G_2$ does not contain any vertex $u$ whose projection $\SubProj_{\nembedA}(u)$ is an
attachment point of a local apex of $\nembedA$. 
Similarly, since the radial distances between projections of vertices of $\ApicesB \setminus \ApicesA$ and $\SubProj_{\nembedA}(W)$ are larger than $q+1$ in $\nembedA$, 
Lemma~\ref{lem:small-subgraph-is-local} implies that 
$G_2$ does not contain any
vertex $u \in \ApicesB \setminus \ApicesA$ nor its neighbor.

Moreover, we 
claim that 
there is no separation
$(A,B)$ of $G-\ApicesB$ of order at most $2$ with $|A| \leq q$ and $V(G_2) \subseteq A$.
Assume the contrary. Without loss of generality, assume that $G[A]$
is connected. 
From Lemma~\ref{lem:small-subgraph-is-local}
applied in $\nembedA$ to the connected component of $G[A]-\ApicesA$ that contains $G_2$,
we infer that $A$ does not contain any vertex of $\ApicesA$.
Then, again from Lemma~\ref{lem:small-subgraph-is-local},
applied in $\nembedA$ to $G[A]$ we infer that $G[A]$ 
does not contain any vertex of $\ApicesB \setminus \ApicesA$ in its closed neighborhood.
Hence, $A$ is also a small side of a separation of order at most $2$ in $G-\Apices$,
a contradiction to the fact that it contains (in $D_0$) a vertex of $G(\Apices)$.

We infer that 
\[ V(G_2) \cap \ApicesB = \emptyset \qquad \mathrm{and} \qquad N_G(V(G_2) \setminus X_0) \subseteq X_0 \cup (\ApicesB \cap \ApicesA). \]
Hence, $G_1$ is also a subgraph of $G(\ApicesB)$ and thus there exists $\widetilde{D}_0 \in \PredonglesB_0$
with $D_0 \subseteq \widetilde{D}_0$. Furthermore, as $G_1$ is a subgraph of $G(\ApicesB)$,
     we infer that $v_1 \in V(W(\ApicesB))$ and $v_1 \in N_{G(\ApicesB)}[\widetilde{D}_0]$, that
     is, $N_{G(\ApicesB)}[\widetilde{D}_0] \cap V(W(\ApicesB)) \neq \emptyset$.

A symmetric argument proves that for every $\widetilde{D}_0 \in \PredonglesB_0$ such that
$N_{G(\ApicesB)}[\widetilde{D}_0]$ contains a vertex of $W(\ApicesB)$ there exists
$D_0 \in \Predongles_0$ 
such that 
$N_{G(\ApicesA)}[D_0]$ contains a vertex of $W(\ApicesA)$ 
and with $\widetilde{D}_0 \subseteq D_0$. 
As families $\Predongles_0$ and $\PredonglesB_0$ both contain pairwise disjoint sets
of vertices, we infer that 
\[ \{D_0 \in \Predongles_0~|~N_{G(\Apices)}[D_0] \cap V(W(\Apices)) \neq \emptyset\} = 
   \{\widetilde{D}_0 \in \PredonglesB_0~|~N_{G(\ApicesB)}[\widetilde{D}_0] \cap V(W(\ApicesB)) \neq \emptyset\}. \]
The claim follows from the definition of $\Predongles$ and $\PredonglesB$.
\cqed\end{proof}
Let
\[ Z \coloneqq V(W(\Apices)) \setminus \bigcup_{D \in \Predongles} D. \]
Since every $D \in \Predongles$ satisfies $|N_{G(\Apices)}(D)|=3$
and the elements of $\Predongles$ are pairwise disjoint and nonadjacent in $G(\Apices)$, 
we infer that every maximal path in $W(\Apices)$ between two degree-3 vertices contains at least
one vertex of $Z$. Thus, $W(\Apices)-Z$ is a forest of paths and subdivided claws, 
and similarly $W(\ApicesB)-Z$. 

Furthermore, since both $\nembedA$ and $\nembedB$ have tidy dongles, $W$ 
does not contain any apex of $\nembedA$ or $\nembedB$
and its projection is further in radial distance than $q$ from any vortex
in both $\nembedA$ and $\nembedB$, we infer that 
$Z \subseteq V(\MainA)$ and $Z \subseteq V(\MainB)$. 

We now consider $\SubProj_{\nembedA}(W)$ and apply Lemma~\ref{lem:rballs} (or Lemma~\ref{lem:rballs:planar} if $\eulerg_\ast = 0$) to obtain a cycle $\Gamma$ in $W$ such that
$\PathProj_{\nembedA}(\Gamma)$ consists of vertices within radial distance (in $W$) exactly 3 from 
the face containing $\SubProj_{\nembedA}(v_0)$ (and one cuff of $\Sigma^\bullet$)
and separates the said face from all other cuffs of $\Sigma^\bullet$. 
In $\Sigma$, $\PathProj_{\nembedA}(\Gamma)$ bounds a disc containing $\SubProj_{\nembedA}(v_0)$; we will
call this disc the \emph{inside} of $\PathProj_{\nembedA}(\Gamma)$, and the other component
of $\Sigma-\PathProj_{\nembedA}(\Gamma)$ the \emph{outside}.

Recall that $\SubProj_{\nembedB}(W)$ is also the natural embedding of a $\Sigma$-wall
into $\Sigma$. Hence, $\PathProj_{\nembedB}(\Gamma)$ is also a contractible cycle
such that $Z \setminus V(\Gamma)$ is split between
the two connected components of $\Sigma-\PathProj_{\nembedB}(\Gamma)$ 
exactly in the same way as it is split in $\Sigma-\PathProj_{\nembedA}(\Gamma)$. 
We again call the corresponding sides of $\PathProj_{\nembedB}(\Gamma)$ the \emph{inside} and the \emph{outside} and note that the inside is a disc.

We will also need the following definition for a predongle $D \in \Predongles$: the \emph{preimage (w.r.t. $\Apices$)} of
$D$ is a subgraph $G_D$ of $G$ that consists of:
\begin{itemize}
\item the vertices of $D$ and $N_{G(\Apices)}(D)$;
\item all edges of $G(\Apices)[N_{G(\Apices)}[D]]$ that are present in $G$;
\item for every connected component $C$ of $G-\Apices-V(G(\Apices))$ with $N_{G-\Apices}(C) \cap D \neq \emptyset$
(so in particular $N_{G-\Apices}(C) \subseteq N_{G(\Apices)}[D]$ as $D \subseteq V(G(\Apices))$),
the entire graph $G[C]$ and all edges between $C$ and $N_{G(\Apices)}[D]$. 
\end{itemize}
Note that for every connected component $C$ of $G(\Apices)[D]$ there exists one connected component $C'$ of $G_D-N_{G(\Apices)}(D)$
such that $C' \cap V(G(\Apices)) = C$ and, vice versa, for every connected component $C'$ of $G_D-N_{G(\Apices)}(D)$
we have exactly one connected component $C$ of $G(\Apices)[D]$ such that $C' \cap V(G(\Apices))=C$. 

Consider now the graph $G-Z-\Apices$ and its connected component $C$. 
Every neighbor $v \in N_G(C)$ is either in $\Apices$ or in $Z$; in the latter case, $v \in V(\MainA)$ as $Z \subseteq V(\MainA)$.
Hence, we can speak about the placement of vertices of $N_{G-\Apices}(C)$ in $\nembedA$. 
We have the following observation:

\begin{claim}\label{cl:prox:C}
For every connected component $C$ of $G-Z-\Apices$, one of the following holds:
\begin{enumerate}
\item There exists a single face $f_C$ of $\SubProj_{\nembedA}(W)$ whose closure contains all vertices of $N_{G-\Apices}(C)$ and such that 
all vertices of $\{\SubProj_{\nembed}(u)~|~u \in C\}$ lie either in the closure of $f_C$
or in one of the faces of $\SubProj_{\nembedA}(W)$ sharing an edge with $f_C$. 
\item There exists $D \in \Predongles$ that contains a degree-3 vertex of $W(\Apices)$ (so in particular $N_{G(\Apices)}(D) \subseteq Z$) 
  and $C$ is one of the connected components of $G_D-N_{G(\Apices)}(D)$, where $G_D$ is the preimage of $D$ in $\nembedA$. 
\end{enumerate}
In particular, either all vertices of $N_{G-\Apices}(C)$ 
are inside or on $\PathProj_{\nembedA}(\Gamma)$ in $\nembedA$, or all vertices
of $N_{G-\Apices}(C)$ are outside or on $\PathProj_{\nembedA}(\Gamma)$ in $\nembedA$. 

The same statement holds for $\nembedA$ replaced with $\nembedB$. 
\end{claim}
\begin{proof}
The proofs for $\nembedA$ and $\nembedB$ are symmetric and hence we present only the first one. Recall that
both in $\nembedA$ and $\nembedB$, the vertices of the projection of $W$ are within radial distance more than $q+1$ from 
the virtual vortex vertices and the society vertices. 

Assume first that $C$ does not contain any vertex $u$ such that $\SubProj_{\nembedA}(u)$ lies on $\SubProj_{\nembedA}(W)$. 
Then, the connectivity of $C$ and the fact that $W$ is edge-disjoint with vortices of $\nembedA$ implies that
there exists a single face of $\SubProj_{\nembedA}(W)$ that contains $\{\SubProj_{\nembedA}(u)~|~u \in C\}$
and that all vertices of $N_{G-\Apices}(C)$ lie in or on the boundary of the said face. This proves that the first option holds in this case. 

Consider now any vertex $u \in C$ such that $u' \coloneqq \SubProj_{\nembedA}(u)$ is a vertex of $\SubProj_{\nembedA}(W)$. 
Note that we have either $u \in V(G(\Apices))$ or there exists a connected component $C_u$ of $G-\Apices-V(G(\Apices))$ that contains $u$. 
Let us explore a bit deeper the latter case. 
Clearly, $|N_{G-\Apices}(C_u)| \leq 2$. 
Also, since $\nembedA$ has tidy dongles, in particular it is aligned with 3-connected components,
there exists a dongle $\DongleA{j}$ such that
$u' = \PlusDongleVtx{j}$ and $C_u \subseteq V(\DongleA{j}) \setminus V(\MainA)$.

We now exclude the corner case when $u \notin V(G(\Apices))$, so
$C_u$ is well-defined, but there is no $D \in \Predongles$ with
$N_{G-\Apices}(C_u) \cap D \neq \emptyset$. 
In that case, $C_u$ is not contained in $G_D$ for any $D \in \Predongles$.
Hence, $C_u \subseteq V(\DongleA{j}) \setminus V(\MainA)$
for a dongle $\DongleA{j}$ with $u' = \PlusDongleVtx{j}$,
$|V(\DongleA{j}) \cap V(\MainA)| = 2$ and with
both vertices of $V(\DongleA{j}) \cap V(\MainA)$ in $Z$. 
Consequently, $C_u$ is one of the connected components of $\DongleA{j} - V(\MainA)$. 
Thus, $C = C_u$ and the first outcome of the claim holds for any of the two faces of
$\SubProj_{\nembedA}(W)$ incident with the virtual edge connecting
$V(\DongleA{j}) \cap V(\MainA)$. 

We now claim that in the remaining case there exists a unique $D \in \Predongles$
such that the preimage $G_D$ contains $u$. 
To this end, consider two cases. If $u \in V(G(\Apices))$, 
then, as $u \notin Z$, there exists a predongle $D \in \Predongles$ with $u \in D$. 
Thus, $u$ is in $G_D$ and, since the predongles in $\Predongles$ are vertex-disjoint and nonadjacent
in $G(\Apices)$, $u$ is not in $G_{D'}$ for any other $D' \in \Predongles$.
In the other case, if $u \notin V(G(\Apices))$ and $C_u$ is defined, 
since we are not in the case excluded in the previous paragraph, there exists
$D \in \Predongles$ with $N_{G-\Apices}(C_u) \cap D \neq \emptyset$. 
Hence, $G_D$ contains $C_u$ by the definition of $G_D$. Again, since distinct elements
of $\Predongles$ are disjoint and nonadjacent in $G(\Apices)$, no other $G_{D'}$
for $D' \in \Predongles$ contains a vertex of $C_u$. This finishes the proof of the existence
of the unique $D \in \Predongles$ with $u$ in $G_D$.

Since $\SubProj_{\nembedA}(u)$ lies on $\SubProj_{\nembedA}(W)$ while the elements of $N_{G(\Apices)}(D) \cap W(\Apices)$ are in $Z$,
we have that actually $u \in V(G_D) \setminus N_{G(\Apices)}(D)$. Since $W(\Apices)$ is 2-connected, we have 
$|N_{G(\Apices)}(D) \cap W(\Apices)| \in \{2,3\}$. 
Note that all vertices of $N_{G(\Apices)}(D) \cap W(\Apices)$ are in $Z$
and these are the only vertices of $G_D$ that are in $Z$.

Consider first the case $|N_{G(\Apices)}(D) \cap W(\Apices)| = 3$.
Since $|N_{G(\Apices)}(D)| = 3$, we have that all vertices of $N_{G(\Apices)}(D)$
are in $Z$ and, hence, $C$ is the connected component of $G_D - N_{G(\Apices)}(D)$
that contains $D$. 
If $G_D - N_{G(\Apices)}(D)$ contains a degree-3 vertex of $W(\Apices)$,
then the second outcome of the claim holds. 
Otherwise, $G_D \cap W$ consists of three vertices of $Z$ and possibly
a path between two of these vertices. 
Since $W$ is edge-disjoint with vortices in $\nembedA$, we infer that
all vertices of $N_{G(\Apices)}(D)$ lie in the closure
of a single face $f_C$ of $W$ and every vertex of $\SubProj_{\nembed}(C)$
lies either in the closure of $f_C$ or in a face of $\SubProj_{\nembed}(W)$
that contains in its closure at least two vertices of $N_{G(\Apices)}(D)$. 

We are left with the case 
when for every choice of $u \in C$ with $\SubProj_{\nembed}(u)$ being
in $\SubProj_{\nembed}(W)$,
the corresponding predongle $D(u) \in \Predongles$ exists
and $|N_{G(\Apices)}(D(u)) \cap W(\Apices)| = 2$. 
Hence, $G_{D(u)} \cap W$ consists of two vertices of $Z$ 
and possibly a path between these two vertices. 
In particular, $G_{D(u)}$ does not contain a degree-3 vertex of $W(\Apices)$
Since $\nembedA$ has tidy dongles, the unique vertex $w(u)$ of $N_{G(\Apices)}(D(u)) \setminus Z$
lies in $\MainA$; in particular, $w(u)$ does not lie on $\SubProj_{\nembed}(W)$. 
Furthermore, the whole $G_{D(u)}$ and $w(u)$ lies in $C$. 
Let $f(u)$ be the face that accommodates $w(u)$.

The crucial observation now is that, regardless of the choice of $u$ as above, $f(u)$ is the same face.
Indeed, assume that for some $u_1,u_2$ as above, $f(u_1) \neq f(u_2)$. In paricular, $D(u_1) \neq D(u_2)$ and $w(u_1) \neq w(u_2)$. Let $P$ be a path in $C$ from $w(u_1)$ to $w(u_2)$
and consider $\PathProj_{\nembedA}(P)$. As $f(u_1) \neq f(u_2)$ and $W$ does not contain an edge of a vortex of $\nembedA$,
this projection needs to contain a vertex $u' = \SubProj_{\nembedA}(u)$ in $\SubProj_{\nembed}(W)$. 
However, $P$ can visit only one vertex of $N_{G(\Apices)}(D(u))$, namely $w(u)$, as the other two are in $Z$. This is a contradiction as $P$ goes from $w(u_1)$ to $w(u_2)$, which are both 
not in $G_{D(u)}$, via a vertex $u$ in $G_{D(u)}$.
This finishes the proof that $f(u)$ is the same face regardless of the choice of $u$; we proclaim to be $f_C$. 

We observe now that  
$\SubProj_{\nembed}(v)$ for any $v \in C$ lies either in the closure of $f_C$, or otherwise
$v \in V(G_{D(u)}) \setminus Z$ for some $u$ as above and $\SubProj_{\nembed}(v)$
lies in the face of $\SubProj_{\nembed}(W)$ that contains the two vertices
of $N_{G(\Apices)}(D(u)) \cap Z$ in its closure. 
This finishes the proof of the claim.
\cqed\end{proof}

Let $\mathcal{I}$ be the family of all connected components $C$ of $G-Z-\Apices$ such that $N_{G-\Apices}(C)$ is nonempty
and all vertices of $N_{G-\Apices}(C)$ are inside or on $\PathProj_{\nembedA}(\Gamma)$ in $\nembedA$.
Similarly, let
$\widetilde{\mathcal{I}}$ be the family of all connected components $C$ of $G-Z-\ApicesB$ such that $N_{G-\ApicesB}(C)$ is nonempty
and all vertices of $N_{G-\ApicesB}(C)$ are inside or on $\PathProj_{\nembedB}(\Gamma)$ in $\nembedB$.

We say that a component $C$ of $G-Z-\Apices$ is \emph{close to the inside (in $\nembedA$)} if 
$N_{G-\Apices}(C)$ contains a vertex of $Z$ that is inside or on $\PathProj_{\nembedA}(\Gamma)$.
Claim~\ref{cl:prox:C} implies that if $C$ is close to the inside, then 
either (in the case of the first outcome) the face $f_C$ is inside $\PathProj_{\nembedA}(\Gamma)$
or one of the faces of $\SubProj_{\nembedA}(W)$ incident with $\PathProj_{\nembedA}(\Gamma)$
or (in the case of the second outcome) $C$ is one of the connected component of $G_D-N_{G(\Apices)}(D)$ for $D \in \Predongles$ that contains a degree-3 vertex of $W$ that is inside or on $\Gamma$
or incident with one of the faces incident with $\Gamma$ in the natural embedding of $W$.
In all cases, since $W$ is of order $h \gg 5$, we infer that $N_G(C)$ does not contain
any local apex of $\nembed$.
We will also use an analogous definition of a component $\widetilde{C}$ of $G-Z-\ApicesB$ being \emph{close to the inside (in $\nembedB$)}.

The separation properties of $W$ and local apices of $\nembedA$ and $\nembedB$ imply the following.
\begin{claim}\label{cl:prox:the-same-inside}
\[ \widetilde{\mathcal{I}} = \bigcup_{C \in \mathcal{I}} \{\widetilde{C}~|~\widetilde{C}\ \textrm{is a connected component of}\ G[C]-\ApicesB\ \textrm{and}\ N_G(\widetilde{C}) \cap Z \neq \emptyset\}. \] 
\end{claim}
\begin{proof}
Consider first $C \in \mathcal{I}$ and apply Claim~\ref{cl:prox:C} to it. Since $N_{G-\Apices}(C)$ is nonempty and in the inside or on $\PathProj_{\nembedA}(\Gamma)$, 
we have that $C$ lies close to the inside in $\nembedA$. 
Hence, $N_G(C)$ does not contain any vertex of $\ApicesA \setminus \ApicesB$. 
Consequently, all connected components of $G[C]-\ApicesB$ are connected components of $G-Z-\ApicesB$.
Since a vertex of $Z$ is inside or on $\PathProj_{\nembedA}(\Gamma)$ in $\nembedA$ if and only if it is inside or on $\PathProj_{\nembedB}(\Gamma)$ in $\nembedB$, 
any connected component $\widetilde{C}$ of $G[C]-\ApicesB$ with $N_G(\widetilde{C}) \cap Z \neq \emptyset$, actually contains in its neighborhood a 
vertex of $Z$ inside or on $\PathProj_{\nembedB}(\Gamma)$ in $\nembedB$, and hence belongs to $\widetilde{\mathcal{I}}$. 
This proves the $\supseteq$ inclusion of the claim.

In the second direction, consider $\widetilde{C} \in \widetilde{\mathcal{I}}$. 
Since no vertex of $Z \subseteq V(W)$ is an attachment point of a local apex of $\nembedA$,
there exists a connected component $C'$ of $G[\widetilde{C}]-\ApicesA$ such that some vertex $x \in N_G(\widetilde{C}) \cap Z$ is in $N_G(C')$.
Then, $C'$ is contained in one connected component $C$ of $G-Z-\ApicesA$ that, furthermore, has $x$ in $N_G(C)$. 
Since $x \in Z$ is inside or on $\PathProj_{\nembedB}(\Gamma)$ in $\nembedB$, it is also inside or on $\PathProj_{\nembedA}(\Gamma)$ in $\nembedA$. 
We infer that $C$ is close to the inside in $\nembedA$. 
Hence, $N_G(C)$ contains no vertex of $\ApicesA \setminus \ApicesB$. Consequently $\widetilde{C} \cap \ApicesA = \emptyset$, that is, $C' = \widetilde{C}$.

We infer that $\widetilde{C}$ is one of the connected components of $G[C]-\ApicesB$. Furthermore, $x$ witnesses that $N_G(\widetilde{C}) \cap Z \neq \emptyset$. 
It remains to show that $C \in \mathcal{I}$. This claim is immediate if $C = \widetilde{C}$, that is, $C$ contains no vertex of $\ApicesB \setminus \ApicesA$, so assume otherwise:
let $w \in C \cap (\ApicesB \setminus \ApicesA)$.

Recall that for any $w \in \ApicesB \setminus \ApicesA$, $\SubProj_{\nembedA}(w)$
is in $\Sigma$ inside one of the cuffs of $\Sigma^\bullet$.
Since $C$ is close to the inside, the only cuff where $w$ may lie is the one 
that also contains $\SubProj_{\nembedA}(v_0)$. 
By Claim~\ref{cl:prox:C} applied to $C$ in $\nembedA$, we infer that $N_{G-\Apices}(C)$
is contained in the closure of the face of $\SubProj_{\nembedA}(W)$ that contains
$\SubProj_{\nembedA}(v_0)$ and its incident faces in $\SubProj_{\nembedA}(W)$. 
Hence, $C$ is inside $\PathProj_{\nembedA}(\Gamma)$ and $C \in \mathcal{I}$, as desired.
\cqed\end{proof}

Without loss of generality, assume that in $\nembedA$ every dongle $\DongleA{j}$ with $V(\DongleA{j}) \cap V(\MainA) = \emptyset$ has its disc not in any face of $\SubProj_{\nembedA}(W)$
incident with a vertex of $Z$ that is inside or on $\PathProj_{\nembedA}(\Gamma)$ and similarly for $\nembedB$. 

\begin{figure}[tb]
\begin{center}
\includegraphics[height=0.83\textheight]{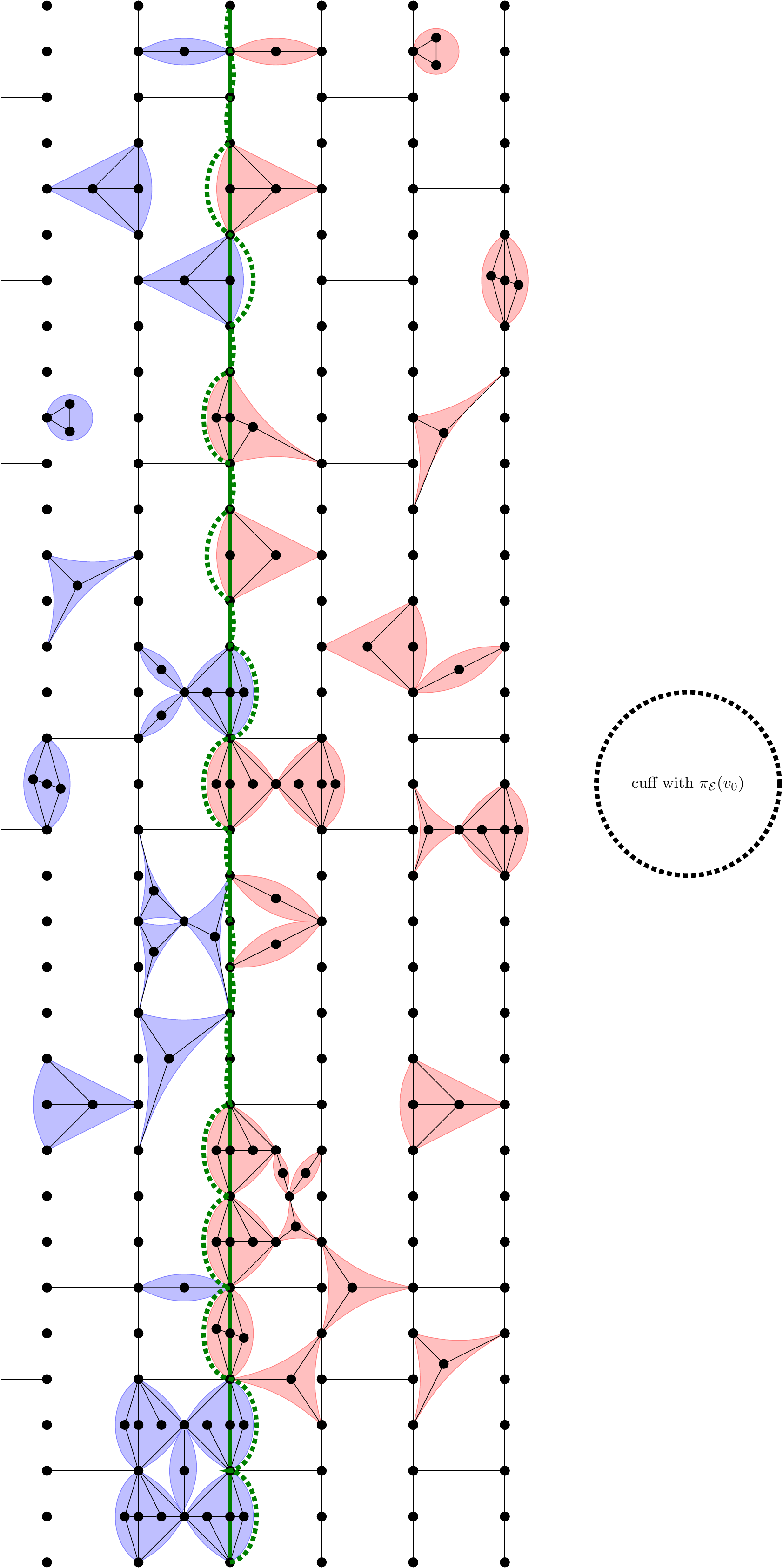}
\caption{Example situation around the curve $\Gamma$ in the wall $W$. 
  The face to the right is the face of $\nembedA$ with $\SubProj_{\nembedA}(v_0)$.
  Red and blue areas correspond to inside and outside dongles (with the subgraph inside the area being the
  dongle itself). 
  The solid green line is the cycle $\Gamma$. The dashed green line is the
  curve $\gamma$ given by Claim~\ref{cl:prox:gamma}.}\label{fig:prox:gamma}
\end{center}
\end{figure}

For the next claim, we refer to Figure~\ref{fig:prox:gamma}.
\begin{claim}\label{cl:prox:gamma}
There exists closed contractible face-vertex curve $\gamma$ in $\MainPlusA$ that visits exactly the vertices of $Z \cap \PathProj_{\nembedA}(\Gamma)$ in the order of their
appearance on $\PathProj_{\nembedA}(\Gamma)$ and encloses a disc of $\Sigma$ containing in its inside exactly the following elements of $\nembedA$ and $\Sigma^\bullet$:
the cuff of $\Sigma^\bullet$ containing $\SubProj_{\nembedA}(v_0)$, all vertices of $Z$ that are inside $\PathProj_{\nembedA}(\Gamma)$, 
and the projections of all components of $\mathcal{I}$ and edges between them and $Z$. 

An analogous curve $\widetilde{\gamma}$ exists for $\nembedB$, $\MainPlusB$, and the components of $\widetilde{\mathcal{I}}$.
\end{claim}
\begin{proof}
Follow $\PathProj_{\nembedA}(\Gamma)$ and consider two consecutive vertices $z_1$ and $z_2$ of $Z$. If there is only an edge of $\MainPlusA$ between $z_1$ and $z_2$ on $\PathProj_{\nembedA}(\Gamma)$,
draw $\gamma$ on either side of this edge from $z_1$ to $z_2$. Otherwise, there is a component $C$ of $G-Z-\Apices$ that contains all vertices of $\Gamma$ between $z_1$ and $z_2$ and, consequently,
whose projection contains all the vertices of $\PathProj_{\nembedA}(\Gamma)$ between $z_1$ and $z_2$.
If all vertices of $N_{G-\Apices}(C)$ are inside or on $\PathProj_{\nembedA}(\Gamma)$, 
we can draw $\gamma$ from $z_1$ to $z_2$ on the outside side of $\PathProj_{\nembedA}(\Gamma)$ (i.e., keeping the projection of $C$ on the same side of $\gamma$ as the inside
    of $\PathProj_{\nembedA}(\Gamma)$). 
Otherwise, if all vertices of $N_{G-\Apices}(C)$ are outside or on $\PathProj_{\nembedA}(\Gamma)$, 
we can draw $\gamma$ from $z_1$ to $z_2$ on the inside side of $\PathProj_{\nembedA}(\Gamma)$ (i.e., keeping the projection of $C$ on the same side of $\gamma$ as the outside
    of $\PathProj_{\nembedA}(\Gamma)$). 
The construction for $\nembedB$ is completely analogous.
\cqed\end{proof}
As before, the disc part of $\Sigma-\gamma$ that contains the promised objects in Claim~\ref{cl:prox:gamma} is called the \emph{inside} of $\gamma$ (and similarly of $\widetilde{\gamma}$). 

We are now ready to conclude the proof of Lemma~\ref{lem:proximity}. 
Consider a near-embedding $\nembedC$ of $G$ constructed from $\nembedB$ as follows:
\begin{enumerate}
\item Take the curve $\widetilde{\gamma}$ provided by Claim~\ref{cl:prox:gamma} and erase everything inside it.
That is, remove from $\nembedC$ all vertices of $\MainPlusC$ embedded inside $\widetilde{\gamma}$, all vortices whose virtual vertex is inside $\widetilde{\gamma}$,
     and all dongles whose virtual vertex is inside $\widetilde{\gamma}$. 
\item Draw inside $\widetilde{\gamma}$ everything exactly as it is drawn inside $\gamma$ in $\nembedA$. 
That is, add to $\MainPlusC$ all edges and vertices that are in $\MainPlusA$ inside $\gamma$ and all dongles whose virtual vertices are inside $\gamma$ in $\nembedA$. 
\end{enumerate}
The fact that the vertices of $Z \cap \PathProj_{\nembedA}(\Gamma)$ appear in the same order on $\gamma$ and $\widetilde{\gamma}$ implies that the above replacement
of a part of $\MainPlusC$ with a part of $\MainPlusA$ is possible. 
Claim~\ref{cl:prox:the-same-inside} implies that every vertex that is removed from a part of $\nembedC$ in the first step is restored in the second step and, moreover,
  every vertex $w \in \ApicesB \setminus \ApicesA$ such that $\SubProj_{\nembedA}(w)$ is inside $\gamma$ is restored as well. 
Furthermore, note that the second step does not create any vortices. 

We infer that $\ApicesC \subseteq \ApicesB$ and $\ApicesC = \ApicesB$ only if no vertex of $\ApicesB \setminus \ApicesA$ has its projection inside $\gamma$ in $\nembedA$. 
By the optimality of $\nembedB$, we infer that the latter case happens and $\nembedC$ is an optimal near-embedding as well.
Due to the existence of $v_0$, we infer that the set of vertices in vortices in $\nembedC$ is a strict subset of the set of vertices in vortices in $\nembedB$.
This contradiction with the minimality assumption of $\nembedB$ finishes the proof of the lemma.
\end{proof}

\section{Wrap-up}
We are now ready to wrap-up the proof of Theorem~\ref{thm:rigid}. 

Let $\OptNembeds$ be the family of all optimal near-embeddings of the input graph $G$
that have tidy dongles and minimal vortices.
Let $\Apices_\OptNembeds$ be the set of those vertices $v \in V(G)$ for which there exists $\nembed \in \OptNembeds$ with $v \in \Apices$.
Let $\VortexVtcs_\OptNembeds$ be the set of these vertices $v \in V(G)$ for which there exists
$\nembed \in \OptNembeds$ and a vortex $\VortexA{i}$ of $\nembed$ with $v \in V(\VortexA{i})$. 

The following statement can be proven using standard irrelevant vertex and bounded treewidth
techniques. 
If $\OptNembeds$ were just the family of optimal near-embeddings (not only the ones with tidy dongles and minimal vortices), then such a statement would follow directly from~\cite{GroheKR13}. 
In Appendix~\ref{app:compute-opt-nembed} we discuss that it is straightforward to adapt
the techniques of~\cite{GroheKR13} to our setting. 
\begin{lemma}\label{lem:compute-opt}
Given the input as in Theorem~\ref{thm:rigid},
the following parameters of an optimal near-embedding can be computed within the running time bound promised by Theorem~\ref{thm:rigid}:
the level $\iota_\ast$, $\napices_\ast$, $\nvortices_\ast$, $\eulerg_\ast$, as well as the sets $\Apices_\OptNembeds$ and $\VortexVtcs_\OptNembeds$.
\end{lemma}
Observe that the sets $\Apices_\OptNembeds$ and $\VortexVtcs_\OptNembeds$ are
defined in an isomorphic-invariant way. 

Pick one $\nembed \in \OptNembeds$ and apply Lemma~\ref{lem:proximity} to it, obtaining
a set $L \subseteq V(\MainPlusA)$ of size at most~$\ProxCoverSize$. 

We consider the following construction, parameterized by a set $Y \subseteq V(G)$ of
size at most~$9$. 
Initialize $R \coloneqq L \cup \{\SubProj_\nembed(v)~|~v \in Y \setminus \Apices\}$
and for every $v \in R$ set $r_v \coloneqq \ProxCoverRadius + 3$.
Iteratively, execute the following process:
while there exist two $u_1,u_2 \in R$ such that the radial discs $\rdisc{\nembed}{r_{u_1}+3}{u_1}$ and $\rdisc{\nembed}{r_{u_2}+3}{u_2}$ are not disjoint,
      apply Lemma~\ref{lem:packdiscs} to these two discs, obtaining a vertex $w$;
set $R \coloneqq (R \setminus \{u_1,u_2\}) \cup \{w\}$ and set $r_w = r_{u_1}+r_{u_2} + 6$. 
(We will later in the proof verify that the radii $(r_v)_{v \in R}$ are small enough during the process so that the application
 of Lemma~\ref{lem:packdiscs} is always justified.)
At the end of the process, we have that the radial discs $\rdisc{\nembed}{r_u+3}{u}$ for $u \in R$ are pairwise disjoint and
\[ |R| \leq \ProxCoverSize + 9\quad\mathrm{and}\quad \sum_{u \in R} r_u \leq (9+\ProxCoverSize)(9 + \ProxCoverRadius). \]
Since $\ProxCoverSize$ and $\ProxCoverRadius$ are bounded by a computable function of $q'$ and $k'$, we can pick $\funstep$ sufficiently quickly growing so
that during the process the radii $r_u$ are smaller than $\rmax/100$ and all applications of Lemma~\ref{lem:packdiscs} are justified.
Furthermore, note the process maintains the invariant
that for every $u \in L$ there exists $u' \in R$ such that
$\rdisc{\nembed}{\ProxCoverRadius+3}{u}$ is contained in $\rdisc{\nembed}{r_{u'}}{u'}$. 

Lemma~\ref{lem:proximity} asserts that for every $v \in \Apices_\OptNembeds \cup \VortexVtcs_\OptNembeds$ either $v \in \Apices$ or there exists $u \in L$ such that $\SubProj_\nembed(v)$ lies in
$\rdisc{\nembed}{\ProxCoverRadius}{u}$. 

For $i=1,2,3$, let $\MainPlusA^i(Y)$ be the result of clearing in $\MainPlusA$ all radial discs
$\rdisc{\nembed}{r_u+i}{u}$ for $u \in R$. By Lemma~\ref{lem:genus:clear-2conn}, %
the graphs $\MainPlusA^i(Y)$ are 2-connected and $\MainPlusA^1(Y)$ contains the unique 3-connected component
$\MainPlusA^\ast(Y)$ that has more than $2q''$ vertices and contains $\MainPlusA^2(Y)$.
(In what follows we mostly use the case $i=1$ and $i=2$ and only look at $i=3$ in the last claim.)

By Lemmas~\ref{lem:opt-ne-rep} and~\ref{lem:genus:punctured-fw}, $\MainPlusA^2(Y)$ still has very large facewidth, namely
at least
\[ k'''/100 - 2q - k - 2(9+\ProxCoverSize)(15+\ProxCoverRadius).\]
As both $\ProxCoverSize$ and $\ProxCoverRadius$ are bounded by a computable function of $q'$ and $k'$, 
   we can pick $\funstep$ sufficiently quickly growing so that the said facewidth is greater than
   \begin{equation}\label{eq:wrap-up-fw}
   2f_{\mathrm{fwf}}(\Sigma, 3q''+2f_{\mathrm{ue}}(\eulerg_\ast)).
   \end{equation}
Since $\MainPlusA^1(Y)$ does not contain any virtual vortex edge or vertex, 
the above is also a (crude) lower bound on the number of vertices
of $V(\MainPlusA^2(Y)) \cap V(\MainA)$. 

Again since $\MainPlusA^1(Y)$ does not contain any virtual vortex edge or vertex,
we can take a lift $G^1(Y)$ of $\MainPlusA^1(Y)$ in $G$. 
Since the projection of $G^1(Y)$ is again $\MainPlusA^1(Y)$, 
it has face-width at least as in~\eqref{eq:wrap-up-fw}. By Theorem~\ref{thm:GM7}, $G^1(Y)$ contains a $\Sigma$-wall $H^1(Y)$
  of order 
  \[h := 3q''+2f_{\mathrm{ue}}(\eulerg_\ast).\] 
Note that all the degree-3 vertices of $H^1(Y)$ are contained in the same 3-connected component
of $G-\Apices_\OptNembeds-\VortexVtcs_\OptNembeds$ and, furthermore, this 3-connected component
contains a $\Sigma$-wall $H^\ast(Y)$ of order $h$ as well, as it contains $H^1(Y)$ with possibly some
of the degree-2 vertices suppressed.
Let us denote this 3-connected component by $H(Y)$. The presence of the wall implies that the treewidth
of $H(Y)$ is at least $h$. Furthermore, note that the 3-connectivity
of $\MainPlusA^\ast(Y)$ implies that the projection $\SubProj_{\nembed}(H(Y))$ contains $\MainPlusA^2(Y)$ (which is a subgraph of $\MainPlusA^\ast(Y)$ present in $\MainPlusA$). 

Since $H(Y)$ is a 3-connected component of a subgraph of $G-\Apices$, it is a minor of one
of the 3-connected components of $G-\Apices$. The only 3-connected component of $G-\Apices$
that has treewidth at least $h$ is $G(\Apices)$. Hence, $H(Y)$ is a minor of $G(\Apices)$. 

Recall that $\MainPlusA^\ast(Y)$ is the 3-connected component of $\MainPlusA^1(Y)$ that contains
$\MainPlusA^2(Y)$. Since $\nembed$ has tidy dongles (in particular, is aligned with 3-connected components) we infer that $V(H(Y))$ contains all vertices of $V(\MainPlusA^2(Y)) \cap V(\MainA)$. 

Since $\MainPlusA^1(Y)$ does not contain any virtual vortex edge or vertex
we infer the following:

\begin{claim}\label{cl:Csmall}
Every connected component
$C$ of $G-\Apices-(V(\MainPlusA^2(Y)) \cap V(\MainA))$ either:
\begin{itemize}
\item has its projection $\SubProj_{\nembedA}(C)$ contained in the disc
$\rdisc{\nembed}{r_u+2}{u}$ for one $u \in R$, or
\item is contained in one of the dongles of $\nembedA$.
\end{itemize}
In particular, Lemma~\ref{lem:radial-disc} implies that the treewidth
of $C$ is bounded by $2q''+10\rmax < 2q''+k''$. 
\end{claim}
 Recall that $\Apices \subseteq \Apices_\OptNembeds$ while
$V(\MainPlusA^2(Y)) \cap V(\MainA) \subseteq V(H(Y))$. 
Hence, the same treewidth bound applies to any 
of any 3-connected component of $G-\Apices_\OptNembeds-\VortexVtcs_\OptNembeds$ different
than $H(Y)$.
As the said bound is smaller than $3q''$,
   we infer that $H(Y)$ is the unique 3-connected component
of $G-\Apices_\OptNembeds-\VortexVtcs_\OptNembeds$ that has treewidth at least $h$.
In particular, $H(Y)$ is defined in a way independent of the choice of $Y$
and in an isomorphic-invariant way.
Henceforth we will denote $H(Y)$ by just $H$.

Observe that $\nembed$ induces a near-embedding of $G-\Apices_\OptNembeds-\VortexVtcs_\OptNembeds$
into $\Sigma$ that does not contain any apices or vortices.
Since $H$ is a minor of $G-\Apices_\OptNembeds-\VortexVtcs_\OptNembeds$, from the near-embedding $\nembedA$ considered on $G-\Apices_\OptNembeds-\VortexVtcs_\OptNembeds$ we may obtain
a near-embedding $\nembedT$ of $H$ that has no apices, no vortices, and is of Euler
genus at most $\eulerg_\ast$. Since $H$ contains a $\Sigma$-wall $H^\ast(\emptyset)$ of order $h$
that is not entirely contained in a dongle of this near-embedding (because $h \gg q$ and dongles are of size at most $q$), $\nembedT$ in fact is a near-embedding
of $H$ into $\Sigma$ (i.e., is of Euler genus exactly $\eulerg_\ast$). 

We now make use of the parameterization of the process by $Y$ to observe
the following.
\begin{claim}\label{cl:Hcut}
For every separation $(A,B)$ of $H$ of order at most $9$, 
 either $H[A \setminus B]$ is of treewidth at most $2q''+k''$ and $H[B \setminus A]$ is of treewidth at least $h$, or vice versa.
Furthermore, the set $\mathcal{T}$ of all separations $(A,B)$ of $H$ of order at most $3$
such that $H[B \setminus A]$ is of treewidth at least $h$ is a tangle of order $4$ in $H$.
\end{claim}
\begin{proof}
Let $(A,B)$ be a separation of $H$ of order at most $9$.
Consider the aforementioned process for $Y \coloneqq A \cap B$, yielding graphs $\MainPlusA^\ast(Y),\MainPlusA^1(Y),\MainPlusA^2(Y)$.
As proven, all connected components of
$G-\Apices-(V(\MainPlusA^2(Y)) \cap V(\MainA))$ have treewidth bounded $2q''+k''$.

On the other hand, recall that $\MainPlusA^2(Y)$ contains a $\Sigma$-wall
of order $h$. Observe that this wall is disjoint with $Y$. Hence,
either $H[B \setminus A]$ or $H[A \setminus B]$ contains a $\Sigma$-wall of order $h$;
without loss of generality, say it is $H[B \setminus A]$.
Furthermore, 
recall that the 3-connectivity of $\MainPlusA^\ast(Y)$ implies that the
projection $\SubProj_{\nembedA}(H)$ contains $\MainPlusA^2(Y)$ (as a subgraph).
Hence, all vertices of $V(\MainPlusA^2(Y)) \cap V(\MainA)$ are in the same connected component of $H-Y$.
Claim~\ref{cl:Csmall} now implies that this component lies in $H[B \setminus A]$, that is, $H[B \setminus A]$ contains all vertices of $V(\MainPlusA^2(Y)) \cap V(\MainA)$.
Consequently, by Claim~\ref{cl:Csmall}, $H[A \setminus B]$ is of treewidth at most $2q''+k''$,
as desired.

For the second part of the claim, let $\mathcal{T}$ be the set of all separations
$(A,B)$ of $H$ of order at most $3$ such that $H[B \setminus A]$ is of treewidth at least $h$.
To show that it is a tangle, we need to show that for any $(A_1,B_1),(A_2,B_2),(A_3,B_3) \in \mathcal{T}$ we have $A_1 \cup A_2 \cup A_3 \neq V(H)$. However, by the above reasoning
for $Y = \bigcup_{i=1}^3 A_i \cap B_i$ we have that $V(\MainPlusA^2(Y)) \cap V(\MainA)$
is disjoint with $A_1 \cup A_2 \cup A_3$. This finishes the proof of the claim.
\cqed\end{proof}

Consider now the family $\OptNembeds_H$ of all near-embeddings of $H$ without any apices or vortices,
and into a surface of Euler genus at most $\eulerg_\ast$. 
The discussed near-embedding $\nembedT$ is in $\OptNembeds_H$, so in particular
$\OptNembeds_H$ is nonempty.
Since $H$ is 3-connected, trivially any near-embedding from $\OptNembeds_H$
is aligned with 3-connected components.

Let $\mathcal{T}$ be the tangle of Claim~\ref{cl:Hcut}. 
By Lemma~\ref{lem:canonize-dongles}, there exists a near-embedding in $\OptNembeds_H$ that has
canonical dongles with respect to $\mathcal{T}$.
We remark that the above uses the set of predongles $\Predongles(H,\mathcal{T})$ 
and that, as there are no vortices in the embeddings of $\OptNembeds_H$, 
every dongle $\DongleB{j}$ of a near-embedding $\nembedB \in \OptNembeds_H$ satisfies
$V(\DongleH{j}) \setminus V(\MainH) \in \Predongles(H,\mathcal{T})$. 

Consider now the graph $H'$ constructed from $H$ by replacing, for every $D \in \Predongles(H,\mathcal{T})$, 
the set $D$ with a single vertex $x_D$, adjacent to $N_H(D)$, and turning $N_H(D)$ into a clique. 
As the elements of $\Predongles(H,\mathcal{T})$ are pairwise disjoint and nonadjacent, this operation
is well-defined. 
As $H$ is 3-connected, and each operation replaces $N_H[D]$ with a 4-vertex clique $N_{H'}[x_D]$,
$H'$ is also 3-connected. Also, note that $H'$ is defined in an isomorphic-invariant way. 

Furthermore, observe that any near-embedding $\nembedB \in \OptNembeds_H$ with canonical
dongles induces an embedding of $H'$ into $\Sigma$. Indeed, start with the
embedding of $\MainPlusH$
and then, for every $D \in \Predongles(H,\mathcal{T})$ we have two options:
\begin{itemize}
\item There is a dongle $\DongleH{j}$ with $V(\DongleH{j}) \setminus V(\MainH) = D$;
then note that $N_{\MainPlusH}[\PlusDongleVtx{j}]$ is isomorphic to 
$N_{H'}[x_D]$ and we are done with $D$.
\item Otherwise, $H[D]$, and all edges between $D$ and $N_H(D)$ are embedded in $\MainB$. 
Since $H$ is 3-connected, every connected component of $H[D]$ is connected and incident
with all three vertices of $N_H(D)$. Discard all but one such connected components and
contract the remaining one into $x_D$.
Furthermore, draw 
the three edges connecting elements of $N_H(D)$
along the edges between $N_H(D)$ and $x_D$.
\end{itemize}

Consider now the $\Sigma$-wall $H^\ast(\emptyset)$ in $H$. Every $D \in \Predongles(H,\mathcal{T})$
contains from $H$ either a subpath of one of the paths between the degree-3 vertices (without the endpoints) or a single degree-3 vertex with parts of the incident paths with internal vertices of degree $2$. Consequently, $H'$ also contains a $\Sigma$-wall of the same order as $H^\ast(\emptyset)$, that is, $h$. As $h > 2f_{\mathrm{ue}}(\eulerg_\ast)$, $H'$ has a unique embedding
into a surface of Euler genus at most $\eulerg_\ast$, and this is the discussed embedding
into $\Sigma$. 

Consider now $H'' \coloneqq H'-\{x_D~|~D \in \Predongles(H,\mathcal{T})\}$. Note that $H''$ is still 3-connected
(we discarded from $H'$ an independent set of simplicial vertices) and 
the discussed embedding of $H'$ into $\Sigma$ yields an embedding of $H''$ into $\Sigma$
with facewidth at least $h/2 > f_{\mathrm{ue}}(\eulerg_\ast)$. 
Consequently, Theorem~\ref{thm:fw-unique} asserts that this is also the unique embedding
of $H''$ of Euler genus at most $\eulerg_\ast$. 
Corollary~\ref{cor:small-auto} allows us to compute an isomorphism-invariant
family $\mathcal{F}_{H''}$ of size at most $4|E(H'')|$ of bijections $V(H'') \to [|V(H'')|]$. 

Note that $V(H'') \subseteq V(G)$. We proclaim $\Vgenus \coloneqq V(H'')$, $\Vtw \coloneqq V(G) \setminus \Vgenus$, and $\famgenus \coloneqq \mathcal{F}_{H''}$. 
The promised properties and bounds are immediate, except for the bound on the treewidth
of $G[\Vtw]$.
To this end, we prove the following.

\begin{claim}\label{cl:Vtw-bound}
For every connected component $C$ of $G-V(H'')-\Apices$,
$\SubProj_{\nembed}(C)$ is contained in $\rdisc{\nembed}{\rmax}{u}$ for some $u \in V(\MainA)$. 
\end{claim}
\begin{proof}
The claim is straightforward if $C$ is contained in a dongle or a vortex
of $\nembedA$, so assume otherwise.
Thus, $C$ contains a vertex of $\MainA$.

Assume first that $C$ does not contain a vertex of $\MainPlusA^3(\emptyset) \cap V(\MainA)$.
Since $C$ is connected, so is $\SubProj_{\nembedA}(C)$. 
Recall that
$\MainPlusA^2(\emptyset)$ does not contain a virtual vortex vertex nor any society vertex.
Since $C$ is not contained in a single dongle, $\SubProj_{\nembedA}(C)$ is not an isolated
virtual dongle vertex. 
Hence, $\SubProj_{\nembedA}(C)$ is contained
in $\rdisc{\nembed}{r_u+3}{u}$ for some $u \in R$. This finishes the proof of the claim in this case.

We are left with the case where
$C$ contains a vertex $v$ of $\MainPlusA^3(\emptyset) \cap V(\MainA)$. 
Recall that 
\[ V(H'') = V(H) \setminus \bigcup \{D~|~D \in \Predongles(H,\mathcal{T})\}. \]
Since $V(\MainPlusA^2(\emptyset)) \cap V(\MainA) \subseteq V(H)$, there exists
$D \in \Predongles(H,\mathcal{T})$ with $v \in V(D)$. 
Hence, for the connected component $C_D$ of $H[D]$ that contains $v$
we have $C_D \subseteq C$, while $N_H(C_D) = N_H(D) \subseteq V(H'')$. 
Let $C_D'$ consist of $C_D$ and all connected components of $G-V(H)-\Apices$
that have at least one neighbor in $C_D$. Clearly, $C_D' \subseteq C$.
Furthermore, as $N_H(D) \subseteq V(H'')$, we have $N_H(C_D) \subseteq N_G(C_D') \cap N_G(C)$.

Recall that $\SubProj_{\nembed}(H)$ contains $\MainPlusA^2(\emptyset)$.
Hence, by connectivity of $C_D$, 
from $C_D \cap V(\MainPlusA^3(\emptyset)) \cap V(\MainA) \neq \emptyset$
and $|N_H(C_D)| = 3$ we infer that 
$\SubProj_{\nembed}(N_H[C_D])$ is contained in $\MainPlusA^2(\emptyset)$.
Since every vertex of $\Apices_\OptNembeds \cup \VortexVtcs_\OptNembeds$
is either in $\Apices$ or its projection is in $\bigcup_{w \in R} \rdisc{\nembedA}{r_w}{w}$, 
$N_G(C_D) \setminus \Apices$ does not contain any vertex
of $\Apices_\OptNembeds \cup \VortexVtcs_\OptNembeds$. 
Since $H$ is a 3-connected component of $G-(\Apices_\OptNembeds \cup \VortexVtcs_\OptNembeds)$,
this implies that actually $C_D' = C$ and $N_H(C_D) = N_G(C) \setminus \Apices$.

Furthermore, from the previous paragraph
we have that for every two vertices of the three-vertex set $N_H(C_D)$ lie on a common
face of $\MainPlusA^2(\emptyset)$ that is not a face created upon clearing
one of the discs $\rdisc{\nembed}{r_w+2}{w}$ for $w \in R$. 
We infer that $\SubProj_{\nembed}(C)$ is contained in the radial disc $\rdisc{\nembed}{3}{u}$
for any $u \in N_H(C_D)$ and we are done.
\cqed\end{proof}

By Lemma~\ref{lem:radial-disc}, the treewidth of every connected component of $G-V(H'')-\Apices$
is bounded by $2q''+10\rmax \leq 2q'' + k''$. 
As $|\Apices| \leq k/100$, the treewidth of $G[\Vtw]$ is bounded by $2q''+k''+k/100$.
This finishes the proof of Theorem~\ref{thm:rigid}.

\bibliographystyle{alpha}
\bibliography{genusIsomorphism,tw-iso,references}

\appendix

\section{Proof of an explicit bound for statement (3.5) of~\cite{RobertsonS88}}\label{app:GM7fix}
In this section we provide a proof of an explicit bound for statement (3.5) of~\cite{RobertsonS88},
which is needed for the computability claim of Theorem~\ref{thm:GM7}. 
Since this section has a purely supplementary role, we assume that the reader is familiar with
the notation of Sections 1--3 of~\cite{RobertsonS88}, in particular their
notation $\Sigma(a,b,c)$ for surfaces, cuffs, matchings, forests, and homoplasty classes. 
The \emph{size} of a matching is the number of paths in the matching, and the \emph{size} of a forest
is the number of endpoints (on cuffs) in the forest. 

We will rely on the result of Geelen, Huynh, and Richter for a specific case:
\begin{theorem}[\cite{GeelenHR18}]\label{thm:GM7-GHR}
Let $\Sigma = \Sigma(a,b,c)$ be a surface and let $M$ and $M_2$ be
two matchings on $\Sigma$, of size $n$ and $n_2$, respectively.
Furthermore, assume that $\Sigma-M$ is connected.
Then there exists a matching $M_2'$ homeomorphic
to $M_2$ such that the number of intersections of $M$ and $M_2'$
is bounded by $n_2(3^n-1)$. 
\end{theorem}

We will prove the following version of the statement (3.5) of~\cite{RobertsonS88}
with explicit bounds:
\begin{theorem}\label{thm:GM7-3-5-fixed}
Let $\Sigma = \Sigma(a,b,c)$ be a surface, let $M$ be a matching in $\Sigma$ of size $n$,
   and let $F_2$ be a forest in $\Sigma$ of size $n_2$. 
Then, there exists a forest $F_2'$ from the same homoplasty class as $F_2$
such that 
the number of intersections of $M$ and $F_2'$
is bounded by $2n_2 \cdot 3^{2n + 3(2a+b)+c+n_2}$. 
\end{theorem}

The main technical part of the proof is a variant where $F_2$ is a matching, that is,
    a generalization of Theorem~\ref{thm:GM7-GHR} where $\Sigma-M$ is not required to be connected.
\begin{lemma}\label{lem:GM7-two-matchings}
Let $\Sigma = \Sigma(a,b,c)$ be a surface and let $M$ and $M_2$ be
two matchings on $\Sigma$, of size $n$ and $n_2$, respectively.
Then there exists a matching $M_2'$ homeomorphic to $M_2$
such that 
the number of intersections of $M$ and $M_2'$
is bounded by $n_2 \cdot 3^{2n + 3(2a+b)+c}$. 
\end{lemma}

Let us see why Lemma~\ref{lem:GM7-two-matchings} easily implies Theorem~\ref{thm:GM7-3-5-fixed}.
\begin{proof}[Proof of Theorem~\ref{thm:GM7-3-5-fixed}.]
Let $P$ be the set of points of $F_2$ of degree more than $2$ in $F_2$; note that $|P| < n_2$.
By shifting $F_2$ a bit, we can assume that no point of $P$
lies on a path of $M$. For every point $p \in P$, fix in $\Sigma$ a tiny
disc $\Delta_p$ around $p$ such that these discs are pairwise disjoint, disjoint with $M$, disjoint with cuffs,
and intersect $F_2$ only at tiny fragments of paths with endpoints in $p$.
By removing all discs $\Delta_p$ from $\Sigma$, we obtain a surface
$\Sigma' = \Sigma(a,b,c+|P|)$, where $M$ remains a matching (of size $n$)
while $F_2$ becomes a matching $M_2$ of size $n_2' < 2n_2$. 
Furthermore, by adding $F_2 \cap \bigcup_{p \in P} \Delta_p$ to any matching $M_2'$ homeomorphic to $M_2$ in $\Sigma'$, $M_2'$ is transformed into a forest $F_2'$ in the same homoplasty class as $F_2$
and with the same number of intersections with $M$ as $M_2'$.
Hence, applying Lemma~\ref{lem:GM7-two-matchings} to $\Sigma'$, $M$, and $M_2$
concludes the proof.
\end{proof}

Thus, we are left with proving Lemma~\ref{lem:GM7-two-matchings}. 
We need the following observations.
\begin{lemma}\label{lem:GM7-cut}
Let $\Sigma = \Sigma(a,b,c)$ be a surface with at least one cuff, $c \geq 1$.
Then, unless $a=b=0$ and $c=1$ (i.e., $\Sigma$ is a disc), there exists a
matching $M_0$ in $\Sigma$ with exactly one path 
such that $\Sigma-M$ is connected 
and $\Sigma-M = \Sigma(a',b',c')$ for some $a',b',c'$ with $3(2a' + b') + c' < 3(2a+b)+c$. 
\end{lemma}
\begin{proof}
If $c > 1$, let the single path of $M_0$ be any path connecting two points on different cuffs.
Then $\Sigma-M_0 = \Sigma(a,b,c-1)$ and we are done.

Otherwise, let $\Sigma' = \Sigma(a,b,0)$ be $\Sigma$ with a disc $\Delta$ glued up to the unique
cuff of $\Sigma$. Unless $a=b=0$ and $\Sigma'$ is a sphere, there exists
a noncontractible nonseparating closed noose $\gamma$ in $\Sigma'$.
Then, $\Sigma' \setminus \gamma = \Sigma(a',b',c')$ for some $c' \leq 2$ and $2a'+b' < 2a+b$.
We can modify $\gamma$ by a homotopy so that its intersection with $\Delta$
is a nonempty connected arc. 
Then, $\gamma\setminus\Delta$ is the desired path on $\Sigma$.
\end{proof}

\begin{lemma}\label{lem:features-additive}
 Let $\Sigma = \Sigma(a,b,c)$ be a surface and let $M$ be a matching in $\Sigma$. For a connected part $\Gamma$ of $\Sigma-M$, let $a_\Gamma,b_\Gamma,c_\Gamma$ be such that $\Gamma$ is homeomorphic to $\Sigma(a_\Gamma,b_\Gamma,c_\Gamma)$. Then
 $$\sum_{\Gamma} 3(2a_\Gamma+b_\Gamma)+c_\Gamma \leq 3(2a+b)+c+|M|,$$
 where the summation ranges over all connected parts $\Gamma$ of $\Sigma-M$.
\end{lemma}
\begin{proof}
 By straightforward induction, it suffices to consider the case $|M|=1$, say $M=\{e\}$. If $\Sigma-e$ is connected, then it has not larger number of cuffs, handles, and crosscaps as $\Sigma$, hence the claim follows. Otherwise, $e$ connects two points on the same cuff and $\Sigma-e$ consists of two connected parts $\Gamma_1$ and $\Gamma_2$. Then $a_{\Gamma_1}+a_{\Gamma_2}=a$, $b_{\Gamma_1}+b_{\Gamma_2}=b$, and $c_{\Gamma_1}+c_{\Gamma_2}=c+1$; this proves the claimed inequality.
\end{proof}

Consider now the following process. Start with $M_C = \emptyset$.
As long as $\Sigma - (M \cup M_C)$ contains a connected part $X$ that is not isomorphic
to a disc, apply Lemma~\ref{lem:GM7-cut} to $X$, obtaining a single-path matching $M_0$.
Without loss of generality, we can assume that $M_0$ has disjoint endpoints
with $M_C$ and $M$ (otherwise we shift them on the same cuff).  
We add $M_0$ to $M_C$ and repeat. 

At the end of the process we have the following properties:
\begin{itemize}
\item $M_C$ and $M$ are disjoint, that is, $M \cup M_C$ is a matching.
\item $\Sigma - (M \cup M_C)$ is a collection of discs.
\item With every path of $M \cup M_C$ we can associate two arcs
on the boundary of discs of $\Sigma-(M \cup M_C)$ where we cut along the said path. 
Furthermore, for each path of $M_C$, the two arcs are on the boundary of the same disc.
\item $|M_C| \leq 3(2a + b) + c +|M|$.
\end{itemize}
Here, the last inequality follows from Lemmas~\ref{lem:GM7-cut} and~\ref{lem:features-additive}.

We now consider the following process. Start with $M_B = \emptyset$.
As long as $\Sigma - ((M \setminus M_B) \cup M_C)$ contains more than two
connected parts, pick a path of $M \setminus M_B$
whose arcs on the boundaries lie on two different parts of $M \setminus M_B$
and add it to $M_B$ (which results in gluing the boundaries back in $\Sigma - ((M \setminus M_B) \cup M_C)$). The existence of such path follows from the connectivity of $\Sigma$
and the fact that every path of $M_C$ has its two corresponding arcs on the boundary
of the same disc of $\Sigma - (M \cup M_C)$. 
Furthermore, this process maintains the invariant that $M_B \subseteq M$ and $\Sigma - ((M \setminus M_B) \cup M_C)$ is a family of discs. 

Denote $M_A = M \setminus M_B$. 
Let $n_A$ and $n_B$ be the sizes of $M_A$ and $M_B$, respectively; then $n = n_A + n_B$.
Note that $\Delta \coloneqq \Sigma-(M_A \cup M_C)$ is a disc, hence in particular it is connected. 
So we can apply Theorem~\ref{thm:GM7-GHR} to $\Sigma$, matching $M_A \cup M_C$ (which has size at most
    $n+n_A + 3(2a+b)+c$), and $M_2$. 
Hence, there exists a matching $M_2'$ homeomorphic to $M_2$ with the number
of intersections with $M_A \cup M_C$ bounded by
$$ n_2 \cdot \left( 3^{n+n_A + 3(2a+b)+c} - 1 \right).$$
Hence, $\Delta \cap M_2'$ is a matching of size at most 
$$ n_2^\Delta := n_2 \cdot 3^{n+n_A + 3(2a+b)+c}$$
in the disc $\Delta$. 

Also, $M_B$ is a matching in $\Delta$. 
Observe that we can modify paths of $M_B$ --- obtaining a homotopic matching $M_B'$ in $\Delta$
--- so that for every path $P$ of $M'_B$ and every 
connected part of $\Delta-M_2'$, the intersection of $P$ with the said part is empty
or connected. Hence, each path of $M_B'$ has at most $n_2^\Delta$ intersections with $M_2'\cap \Delta$,
so the total number of intersetions of $M_B'$ and $M_2'\cap \Delta$ is bounded by $n_B \cdot n_2^\Delta$. 
By following the inverse of the homotopy between $M_B$ and $M_B'$, we obtain a matching
$M_2^\Delta$ homotopic to $M_2' \cap \Delta$ in~$\Delta$, that has at most
$n_B \cdot n_2^\Delta$ intersections with $M_B$.
Clearly, $M_2^\Delta$ lifts to a matching $M_2''$ homotopic to $M_2'$ in $\Sigma$
(and thus homeomorphic to $M_2$)
that has at most
$$ n_2 \cdot \left( 3^{n+n_A + 3(2a+b)+c} - 1 \right)$$
intersections with $M_A \cup M_C$, and at most 
$$n_B \cdot n_2^\Delta$$
intersections with $M_B$. 
Hence, the number of intersections of $M_2''$ and $M = M_A \cup M_C$ is bounded by
$$(n_B + 1) \cdot n_2 \cdot 3^{n+n_A + 3(2a+b)+c} \leq n_2 \cdot 3^{2n+3(2a+b)+c}.$$
This finishes the proof of Lemma~\ref{lem:GM7-two-matchings}
and of Theorem~\ref{thm:GM7-3-5-fixed}.

\section{Computability of optimal near-embeddings (proof sketch of Lemma~\ref{lem:compute-opt})}\label{app:compute-opt-nembed}
\newcommand{\msotwo}{MSO$_2$}

In this section we discuss how to adapt the techniques of~\cite{GroheKR13} to the setting
of Lemma~\ref{lem:compute-opt}.

The contribution of~\cite{GroheKR13} consists of two parts. 
  
\paragraph{Bounded treewidth case.} The first part
(described in Section~5 of~\cite{GroheKR13})
shows how to describe a near-embedding of a graph in monadic second-order logic (\msotwo{}).

In~\cite{GroheKR13}, the near-embedding is relative to tangle defined by a large unbreakable set;
in our case, it is sufficient just to take \emph{any} set of size more than $3q$, as we are
working with the $(q,k)$-unbreakability tangle. 

Since connectivity and cuts of constant size can be easily described in \msotwo{}, given $\Apices$
we can define $G(\Apices)$. This allows us to easily restrict the near-embeddings to those
aligned with 3-connected components by just considering near-embeddings of $G(\Apices)$ without
any apices. Also, requiring dongles to be properly connected is straightforward.

The definition of $\Predongles_0$ and $\Predongles$ is written essentially outright in \msotwo{}. 
To restrict considered near-embeddings to the ones with tidy dongles, we need also to be able
to verify being within radial distance at most $q+1$ from a vortex.
This can be done e.g. using the techniques of Lemma~5.3 of~\cite{GroheKR13}. 

Finally, once we defined near-embeddings with tidy dongles, 
 it is straightforward to take into account the notion of minimal vortices.

To sum up, with Courcelle's theorem, 
 we obtain a parameterized algorithm for the task of Lemma~\ref{lem:compute-opt},
but with the treewidth of $G$ as an additional parameter.

\paragraph{Irrelevant vertex rule.}
The second part of the contribution of~\cite{GroheKR13} is an irrelevant vertex rule:
if $G$ has very high treewidth (compared to the parameters of the near-embedding we seek),
then one can identify a vertex $v \in V(G)$ such that, given a near-embedding of $G-v$
it is easy to put $v$ back and obtain a near-embedding of $G$. Such a vertex can be iteratively
found and removed, until the treewidth drops to a value bounded by a function of the parameters
of the sought near-embedding. Then, we find a near-embedding of the current graph
using the techniques for graphs of bounded treewidth, and put back the removed vertices one-by-one.

The ``irrelevant'' vertex $v$ is found in the following way.
First, the Weak Structure Theorem of Robertson and Seymour~\cite{GM13} is invoked. 
This theorem asserts that, for every $h$ and $H$, there exists a $t$
such that any $H$-minor free graph with treewidth at least $t$ contains a
set $X \subseteq V(G)$ of size at most $\binom{|V(H)|}{2}$ and \emph{flat wall} in $G-X$
of order $h$. 

With the language introduced in this work, a \emph{flat wall} of $G' := G-X$ is 
a subgraph $G''$ of $G'$, a wall $W$
of width and height $h$ in $G''$, and a near-embedding $\nembedB$ of $G''$ into the plane
with no apices and vortices such that $V(\MainB) \subseteq V(W)$ and the vertices
of $V(G') \setminus V(G'')$ have neighbors only in the outer cycle of $W$ and the vertices
of dongles $\DongleB{j}$ with a vertex of $V(\MainB) \cap V(\DongleB{j})$ on this outer cycle. 

A flat wall $W$ has a natural notion of concentric cycles around its middle part; there
$h/2 - \Oh(1)$ of them. Let $v$ be a vertex inside all of them. 
The core part of the argument of~\cite{GroheKR13} is that if $h$ is high enough, compared
to the parameters of the considered near-embeddings of $G-v$, any near-embedding $\nembed$ 
of $G-v$ needs to embed a 
ring consisting of a few consecutive concentic cycles in the same way as in
the aforementioned near-embedding $\nembedB$ (with potentially minor differences in the choices
of dongles). This allows a replacement argument: one can cut $\nembed$ (and the surface) 
along the said ring, cap the holes with discs, and glue the interior of the said ring 
into one of this discs, as it is embedded in $\nembedB$. The resulting near-embedding
has the same (or lower, because the surface may get simpler) parameters than $\nembed$, 
and accomodates $v$. 

The replacement argument goes along the same lines as the replacement argument of
Lemma~\ref{lem:proximity}. Following the arugments of Lemma~\ref{lem:canonize-dongles},
we can also assume that $\nembedB$ has tidy dongles (except for possibly dongles
close to the boundary of $W$, similarly as Lemma~\ref{lem:canonize-dongles} chooses to
ignore dongles close to vortices). 
Thus, the replacement argument, applied to $\nembed$ with tidy dongles, also produces
a near-embedding with tidy dongles. 

We infer that the irrelevant vertex $v$, as it is identified in~\cite{GroheKR13}, also
suits our needs: from a near-embedding of $G-v$ with tidy dongles and minimal vortices,
we obtain a near-embedding of $G$ also with tidy dongles, minimal vortices, and
without increasing any of the parameters considered in the definition of
an optimal near-embedding.

This finishes the sketch of the proof of Lemma~\ref{lem:compute-opt}.

\end{document}